\documentclass[fleqn,a4paper]{cas-sc}
\usepackage[numbers]{natbib}

\usepackage{amsmath,amssymb,epsfig,float,color,mathtools,mathrsfs,xcolor,comment,algorithm,algorithmic}
\usepackage{amsfonts,amsmath,mathtools}
\usepackage{anyfontsize}

\numberwithin{equation}{section}
\numberwithin{figure}{section}
\numberwithin{table}{section}
\numberwithin{algorithm}{section}
\newcommand\disp{\displaystyle}

\newcommand{\dx}{\,\mathrm{d}\bx}

\newtheorem{theorem}{Theorem}[section]
\newtheorem{lemma}[theorem]{Lemma}
\newenvironment{proof}{\noindent{\it Proof.}}{\hfill$\square$}
\newcommand{\bphi}{{\boldsymbol\phi}}

\newcommand{\bPi}{{\mbox{\boldmath $\Pi$}}}
\newcommand{\btau}{{\boldsymbol\tau}}

\newcommand{\bchi}{{\boldsymbol\chi}}
\newcommand{\bdelta}{{\boldsymbol\delta}}
\newcommand{\bxi}{{\boldsymbol\xi}}
\newcommand{\bomega}{{\boldsymbol\omega}}
\newcommand{\btheta}{{\boldsymbol\theta}}

\newcommand{\bzeta}{{\boldsymbol\zeta}}

\newcommand{\bDelta}{{\boldsymbol\Delta}}
\newcommand{\bnabla}{{\boldsymbol\nabla}}
\newcommand{\bv}{{\boldsymbol{v}}}
\newcommand{\bb}{{\boldsymbol{b}}}
\newcommand{\bw}{{\boldsymbol{w}}}
\newcommand{\f}{\boldsymbol{f}}
\newcommand{\g}{\boldsymbol{g}}
\newcommand{\ba}{\boldsymbol{a}}
\newcommand{\bq}{\boldsymbol{q}}

\newcommand{\bc}{\mathbf{c}}
\newcommand{\bu}{{\boldsymbol{u}}}
\newcommand{\bx}{\boldsymbol{x}}
\newcommand{\bB}{\mathbf{B}}

\newcommand{\cA}{\mathcal{A}}
\newcommand{\cB}{\mathcal{B}}

\newcommand{\cJ}{\mathcal{J}}

\def\bphi{\boldsymbol{\phi}}

\def\rS{\mathrm{S}}
\def\H{\mathrm{H}}

\newcommand{\bz}{{\mathbf{z}}}

\newcommand{\bn}{{\mathbf{n}}}
\newcommand{\be}{{\mathbf{e}}}

\def\rp{\mathsf{p}}

\def\rF{\mathsf{F}}

\newcommand{\bzero}{\boldsymbol{0}}
\newcommand{\Om}{\Omega}
\newcommand{\0}{{\boldsymbol{0}}}

\def\bD{\mathbf{D}}
\def\bF{\mathbf{F}}

\def\bK{\mathbf{K}}

\def\bV{\mathbf{V}}
\def\bW{\mathbf{W}}

\def\bS{\mathbf{S}}

\newcommand{\bL}{\mathbf{L}}
\newcommand\bH{\mathbf{H}}

\newcommand{\bbR}{\mathbb{R}}

\newcommand{\cE}{\mathcal{E}}
\newcommand{\cT}{\mathcal{T}}

\def\wh{\widehat}
\def\wt{\widetilde}

\newcommand{\bcurl}{\mathbf{curl}\,}

\newcommand{\rH}{\mathrm{H}}
\newcommand{\rL}{\mathrm{L}}

\newcommand{\rQ}{\mathrm{Q}}
\def\rP{\mathrm{P}}
\newcommand{\rW}{\mathrm{W}\,}

\renewcommand\div{\mathop{\mathrm{div}\,}\nolimits}

\renewcommand{\div}{\mathrm{div}}

\def\qin{{\quad\hbox{in}\quad}}
\def\qon{{\quad\hbox{on}\quad}}
\def\qan{{\quad\hbox{and}\quad}}

\DeclareMathAlphabet\mathbfcal{OMS}{cmsy}{b}{n}

\renewcommand\div{{\mathrm{div}}}

\newcommand\crF{{\mathscr{F}}}
\newcommand\cI{{\mathcal{I}}}

\newcommand\RTinterp{{\cI^{\mathsf{RT}}}}
\newcommand\CRinterp{{\cI^{\mathsf{CR}}}}

\newcommand{\trinl}{\ensuremath{|\!|\!|}}
\newcommand{\trinr}{\ensuremath{|\!|\!|}}
\newcommand\pwnorm[1]{\trinl#1\trinr_{\mathrm{pw}}}

\definecolor{mygray}{rgb}{0.7,0.7,0.7}

\allowdisplaybreaks

\begin{document}
\let\WriteBookmarks\relax
\def\floatpagepagefraction{1}
\def\textpagefraction{.001}
\shorttitle{Velocity-vorticity-pressure formulation for Navier--Stokes--Brinkman--Forchheimer}
\shortauthors{Badia, Carstensen, Mart\'in, Ruiz-Baier, Villa-Fuentes}
\title[mode=title]{{A}
%Non-augmented 
velocity-vorticity-pressure formulation for the {steady} Navier--Stokes--Brinkman--Forchheimer problem}

\author[1]{Santiago Badia}[orcid=0000-0003-2391-4086]\ead{santiago.badia@monash.edu}
\author[2]{Carsten Carstensen}[orcid=0009-0008-6336-7332]\ead{cc@math.hu-berlin.de}
\author[3]{Alberto F. Mart\'in}[orcid=0000-0001-5751-4561]\ead{alberto.f.martin@anu.edu.au}
\author[1,4,5]{Ricardo Ruiz-Baier}[orcid=0000-0003-3144-5822]\ead{ricardo.ruizbaier@monash.edu}\cormark[1]
\author[1,6]{Segundo Villa-Fuentes}[orcid=0000-0002-0377-6555]\ead{svilla@ubiobio.cl}

\affiliation[1]{organization={School of Mathematics, Monash University},
    addressline={9 Rainforest Walk},
    postcode={3800}, 
    city={Melbourne},
    state={Victoria},
    country={Australia}}

\affiliation[2]{organization={Department of Mathematics, Humboldt-Universit\"at zu Berlin},
   % addressline={Casilla 7-D},
    postcode={10099}, 
    city={Berlin},
    country={Germany}}    

\affiliation[3]{organization={School of Computing, Australian National University},
    addressline={9 Rainforest Walk},
    postcode={2601}, 
    city={Acton},
    state={ACT},
    country={Australia}}
    
\affiliation[4]{organization={Universidad Adventista de Chile},
    addressline={Casilla 7-D},
    %postcode={3800}, 
    city={Chill\'an},
    country={Chile}}    
    
\affiliation[5]{organization={Institute for Computer Science and Mathematical Modelling, Sechenov First Moscow State Medical University},
    %addressline={9 Rainforest Walk},
    %postcode={2601}, 
    city={Moscow},
    %state={ACT},
    country={Russia}}    

\affiliation[6]{organization={Departamento de Matem\'atica, Universidad del B\'io-B\'io},
    addressline={Casilla 5-C},
    %postcode={3800}, 
    city={Concepci\'on},
    country={Chile}}    

\cortext[cor1]{Corresponding author.}

\begin{abstract}
The flow of incompressible fluid in highly permeable porous media  in vorticity - velocity - Bernoulli pressure form leads to 
 a double saddle-point problem in the  Navier--Stokes--Brinkman--Forchheimer equations. The paper establishes, {for small sources,} the existence of solutions on the continuous and discrete level of lowest-order piecewise divergence-free Crouzeix--Raviart finite elements. The vorticity employs a vector version of the pressure space with normal and tangential velocity jump penalisation terms. A simple Raviart--Thomas interpolant leads to pressure-robust  a priori error estimates. An explicit residual-based a posteriori error estimate allows for  efficient and reliable a posteriori error control. The efficiency for the Forchheimer nonlinearity requires a novel discrete inequality of independent interest. The implementation is based upon a light-weight forest-of-trees data structure handled by a highly parallel set of adaptive {mesh refining} algorithms. Numerical simulations reveal  robustness of the a posteriori error estimates and improved convergence rates by  adaptive mesh-refining. 
\end{abstract}


%%%%%%%%%%%%%%%%%%%%%%%%%%%%%%%%%%%%%%%%%%%%%%%%%%%%%%%%%%%%%%%%%%%%%%%%%%%%%%%%%

\maketitle

\section{Introduction}

\paragraph{Scope.} Equations of incompressible flow in dimensionless rotational or Lamb form (using vorticity) are of high importance in a number of applications. See, for example, the following non-exhaustive list of recent contributions analysing numerical methods based on different formulations \cite{amara07,benzi12,caraballo24,hanot23,layton09,salaun15,tsai05,zhang24} (and see also the references therein). {It} was observed in \cite{anaya21,anaya23a,anaya23b} that, in order to control the full $\mathbf{H}^1(\Omega)$ norm of the velocity and maintain optimal convergence, vorticity-based formulations for incompressible flow (with vorticity sought in $\mathbf{L}^2(\Omega)$ instead of the more common $\mathbf{H}(\bcurl,\Omega)$ case, velocity in $\mathbf{H}(\bcurl,\Omega)\cap\mathbf{H}(\div,\Omega)$, and Bernoulli pressure in $\rL^2(\Omega)$) required augmentation least-squares terms coming from the incompressibility condition and constitutive equation for vorticity (the latter resembling also the vorticity-stabilisation from, e.g., \cite{ahmed21,beirao21}).  The aforementioned works \cite{anaya21,anaya23a,anaya23b}  (which tackle Oseen, Navier--Stokes, and Forchheimer equations) also contain numerical evidence that either the vorticity or divergence stabilisation parameters could be zero for some finite element pairs approximating velocity and Bernoulli pressure. 

{The general Navier--Stokes and Forchheimer nonlinearities as well as the Brinkman drag will be included in this paper using fixed-point arguments. In contrast to the works above,} we do not use augmentation techniques and treat the problem as a (perturbation of a) perturbed saddle-point problem embedded in another saddle-point problem. The analysis hinges on working on the kernel of the divergence operator. At the discrete level we can perform a fairly similar analysis as long as we use kernel-characterising spaces, such that the divergence of the discrete velocity is zero locally in each {cell}. 
%space is contained in the discrete pressure space). As we work in quadrilateral meshes (hexahedral in 3D), 
For this we can take for instance the nonconforming Crouzeix--Raviart 
finite element pair {\cite{crouzeix1973conforming},    
which is Stokes inf-sup stable and  satisfies the required local kernel property}. 
{Pressure robust discretisations achieve velocity errors in the broken $\bH^1(\Omega)$ norm which are proportional to the best approximation error in the velocity, without dependence on the velocity error. 
In order to do so in the present setting,} we include a modification in the discrete right-hand side functional and in the  nonlinear variational forms using a lowest-order $\bH(\div,\Omega)$-conforming interpolate of the velocity test function.  This approach has been used for Stokes, Navier--Stokes, and many other variants in \cite{ahmed21,brennecke2015optimal,linke2014role,linke2015guaranteed,linke2017optimal,verfurth2019quasi}, see also the numerous references therein. This property is also closely related to the {$\bL^2(\Omega)$ orthogonality} of divergence-free functions onto the space of gradients of functions in  $\rH^1(\Omega)$. This approach comes at the price a small consistency error of optimal order, which is independent of the kinematic viscosity. 

It is important to mention that the nonconformity of the method in combination with the need to control the $\mathbf{H}(\bcurl,\Omega)\cap\mathbf{H}(\div,\Omega)$ part of the velocity norm, imply that we need to stabilise the velocity-velocity bilinear form with tangential and normal jump terms to control the divergence and curl part of the velocity norm {on the discrete level. This was} done in, e.g., \cite{brenner07,brenner08} for the grad-div, curl-curl, and reduced Maxwell problems, and we recall that such stabilisation is not required for, e.g., Stokes equations with velocity in $\mathbf{H}^1(\Omega)$.

{In many problems, solution singularities can cause suboptimal convergence, and adaptive mesh refinement is essential for recovering optimal rates.} A number of variants of residual-based a posteriori estimators are available for Crouzeix--Raviart schemes applied to Stokes equations \cite{carstensen14,carstensen12}. In addition, the literature of pressure-robust methods also has works designing a posteriori error estimators \cite{linke2015guaranteed,lederer19}, including the use of divergence-free reconstruction operators and techniques that are commonly encountered in stream function-vorticity formulations of incompressible flow. The explicit residual-based a posteriori error estimate follows the overall frame of  \cite{CGN}  (that includes historical remarks as well) 
with several additions for robustness (e.g. Lemma \ref{lem:aux-24}) required for the discrete norms. The efficiency result, however, was surprisingly subtle for some novel inverse estimate of independent innovation, we therefore establish in an appendix. 
The a posteriori error estimator contains a residual contribution and a non-conformity contribution. We show that it is reliable and efficient, where the efficiency proof relies upon a novel inverse estimate associated with element bubble functions. We use the estimator as an indicator for adaptive mesh  refinement, and this restores optimal convergence in the case of singular solutions (which under uniform mesh refinement yield suboptimal convergence). 

The particular adaptive meshes that we leverage in this work are hierarchically-adapted (i.e., nested) non-conforming octree-based meshes endowed with Morton (a.k.a., Z-shaped) space-filling-curves for storage and data partitioning; see, e.g., \cite{badia20}. This family of meshes can be very efficiently handled (i.e., refined, coarsened, re-partitioned, etc.) using high-performance and low-memory footprint algorithms \cite{Burstedde2011}. While these are $n$-cube meshes (e.g., made of quadrilaterals  or cubes in 2D and 3D, resp.), we split their elements (e.g., into 2 triangles or 6 tetrahedra, resp.) to obtain the simplicial meshes required by our finite element formulation. However, as these meshes are non-conforming (in particular they have hanging faces at cell interfaces between cells located at different levels of refinement), one needs to add additional multi-point constraints to the Crouzeix--Raviart finite element spaces in order to have optimal approximability properties. In particular, following \cite{bangerth17}, these additional constraints impose the trace average on a parent coarse face to be equivalent to the average of the trace averages on the children faces, and we show via numerical experiments that this approach recovers optimal convergence in the case of singular solutions.

\paragraph{Main contributions.} 
In summary, to the best of our knowledge, the combination of the contributions addressed above (residual a posteriori error estimators for Navier--Stokes--Brinkman--Forchheimer equations in vorticity form using non-conforming methods) is novel. In particular this work features  
\begin{itemize} 
\item a new non-augmented vorticity-based weak formulation for the Navier--Stokes equations with Brinkman and Forchheimer effects, generalising the recent work \cite{anaya23b},
\item a rigorous solvability analysis and discrete problem for vorticity that attains pressure-robustness, complementing the works  \cite{linke2014role,linke2017optimal},
\item a continuous and discrete analysis valid in 3D, extending the similar works for linear curl-curl type problems that address the 2D case  \cite{brenner07,brenner08,barker2024nonconforming},
\item a relatively simple residual-based a posteriori estimator for the pressure-robust scheme (compared to those from, e.g.,  \cite{carstensen14,lederer19}),
\item the efficiency of the residual a posteriori error estimator requires an interesting novel result regarding inverse estimates for the Forchheimer nonlinearity,
\item efficient and reliable a posteriori estimators, which we prove theoretically and also confirmed numerically (noting that the effectivity index --the ratio between the total error and the global a posteriori error estimator-- remains bounded between 1.6 and 1.9 for the tested cases), 
\item new handling of multipoint constraints needed for implementation of Crouzeix--Raviart elements with hanging nodes in adaptive meshes constructed with octrees. This generalises the results from \cite{badia20}.   
\end{itemize}

\paragraph{Outline.} This article is organised as follows: in the remainder of this section we provide notational conventions and main assumptions to use throughout the paper. In Section~\ref{sec:model} we give a brief overview on the governing equations and their statement in weak perturbed saddle-point form. Section~\ref{sec:continuous-analysis} deals with the well-posedness analysis using Banach fixed-point theorem under small data assumptions and a global inf-sup argument. The definition of the discrete problem and the analysis of its unique solvability are addressed in Section~\ref{sec:discrete-problem}. In Section~\ref{sec:apriori} we derive C\'ea estimates and error bounds for the specific finite element subspaces mentioned above. The definition of a posteriori error estimators and their robustness analysis is given in Section~\ref{sec:aposteriori}.   We {continue} in Section~\ref{sec:examples} describing the benchmark setups we used in the numerical experiments, showcasing the properties of the proposed schemes, and confirming numerically the predicted a priori convergence estimates and robustness of the a posteriori error estimators.  
{In Section~\ref{sec:concl} we give some concluding remarks, and Appendix \ref{sec:inverse} provides an inverse estimate for the efficiency proof of its own interest}. 

%%%%%%%%%%%%%%%%%%%%%%%%%%%%%%%%%%%%%%%%%%%%%%%

\paragraph{Preliminaries and notation.}
Let us denote by $\Om \subset \bbR^3$ a {connected} bounded polyhedral Lipschitz domain  with boundary $\Gamma$.  Standard notations will be adopted for Lebesgue spaces $\rL^p(\Om)$, with $p\in[1,\infty]$ and Sobolev spaces $\rW^{r,p}(\Om)$ with  $r\geq 0$, endowed with the norms $\|\bullet\|_{\rL^p(\Om)}$ and $\|\bullet\|_{\rW^{r,p}(\Om)}$. Note that $\rW^{0,p}(\Om)=\rL^p(\Om)$ and if $p=2$, we write $\rH^r(\Om)$ in place of 
$\rW^{r,2}(\Om)$, with the corresponding Lebesgue and Sobolev norms denoted by 
$\|\bullet\|_{0,\Om}$ and $\|\bullet\|_{r,\Om}$.  We also write $|\bullet|_{r,\Om}$ for the $\rH^r$-seminorm. The {space $\rL_0^2(\Omega)$ denotes the set of square-integrable functions on $\Omega$ with zero mean value $\rL_0^2(\Omega):=\{f\in \rL^2(\Omega): \int_\Omega f = 0\}$}. {The bracket $\langle \bullet,\bullet \rangle_\Gamma$ denotes duality that extends the $\rL^2(\Gamma)$  scalar product for smooth functions in the trace space  $\rH^{1/2}(\Gamma)$ of $\rH^1(\Omega)$ and its dual $\rH^{-1/2}(\Gamma)$}. By $\bS$ we will denote the corresponding vectorial counterpart of the generic scalar functional space $S$. {The gradient}, symmetric gradient, divergence and curl of a generic vector field $\bv=(v_i)$ are defined as
\[
\disp
 \bnabla \bv:=\left(\frac{\partial v_i}{\partial x_j}\right)_{i,j=1,3},\quad \bD\bv:=\frac12(\bnabla\bv +\bnabla\bv^T), \quad 
 \div \,\bv:=\sum_{j=1}^{3}\frac{\partial v_j}{\partial x_j},\quad \mbox{ and }\quad
 \bcurl\bv := \bnabla\times\bv.
\]
In addition, we recall that the spaces
\begin{align*}
\bH_0(\div;\Om)&:=\,\big\{ \bv \in \bL^2(\Om) :\quad \div\,\bv \in \rL^2(\Om) \qan {\gamma_\bn\bv=0} \qon \partial\Omega\big\},\\
\bH_0(\bcurl;\Om)&:=\,\big\{ \bv \in \bL^2(\Om) :\quad \bcurl\,\bv \in \bL^2(\Om)\qan {\gamma_\btau\bv=\bzero} \qon \partial\Omega\big\},
\end{align*}
where $\gamma_\bn$ and $\gamma_\btau$ represent the normal and tangential trace {operators}, respectively, and Hilbert when equipped with the norms $\|\bv\|_{\div,\Om}^{2}:=\|\bv\|_{0,\Om}^{2}
+\|\div\,\bv\|_{0,\Om}^{2}$ and  $\|\bv\|_{\bcurl,\Om}^{2}:=\|\bv\|_{0,\Om}^{2}
+\|\bcurl\,\bv\|_{0,\Om}^{2}$, respectively. Then, we define the following space
\[
\bV:=\bH_0(\div;\Om)\cap\bH_0(\bcurl;\Om),
\]
endowed with the norm
\[
\|\bv\|_{\bV}^{2}:=\|\bv\|_{0,\Om}^{2} +\|\div\,\bv\|_{0,\Om}^{2}
+\|\bcurl\,\bv\|_{0,\Om}^{2}.
\]
{Finally, the notation $A\lesssim B$ abbreviates $A\le CB$ with a generic  $h$  
 (mesh size)-independent  constant $C$, while some of the constants are still written {explicitly} to emphasise and quantify particular assumptions.}

%%%%%%%%%%%%%%%%%%%%%%%%%%%%%%%%%%%%%%%%%%%%%%%
%%%%%%%%%%%%%%%%%%%%%%%%%%%%%%%%%%%%%%%%%%%%%%%
\section{Model problem and its weak formulation}\label{sec:model}
%Let us consider a simply connected bounded and Lipschitz domain $\Omega\subset\bbR^3$ with boundary $\partial\Omega$, occupied by an incompressible fluid. 
\subsection{The governing equations}
We start with the steady Navier--Stokes--Brinkman--Forchheimer equations in their usual velocity--pressure form. They consist in finding velocity $\bu$ and kinematic pressure $P$ such that 
\begin{equation}\label{eq:NS}
\kappa^{-1}\bu-\nu \bDelta \bu + (\bu\cdot\bnabla)\bu + \rF|\bu|\,\bu  + \dfrac{1}{\rho}\nabla P  = \, \f \qin\Omega,\qquad
\div\,\bu =0  \qin\Omega
\end{equation}
with $\kappa$ the permeability of the porous media (assumed {heterogeneous but uniformly bounded away from zero}), $\nu$  the kinematic fluid viscosity, $\rF>0$ the Forchheimer coefficient, $\rho$ the fluid mass density, and $\f$ a given external force.
Problem \eqref{eq:NS} can be equivalently set in terms of vorticity, velocity and pressure (similarly as done in, e.g., \cite{anaya16} for Brinkman and in \cite{anaya19} for Oseen equations). For this, we introduce the rescaled vorticity vector
\[ \bomega := \sqrt{\nu} \,\bcurl \bu\]
and use the identity $\bu\cdot\bnabla\bu = \bcurl\bu \times \bu + \frac12 \nabla(\bu\cdot\bu) $. Then we introduce the rescaled Bernoulli pressure 
\[ p := \frac{1}{\rho} 
P + \frac{1}{2} \bu \cdot \bu - \lambda\]
for $\lambda$ defined as the mean value of $\frac{1}{2} \bu \cdot \bu$; and employ the following vector identity
\[ \bcurl \bcurl \bu = -\bDelta\bu + \nabla(\div\,\bu)\]
together with the incompressibility constraint. These steps lead to the following equations
\begin{equation}\label{eq:ns-new}
 \kappa^{-1}\bu + \sqrt{\nu} \bcurl\bomega + \rF|\bu|\,\bu + \nabla p + \frac{1}{\sqrt{\nu}}\,\bomega\times \bu  \,=\, \f, \qquad \bomega - \sqrt{\nu} \bcurl \bu  \,=\, \0, \qquad  \div\,\bu \,=\, 0.
\end{equation}
Furthermore, we focus on  homogeneous Dirichlet boundary conditions for velocity and therefore an additional condition is required to enforce the uniqueness of the Bernoulli pressure. This gives 
\begin{equation}\label{eq:boundary}
\bu=\0  \qon \Gamma \qan \int_\Omega p=0. 
\end{equation}
However, similar results as those shown below are also valid for other types of boundary conditions. 
%Then, we define the spaces
%\begin{equation*}
%\bH_0(\bcurl;\Om)\,:=\,\big\{ \bv \in \bH(\bcurl;\Om):\quad \bv \times\bn=0\qon\Gamma\big\},
%\end{equation*}
%\begin{equation*}
%\bH_0(\div;\Om)\,:=\,\big\{ \bv \in \bH(\div;\Om):\quad \bv \cdot\bn=0\qon\Gamma\big\},
%\end{equation*}
%\begin{equation*}
%\bH_0^1(\Om)\,:=\,\big\{ \bv \in \bH^1(\Om):\quad \bv=0\qon\Gamma\big\},
%\end{equation*}
%\begin{equation*}
%\bX:= \bH_\kappa^{-1}(\bcurl;\Om)\cap\bH_\Gamma(\div;\Om),\qquad \bY:=\bL^2(\Omega)\qan \rZ:=\rL^2(\Omega)
%\end{equation*}
%%%%%%%%%%%%%%%%%%%%%%%%%%%%%%%%%%%%%%%%%%%%%%%
\subsection{Mixed weak formulation}
%Here, we derive our mixed problem and define the forms and functionals involved.
First, {assume that $\f\in \bL^2(\Omega)$, $\kappa \in \rL^\infty(\Omega)$ and that there exists constants $\kappa_{\min},\kappa_{\max}$ such that $0<\kappa_{\min}\leq \kappa(\bx) \leq \kappa_{\max}$ a.e. in $\Omega$. Multiplying} the first, second and third equations of \eqref{eq:ns-new} by $\bv\in\bV$, $\btheta\in\bL^2(\Om)$ and $q\in \rL_0^2(\Om)$,
%$:=\{q\in\rL^2(\Om):\quad \int_\Om q=0\}$, 
respectively, integrating by parts and utilising the boundary condition, we obtain the problem: Find 
$((\bu,\bomega),p)\in[\bV\times\bL^2(\Om)]\times \rL_0^2(\Om)$ such that
\begin{align}\label{eq:weak}\nonumber
\int_\Omega {\kappa^{-1}}\bu\cdot \bv +\sqrt{\nu}\int_\Omega \bomega\cdot \bcurl \bv  - \int_\Omega p\,\div \bv - \dfrac{1}{\sqrt{\nu}}\int_\Omega (\bu\times\bomega)\cdot\bv + \rF\int_\Omega |\bu|\,\bu\cdot \bv &= \int_\Omega \f\cdot\bv,\\ 
\sqrt{\nu}\int_\Omega \btheta\cdot \bcurl \bu - \int_\Omega \bomega\cdot\btheta  &= 0,\\
 -\int_\Omega q\,\div\, \bu  &=  0.\nonumber 
\end{align}
%Now, for the sake of the subsequent analysis, we closely follow \cite{lamichhane24} (see also  \cite{gatica03} for similar works) 
%We adopt a two-fold saddle point formulation with a nonlinear %perturbation. Then, 
We introduce the bounded bilinear forms $a:[\bV\times\bL^2(\Om)]\times[\bV\times\bL^2(\Om)]\to\bbR$,  $b:[\bV\times\bL^2(\Om)]\times \rL_0^2(\Om)\to\bbR$, and, {for a given and fixed $\wh\bu\in\bV$ (to be considered in the fixed-point argument of Section~\ref{sec:continuous-analysis})}, the bilinear form $c^{\wh\bu}:[\bV\times\bL^2(\Om)]\times[\bV\times\bL^2(\Om)]\to\bbR$ as 
\begin{subequations}
\begin{align}
\label{eq:def-a}
a(\bu, \bomega;\bv,\btheta)& := \int_\Omega {\kappa^{-1}}\bu\cdot \bv + \sqrt{\nu}\int_\Omega \bomega\cdot \bcurl \bv + \sqrt{\nu}\int_\Omega \btheta\cdot \bcurl \bu  - \int_\Omega \bomega\cdot \btheta, \\
\label{eq:forms-b}
b(\bv,\btheta;q) & := - \int_\Omega q\,\div \bv,\\
\label{eq:forms-c}
c^{\wh\bu}(\bu,\bomega;\bv,\btheta) &:=  -\dfrac{1}{\sqrt{\nu}}\int_\Omega (\wh\bu\times\bomega)\cdot\bv  +\rF\int_\Omega |\wh\bu|\,\bu\cdot \bv.
\end{align}\end{subequations}
{Note that even if the bilinear form $b$ does not depend explicitly on the vorticity, we still write as in \eqref{eq:forms-b} to indicate a saddle-point structure in \eqref{eq:weak-formulation} below using the space $\bV\times\bL^2(\Om)$.}
On the other hand, we define the functional $\bF\in [\bV\times\bL^2(\Om)]'$ as
\begin{equation}\label{def:functional-F}
 F(\bv,\btheta) := \int_\Omega \f\cdot\bv.
\end{equation}
Then, the formulation consists in finding $((\bu, \bomega),p)\in [\bV\times\bL^2(\Om)]\times \rL_0^2(\Om)$, such that:
\begin{equation}\label{eq:weak-formulation}
\begin{array}{llll}
a(\bu,\bomega;\bv,\btheta)  &  +\quad b(\bv,\btheta;p) & +\quad c^{\bu}(\bu,\bomega;\bv,\btheta) & =  F(\bv,\btheta), \\ [1ex]
 \qquad b(\bu,\bomega;q) & \, & \,& = 0,
\end{array}
\end{equation}
for all $((\bv,\btheta),q)\in [\bV\times\bL^2(\Om)]\times \rL_0^2(\Om)$.

%%%%%%%%%%%%%%%%%%%%%%%%%%%%%%%%%%%%%%%%%%%%%%%%%%%%%
\section{Analysis of the coupled problem}\label{sec:continuous-analysis}
The following well-known symmetric and non-symmetric versions of the generalised Lax--Milgram lemma will be used in the forthcoming analysis (for a proof see, e.g.,  \cite[Theorems 1.3 \& 1.2]{gatica14}). 
\begin{lemma}\label{lemma:invertibility-of-T}
Let $\rH$ be a real Hilbert space, and let $A: \rH\times\rH \to \bbR$ be a symmetric and bounded bilinear form. Assume that  
\begin{equation}\label{eq:inf-sup-lemma}
\sup_{0\neq v\in \rH} \frac{A(u,v)}{\|v\|_\rH} \,\geq\, \wt\alpha\,\|u\|_\rH \quad \forall\,u\in \rH.
\end{equation}
Then, for each $F\in\rH'$ there exists a unique $u\in\rH$ such that 
\begin{equation*}
A(u,v)=F(v) \qquad \forall\,v\in\rH,
\qquad 
\text{and}
\qquad %begin{equation*}
\|u\|_\rH \leq \dfrac{1}{\wt\alpha}\|F\|_{\rH'}.
\end{equation*}
\end{lemma}

\begin{lemma}\label{lemma:invertibility-of-T2}
Let $\rH_1,\rH_2$ be real Hilbert spaces, and let $B: \rH_1\times\rH_2 \to \bbR$ be a  bounded bilinear form. Assume that  
\begin{subequations}\label{eq:inf-sup-lemma2}
\begin{align}
\sup_{0\neq v\in \rH_2} \frac{B(u,v)}{\|v\|_{\rH_2}} &\,\geq\, \widehat{\alpha}\,\|u\|_{\rH_1} \quad \forall\,u\in \rH_1,\\
\sup_{u\in \rH_1} B(u,v)&\,>0 \quad \forall\,v\in \rH_2,v\neq 0. 
\end{align}
\end{subequations}
Then, for each $F\in\rH_2'$ there exists a unique $u\in\rH_1$ such that 
\begin{equation*}
B(u,v)=F(v) \qquad \forall\,v\in\rH_2,
\qquad 
\text{and}
\qquad %begin{equation*}
\|u\|_{\rH_1} \leq \dfrac{1}{\widehat{\alpha}}\|F\|_{\rH_2'}.
\end{equation*}
\end{lemma}

We will  combine these results with 
%the Banach--Ne{\v c}as--Babu{\v s}ka theorem  \cite[Theorem 2.6]{ern04} and 
the Banach fixed-point theorem to demonstrate the well-posedness of \eqref{eq:weak-formulation} under a small data assumption.

\subsection{Stability properties of a linear problem}
First we recall the continuous embedding from  $\H^1(\Omega)$  into $\rL^\rp(\Omega)$, for all $\rp\in [1,6]$: 
\begin{equation}\label{eq:Sobolev-inequality}
\|w\|_{\rL^\rp(\Omega)}\leq   C_\rS\,\|w\|_{1,\Omega}\quad 
\forall\,w \in \H^1(\Omega)
\end{equation}
with $C_\rS> 0$ depending only on $|\Omega|$ and $\rp$ (see \cite[Theorem 1.3.4]{quarteroni94}).

Next we easily deduce {from} the Cauchy--Schwarz inequality, the continuity of  $a(\bullet,\bullet)$, $b(\bullet,\bullet)$: 
\begin{subequations}
\begin{align}\label{eq:bounded-a}
\big|a(\bu, \bomega;\bv, \btheta)\big|&\leq ({\kappa_{\min}^{-1}}+\sqrt{\nu}+1)(\|\bu\|_{\bV} + \|\bomega\|_{0,\Omega} )(\|\bv\|_{\bV} + \|\btheta\|_{0,\Omega}),\\
\label{eq:bounded-b}
\big|b(\bv,\btheta;q)\big|&\leq \|q\|_{0,\Omega}(\|\bv\|_{\bV} + \|\btheta\|_{0,\Omega}).
\end{align}
\end{subequations}
In turn, using H\"older's inequality  
%\begin{equation}\label{eq:Holder-inequality}
%\int_{\Omega} |fg| \leq \|f\|_{\rL^\rp(\Omega)} \|g\|_{\rL^\rq(\Omega)},\quad \forall\, f\in \rL^\rp(\Omega),\,\,\forall\, g\in \rL^\rq(\Omega), \quad\mbox{with}\quad \frac{1}{\rp} + \frac{1}{\rq} = 1,
%\end{equation}
together with \eqref{eq:Sobolev-inequality}, we readily deduce that
\begin{equation}\label{eq:bounded-c}
\big|c^{\wh\bu}(\bu,\bomega;\bv, \btheta)\big|\leq C_\rS^2 \Big(\dfrac{1}{\sqrt{\nu}}+\rF \Big)\|\wh\bu\|_{\bV}(\|\bu\|_{\bV} + \|\bomega\|_{0,\Om})(\|\bv\|_{\bV} + \|\btheta\|_{0,\Om}).
\end{equation}
%\begin{equation}\label{eq:bounded-C}
%\big|C_{\wh\bu}((\bomega,\bu),\bv)\big|\leq \dfrac{1}{\sqrt{\nu}}\|\wh\bu\|_{\bL^4(\Om)}\|\bomega\|_{0,\Om}\|\bv\|_{\bL^4(\Om)} + \rF\|\wh\bu\|_{\bL^4(\Om)}\|\bu\|_{0,\Om}\|\bv\|_{\bL^4(\Om)},
%\end{equation}
Similarly, the linear functional $F(\bullet)$  is bounded
\begin{equation}\label{des:bound-F}
\big|F(\bv,\btheta)\big|\leq   \|\f\|_{0,\Omega}(\|\bv\|_{\bV} + \|\btheta\|_{0,\Omega}).
\end{equation}

Now, it is straightforward to see that the kernel of the bilinear form $b(\bullet,\bullet)$ is a closed subspace of $\bV\times\bL^2(\Om)$. It is denoted as $\bV_0 \times\bL^2(\Om)$, and the first component admits the characterisation
\begin{equation}\label{eq:kernel-b}
\bV_0:=\{\bv\in\bV:\quad \div\,\bv=0\qin\Om\}\,.
\end{equation}
\begin{lemma}\label{lemma:invert-a}
The bilinear form $a(\bullet,\bullet)$ induces an invertible operator on $\bV_0\times\bL^2(\Omega)$. 
%the kernel of the bilinear form $b(\bullet,\bullet)$. 
\end{lemma}
\begin{proof}
We proceed using Lemma~\ref{lemma:invertibility-of-T}. 
First, from \eqref{eq:bounded-a} we observe that   $a(\bullet,\bullet)$ is bounded. To show that it also satisfies   the inf-sup condition \eqref{eq:inf-sup-lemma}, we proceed as in, e.g., \cite[Section~2.3]{anaya19}. For all $(\bz,\bzeta)\in\bV_0\times\bL^2(\Omega)$ (see \eqref{eq:kernel-b}), we can define $\wh\bv:= 2\bz$ and $\wh\btheta:= \sqrt{\nu}\bcurl\bz - \bzeta$, and then immediately assert that 
\begin{equation*}
a(\bz,\bzeta;\wh\bv,\wh\btheta) = 2\|{\kappa^{-1/2}}\bz\|^2_{0,\Omega} + \nu\|\bcurl\bz\|^2_{0,\Omega} + \|\bzeta\|^2_{0,\Omega}\geq \min\{2{\kappa_{\max}^{-1}},\nu\}\|\bz\|^2_{\bV} + \|\bzeta\|^2_{0,\Omega}.
\end{equation*}
Furthermore, it is clear that $\|\wh\bv\|_{\bV}=2\|\bz\|_{\bV}$ and $\|\wh\btheta\|_{0,\Omega}\leq (1+\sqrt{\nu})(\|\bz\|_{\bV} + \|\bzeta\|_{0,\Omega})$, and from this, we can conclude that
\begin{equation}\label{eq:inf-sup-a}
\sup_{\0\neq (\bv,\btheta)\in \bV_0\times\bL^2(\Om)}
\frac{a(\bz,\bzeta;\bv,\btheta)}{\|\bv\|_\bV + \|\btheta\|_{0,\Om}}\, \geq\, \frac{a(\bz,\bzeta;\wh\bv,\wh\btheta)}{\|\wh\bv\|_\bV + \|\wh\btheta\|_{0,\Om}} 
\geq \alpha\,( \|\bz\|_\bV + \|\bzeta\|_{0,\Om})
\quad \forall\,(\bz, \bzeta)\in \bV_0\times\bL^2(\Om)
\end{equation}
with $\alpha:=\dfrac{\min\{2{\kappa_{\max}^{-1}},\nu,1\}}{2(3+\sqrt{\nu})}$. Thus, the result follows.
\end{proof}

On the other hand, from the equivalence $\bH_0^1(\Om)=\bV$ (see \cite[Lemma~2.5]{girault79}), we have that $b(\bullet,\bullet)$ satisfies the following inf-sup condition (see \cite[Section~5.1]{girault79})
\begin{equation}\label{eq:inf-sup-b}
\sup_{\0\neq(\bv,\btheta)\in \bV\times\bL^2(\Om)} \frac{ b(\bv,\btheta;q)}{\|\bv\|_\bV + \|\btheta\|_{0,\Omega}} \geq \beta\,\|q\|_{0,\Om}
\quad \forall\,q\in \rL_0^2(\Om).
\end{equation}

Let us now define the bilinear form $\cA: \big([\bV\times\bL^2(\Om)]\times \rL_0^2(\Om)\big)\times\big([\bV\times\bL^2(\Om)]\times \rL_0^2(\Om)\big)\to\bbR$ as
\begin{equation}\label{eq:def-A}
\cA\big(\bz,\bzeta,r;\bv,\btheta,q):=a(\bz,\bzeta;\bv,\btheta)\, + \,b(\bz,\bzeta;q) \, +\, b(\bv,\btheta;r).\, 
\end{equation}
Owing to \eqref{eq:bounded-a} and \eqref{eq:bounded-b}, it is clear that $\cA(\bullet,\bullet)$ is bounded. Moreover, from \eqref{eq:inf-sup-a}, \eqref{eq:inf-sup-b} and \cite[Proposition 2.36]{ern04} it is not difficult to see that the following inf-sup condition holds:
\begin{equation}\label{eq:global-inf-sup-A}
\sup_{\0\neq((\bv,\btheta),q)\in [\bV\times\bL^2(\Om)]\times \rL_0^2(\Om) } \frac{ \cA(\bz,\bzeta,r;\bv,\btheta,q)}{\|((\bv,\btheta),q)\|} \geq \gamma\,\big\|((\bz,\bzeta),r)\big\| 
\end{equation}
for all $((\bz,\bzeta),r)\in [\bV\times\bL^2(\Om)]\times \rL_0^2(\Om)$, where $\|((\bz,\bzeta),r)\|:=\|\bz\|_{\bV} + \|\bzeta\|_{0,\Om} + \|r\|_{0,\Om}$, and
\begin{equation}\label{def:gamma}
\gamma= \dfrac{\min\{2{\kappa_{\max}^{-1}},\nu,1\}\beta^2}{(\beta+ {\kappa_{\min}^{-1}} + \sqrt{\nu}+ 2 )^2}.
\end{equation}

%%%%%%%%%%%%%%%%%%%%%%%%%%%%%%%%%%%%%%%%%%%%%%%
\subsection{Well-posedness analysis via Banach fixed-point}\label{sec:Well-posedness of the model}
%%%%%%%%%%%%%%%%%%%%%%%%%%%%%%%%%%%%%%%%%%%%%%%%%%%%%%%%%%%%%%%%%%%%%%%%%%%%%%%%%%%%%%%%%%%%%%%%%%%%%%%%%%%%
%\subsection{The fixed-point operator}\label{sec:fixed-point-opetaror}
We proceed similarly to \cite{caucao20} using a fixed-point strategy to  prove the well-posedness of \eqref{eq:weak-formulation}. Let us introduce the bounded set
\begin{equation}\label{eq:def-W}
\bK\,:=\,\Big\{ \wh\bu \in \bV :\quad  \|\wh\bu\|_{\bV}\leq \dfrac{2}{\gamma}\|\f\|_{0,\Om}   \Big\}
\end{equation}
with $\gamma$ the constant defined in \eqref{def:gamma}.
Then, we define a fixed-point operator as
\begin{equation}\label{eq:def-J}
\crF:\bK\to \bK, \quad \wh\bu\to\crF(\wh\bu)=\bu,
\end{equation}
where, given $\wh\bu\in\bK$, $\bu$ is the first component of $(\bu,\bomega)$, where $((\bu,\bomega),p)\in[\bV\times\bL^2(\Om)]\times \rL_0^2(\Om)$ is the solution of the linearised version of problem \eqref{eq:weak-formulation}: Find $((\bu,\bomega),p)\in[\bV\times\bL^2(\Om)]\times \rL_0^2(\Om)$ such that
\begin{equation}\label{eq:linear-weak-formulation}
\begin{array}{llll}
a(\bu,\bomega;\bv, \btheta)  &  +\quad b(\bv,\btheta;p) & +\quad c^{\wh\bu}(\bu,\bomega;\bv, \btheta) & =  F(\bv,\btheta), \\ [1ex]
 \qquad b(\bu,\bomega;q) & \, & \,& = 0 
\end{array}
\end{equation}
for all $((\bv,\btheta),q)\in [\bV\times\bL^2(\Om)]\times \rL_0^2(\Om)$. It is clear that $((\bu,\bomega),p)$ is a solution to \eqref{eq:weak-formulation} if and only if $\bu$ satisfies $\crF(\bu) = \bu$, and consequently, the well-posedness of \eqref{eq:weak-formulation} is equivalent to the unique solvability of the fixed-point problem: Find $\bu\in \bK$ such that
\begin{equation}\label{eq:fixed-point-problem}
\crF(\bu) = \bu.
\end{equation}

\noindent In this way, in what follows we focus on proving the unique solvability of \eqref{eq:fixed-point-problem}.
%%%%%%%%%%%%%%%%%%%%%%%%%%%%%%%%%%%%%%%%%%%%%%%%%%%%%%%%%%%%%%%%%%%%%%%%%%%%%%%%%%%%%%%%%%%%%%%%%%%%%%%%%%%%%%%%%%%%%%%%%%%
\subsection{Well-definiteness of the fixed-point map}
Let us first provide sufficient conditions under which the operator $\crF$ (cf. \eqref{eq:def-J}) is well-defined, or equivalently, the problem \eqref{eq:linear-weak-formulation} is well-posed.
\begin{lemma}[Unique solvability of the linearised problem]\label{lem:well-def-J}
Let	$\wh\bu \in \bK$ and assume that 
\begin{equation}\label{eq:assumption-J}
\dfrac{4}{\gamma^2} C_\rS^2 \Big(\dfrac{1}{\sqrt{\nu}}+\rF \Big)\|\f\|_{0,\Omega} \leq 1 
\end{equation}
with $\gamma$ the positive constant in \eqref{def:gamma}.
Then, there exists a unique $((\bu,\bomega),p)\in[\bV\times\bL^2(\Om)]\times \rL_0^2(\Om)$ solution to \eqref{eq:linear-weak-formulation}.
In addition, there holds
\begin{equation}\label{eq:dep}
\|((\bu,\bomega),p)\| \leq \frac{2}{\gamma}\|\f\|_{0,\Om}.
\end{equation}
\end{lemma}
\begin{proof} 
We proceed similarly as in the proof of \cite[Theorem 3.6]{CGO2021}.
In fact, given $\wh\bu \in \bK$, we begin by defining the bilinear form:
\begin{equation}\label{def:A-hat-u}
\cB^{\wh\bu}(\bz,\bzeta,r;\bv,\btheta,q) := 
\cA(\bz,\bzeta,r;\bv,\btheta,q)\,+\,c^{\wh\bu}(\bz,\bzeta;\bv,\theta)
\end{equation}
with $\cA(\bullet,\bullet)$ and $c^{\wh\bu}(\bullet,\bullet)$  the forms defined in \eqref{eq:def-A} and \eqref{eq:forms-c}.
Then, 
%taking $\cF((\bv,\btheta),q):= F(\bv,\btheta)$, 
problem \eqref{eq:linear-weak-formulation} can be rewritten equivalently as: 
Find $((\bu,\bomega),p)\in[\bV\times\bL^2(\Om)]\times \rL_0^2(\Om)$, such that
\begin{equation}\label{eq:variational-system-Problem1}
\cB^{\wh\bu}(\bu,\bomega,p;\bv,\btheta,q) = 
%\cF((\btheta,\bv),q)
F(\btheta,\bv)
\quad 
\forall\, ((\bv,\btheta),q)\in [\bV\times\bL^2(\Om)]\times \rL_0^2(\Om).
\end{equation}
Therefore, to prove the well-definiteness of $\crF$, in the sequel we equivalently prove that problem \eqref{eq:variational-system-Problem1}
is well-posed by means of Lemma~\ref{lemma:invertibility-of-T2}.
%the Banach--Ne\v cas--Babu\v ska theorem (see, for instance  \cite[Theorem 2.6]{ern04}).
First, given $((\bz,\bzeta),r), ((\wt\bv,\wt\btheta),\wt q)\in [\bV\times\bL^2(\Om)]\times \rL_0^2(\Om)$ with $((\wt\bv,\wt\btheta),\wt q)\neq \0$, from \eqref{eq:bounded-c} we observe that
\begin{align*}
\sup_{\0\neq((\bv,\btheta),q)  \in [\bV\times\bL^2(\Om)]\times \rL_0^2(\Om)}
&\frac{\cB^{\wh\bu}(\bz,\bzeta,r;\bv,\btheta,q)}{\|((\bv,\btheta),q)\|}
 \geq \frac{\big|\cA(\bz,\bzeta,r;\wt\bv,\wt\btheta,\wt q)\big|}{\|((\wt\bv,\wt\btheta),\wt q)\|}  -  \frac{\big|c^{\wh\bu}(\bz,\bzeta;\wt\bv,\wt\btheta)\big|}{\|((\wt\bv,\wt\btheta),\wt q)\|} \\ 
& \geq \frac{\big|\cA(\bz,\bzeta,r;\wt\bv,\wt\btheta,\wt q)\big|}{\|((\wt\bv,\wt\btheta),\wt q)\|}  - C_\rS^2 \Big(\dfrac{1}{\sqrt{\nu}}+\rF \Big)\|\wh\bu\|_{\bV}(\|\bz\|_{\bV} + \|\bzeta\|_{0,\Om}).
\end{align*}
Together with the global inf-sup condition \eqref{eq:global-inf-sup-A} and the fact that $((\wh\bv,\wh\btheta),\wt q)$ is arbitrary, this implies  
\begin{equation}\label{eq:auxiliar-eq-1}
\sup_{\0\neq((\bv,\btheta),q)  \in [\bV\times\bL^2(\Om)]\times \rL_0^2(\Om)}
\frac{\cB^{\wh\bu}(\bz,\bzeta,r;\bv,\btheta,q)}{\|((\bv,\btheta),q)\|}
\geq \left(\gamma - C_\rS^2 \Big(\dfrac{1}{\sqrt{\nu}}+\rF \Big)\|\wh\bu\|_{\bV}\right) \|((\bzeta,\bz),r)\|.
\end{equation}
Hence, from the definition of the set $\bK$ (cf. \eqref{eq:def-W}), and assumption \eqref{eq:assumption-J}, we easily get
\begin{equation}\label{eq:auxiliar-eq-2}
C_\rS^2 \Big(\dfrac{1}{\sqrt{\nu}}+\rF \Big)\|\wh\bu\|_{\bV} \leq \dfrac{2}{\gamma}C_\rS^2 \Big(\dfrac{1}{\sqrt{\nu}}+\rF \Big)\|\f\|_{0,\Omega}\leq \dfrac{\gamma}{2}
\end{equation}
and then, combining \eqref{eq:auxiliar-eq-1} and \eqref{eq:auxiliar-eq-2}, we obtain 
\begin{equation}\label{eq:BNB-1}
\sup_{\0\neq((\bv,\btheta),q)  \in [\bV\times\bL^2(\Om)]\times \rL_0^2(\Om)}
\frac{\cB^{\wh\bu}(\bz,\bzeta,r;\bv,\btheta,q)}{\|((\bv,\btheta),q)\|}
\geq \dfrac{\gamma}{2} \|((\bzeta,\bz),r)\|.
\end{equation}

On the other hand, for a given $((\bz,\bzeta),r)\in [\bV\times\bL^2(\Om)]\times \rL_0^2(\Om)$, we observe that
\begin{align*}
&\sup_{\0\neq((\bv,\btheta),q)  \in [\bV\times\bL^2(\Om)]\times \rL_0^2(\Om)}
\cB^{\wh\bu}(\bv,\btheta,q;\bz,\bzeta,r)\\  
& \qquad \geq  \sup_{\0\neq((\bv,\btheta),q)  \in [\bV\times\bL^2(\Om)]\times \rL_0^2(\Om)}
\frac{\cB^{\wh\bu}(\bv,\btheta,q;\bz,\bzeta,r)}{\|((\bv,\btheta),q)\|}\\ 
&\qquad =  \sup_{\0\neq((\bv,\btheta),q)  \in [\bV\times\bL^2(\Om)]\times \rL_0^2(\Om)}
\frac{\cA(\bv,\btheta,q;\bz,\bzeta,r) \,+\,c^{\wh\bu}(\bv,\btheta;\bz,\bzeta)}{\|((\bv,\btheta),q)\|}
\end{align*}
with the problem definition in the last step. Putting this 
 together with \eqref{eq:bounded-c} implies
\begin{align}\label{eq:auxiliar-eq-3}\nonumber 
&\sup_{\0\neq((\bv,\btheta),q)  \in [\bV\times\bL^2(\Om)]\times \rL_0^2(\Om)}
\cB^{\wh\bu}(\bv,\btheta,q;\bz,\bzeta,r)\\
&\qquad \geq \sup_{\0\neq((\bv,\btheta),q)  \in [\bV\times\bL^2(\Om)]\times \rL_0^2(\Om)}
\frac{\cA(\bv,\btheta,q;\bz,\bzeta,r)}{\|((\bv,\btheta),q)\|} - C_\rS^2 \Big(\dfrac{1}{\sqrt{\nu}}+\rF \Big)\|\wh\bu\|_{\bV}(\|\bz\|_{\bV} + \|\bzeta\|_{0,\Om}). 
\end{align}
Therefore, using the fact that $\cB^{\wh\bu}(\bullet,\bullet)$ is symmetric, from \eqref{eq:global-inf-sup-A} and \eqref{eq:auxiliar-eq-3} we obtain
\begin{equation*}
\sup_{\0\neq((\bv,\btheta),q)  \in [\bV\times\bL^2(\Om)]\times \rL_0^2(\Om)}
\cB^{\wh\bu}(\bv,\btheta,q;\bz,\bzeta,r)
\geq \left(\gamma - C_\rS^2 \Big(\dfrac{1}{\sqrt{\nu}}+\rF \Big)\|\wh\bu\|_{\bV}\right) \|((\bz,\bzeta),r)\|.
\end{equation*}
Using also \eqref{eq:auxiliar-eq-2}, yields
\begin{equation}\label{eq:BNB-2}
\sup_{\0\neq((\bv,\btheta),q)  \in [\bV\times\bL^2(\Om)]\times \rL_0^2(\Om)}
\cB^{\wh\bu}(\bv,\btheta,q;\bz,\bzeta,r)
\geq \frac{\gamma}{2}\,\|((\bz,\bzeta),r)\| > 0
\end{equation}
for all $((\bz,\bzeta),r)\in [\bV\times\bL^2(\Om)]\times \rL_0^2(\Om)$.

In this way, from \eqref{eq:BNB-1} and \eqref{eq:BNB-2} we obtain that $\cB^{\wh\bu}(\bullet,\bullet)$ satisfies the hypotheses of Lemma \ref{lemma:invertibility-of-T2}, 
%the Banach--Ne\v cas--Babu\v ska theorem \cite[Theorem 2.6]{ern04}, 
which allows us to conclude the unique solvability of 
%existence of a unique $((\bu,\bomega),p)\in[\bV\times\bL^2(\Om)]\times \rL_0^2(\Om)$ solution to 
\eqref{eq:linear-weak-formulation}, or equivalently, the existence of a unique $((\bu,\bomega),p)\in [\bV\times\bL^2(\Om)]\times \rL_0^2(\Om)$ such that $\crF(\bu)=\bu$.
Finally, from \eqref{eq:BNB-1}, with $((\bz,\bzeta),r)=((\bu,\bomega),p)$ and \eqref{eq:variational-system-Problem1}, we readily obtain that
\begin{equation}\label{eq:dep-b}
\|\bu\|_{\bV}\leq \|((\bu,\bomega),p)\|\leq \dfrac{2}{\gamma}\|\f\|_{0,\Om}
\end{equation}
implying that $\bu$ belongs to $\bK$ and concludes the proof.
\end{proof}
%%%%%%%%%%%%%%%%%%%%%%%%%%%%%%%%%%%%%%%%%%%%%%%%%%%%%%%%%%%%%%%%%%%%%%%%%%%%%%%%%%%%%%%%%%%%%%%%%%%%%%%%%%%%%%%%%%%%%%%%%%%
\subsection{Well-posedness of the continuous problem}
Now we provide the main result of this section, namely, the existence and uniqueness of solution of problem \eqref{eq:weak-formulation}. This result is established in the following theorem.

\begin{theorem}[Unique solvability]\label{th:unique-solution}
Let $\f \in \bL^2(\Omega)$ such that
\begin{equation}\label{eq:assumption-J-2}
\dfrac{4}{\gamma^2} C_\rS^2 \Big(\dfrac{1}{\sqrt{\nu}}+\rF \Big)\|\f\|_{0,\Omega} < 1
\end{equation}
with $\gamma$ the positive constant in \eqref{def:gamma}.
Then, $\crF$ (cf. \eqref{eq:def-J}) has a unique fixed-point $\bu\in\bK$. Equivalently, problem \eqref{eq:weak-formulation} has a unique solution $((\bu,\bomega),p)\in[\bK\times\bL^2(\Om)]\times \rL_0^2(\Om)$. %with $\bu\in\bK$.
Moreover, there holds
\begin{equation}\label{eq:stability}
\|((\bu,\bomega),p)\| \leq \dfrac{2}{\gamma}\|\f\|_{0,\Om}.
\end{equation}
\end{theorem}
\begin{proof}
Recall that 
%We begin by recalling from the previous analysis that assumption
\eqref{eq:assumption-J-2} ensures the well-definiteness of $\crF$. 
Now, let $\wh\bu_1$, $\wh\bu_2$, $\bu_1$, $\bu_2\in\bK$, be such that $\crF(\wh\bu_1)=\bu_1$ and $\crF(\wh\bu_2)=\bu_2$. 
According to \eqref{eq:def-J}, it follows that there exist unique $(\bomega_1, p_1)$, $(\bomega_2, p_2)$ $\in\bL^2(\Om)\times\rL_0^2(\Om)$, such that for all $((\bv,\btheta),q)\in[\bV\times\bL^2(\Om)]\times\rL_0^2(\Om)$, there hold
\begin{equation*}
\cB^{\wh\bu_1}(\bu_1,\bomega_1,p_1;\bv,\btheta,q) = 
F(\bv,\btheta)\, , \qan  \cB^{\wh\bu_2}(\bu_2,\bomega_2,p_2;\bv,\btheta,q) = 
F(\bv,\btheta)\,.
 \end{equation*}
Then, subtracting both equations, adding $\pm c^{\wh\bu_1}(\bu_2,\bomega_2;\bv,\theta)$, and recalling the definition of $\cB^{\wh\bu}$ in \eqref{def:A-hat-u}, we easily arrive at
\begin{equation*}
\cB^{\wh\bu_1}(\bu_1-\bu_2,\bomega_1-\bomega_2,p_1 -p_2;\bv,\btheta,q) = c^{\wh\bu_2}(\bu_2,\bomega_2;\bv,\theta) - c^{\wh\bu_1}(\bu_2,\bomega_2;\bv,\theta) \,.
\end{equation*}
Therefore, recalling that $\wh\bu_1\in\bK$ from the latter identity, together with \eqref{eq:BNB-1}, % the bound of $C_{\wh\bu}$ (see \eqref{eq:bounded-C})
the inequality $|\wh\bu_1|-|\wh\bu_2|\leq |\wh\bu_1-\wh\bu_2|$, and simple computations, we obtain
\begin{align*}
\frac{\gamma}{2}\,\|\bu_1-\bu_2\|_\bV
& \leq\, \sup_{\0\neq((\bv,\btheta),q)  \in [\bV\times\bL^2(\Om)]\times \rL_0^2(\Om) } 
\frac{\cB^{\wh\bu_1}(\bu_1-\bu_2,\bomega_1-\bomega_2,p_1 -p_2;\bv,\btheta,q)} {\|((\bv,\btheta),q)\|} \\ 
& =\, \sup_{\0\neq((\bv,\btheta),q)  \in [\bV\times\bL^2(\Om)]\times \rL_0^2(\Om) } 
\frac{ c^{\wh\bu_2}(\bu_2,\bomega_2;\bv,\theta) - 
c^{\wh\bu_1}(\bu_2,\bomega_2;\bv,\theta)}{\|((\bv,\btheta),q)\|} \\
& \leq\, C_\rS^2 \Big(\dfrac{1}{\sqrt{\nu}}+\rF \Big)\|\wh\bu_1-\wh\bu_2\|_{\bV}(\|\bu_2\|_{\bV} + \|\bomega_2\|_{0,\Om}). 
\end{align*}
Together with the fact that $(\bu_2,\bomega_2)$ satisfy \eqref{eq:dep}, this yields  
\begin{equation*}
\begin{array}{l}
\|\crF(\wh\bu_1) - \crF(\wh\bu_2)\|_{\bV} \,=\, \|\bu_1 -\bu_2\|_{\bV} \,\leq\, \dfrac{2}{\gamma} C_\rS^2 \Big(\dfrac{1}{\sqrt{\nu}}+\rF \Big) \dfrac{2}{\gamma}\|\f\|_{0,\Om} \|\wh\bu_1 - \wh\bu_2\|_{\bV} \,.
\end{array}
\end{equation*}
Combining the previous estimate with  \eqref{eq:assumption-J-2} and the Banach fixed-point theorem, readily implies that $\crF$ has a unique fixed-point in $\bK$, and so there exists a unique $((\bu,\bomega),p)\in[\bV\times\bL^2(\Om)]\times \rL_0^2(\Om)$ solution to \eqref{eq:weak-formulation}.
Finally, the estimate \eqref{eq:stability} is obtained analogously to \eqref{eq:dep-b}.
\end{proof}

Note that the formulation analysed above can also be defined in 2D. The vorticity is then the scalar $\omega = \sqrt{\nu} \mathrm{curl}\, \bu$,  the operator $\bcurl$ is to be replaced by $\mathrm{rot}$, and the weak convective term is now written as $-\frac{1}{\sqrt{\nu}}\int_\Omega \omega \bu \cdot \bv$. The space for vorticity is then $\rL^2(\Omega)$. At the discrete level these considerations also hold, but for sake of conciseness of the presentation we only discuss the 3D case. 
%%%%%%%%%%%%%%%%%%%%%%%%%%%%%%%%%%%%%%%%%%%%%%%%%%%%%%%%%%%%%%%%%%%%%%%%%%%%%%
\section{Galerkin scheme}\label{sec:discrete-problem}

In this section we introduce the Galerkin scheme associated with problem \eqref{eq:weak}, and show using Banach's fixed-point arguments that it admits a unique discrete solution. 

\subsection{Definition of the non-conforming method}
First, let us denote by $\{\cT_h\}$ a family of non-degenerate  simplicial meshes on $\Omega\subset\mathbb{R}^d$ (we simply assume that the domain is polytopal, so that no special treatment of the boundary is needed), and  denote by $\cE_h$ the set of all facets (edges in 2D) in the mesh, distinguishing between inner facets $\cE_h^{\mathrm{int}}$ and the set of facets lying on $\Gamma$, $\cE_h^\Gamma$. By $h_K$ we denote the diameter of the element $K$ and by $h_F$ we denote the length/area of the facet $F$ {bounded by the two cells $K^+$ and $K^-$}. As usual, by $h$ we denote the maximum of the diameters of elements in $\cT_h$. For a smooth vector, scalar, or tensor field $\zeta$ defined on $\cT_h$, $\zeta^{\pm}$ denote its traces taken from the interior of $K^+$ and $K^-$, respectively. We also
denote by $\bn^\pm$ the outward unit normal vector to $K^\pm$ (and for boundary faces it points outward of the domain $\Omega$). For any inner facet $F$ we define the {full},  normal, and tangential jumps of any element-wise defined vector function $\bv\in\bL^2(\Omega)$ across $F$ by 
\[ {[\![\bv]\!]_F:= \bv^+-\bv^-}, \quad [\![\bv\cdot\bn]\!]_F:= \bv^+\cdot\bn^++\bv^-\cdot\bn^-,\quad [\![\bv\times\bn]\!]_F:= \bv^+\times\bn^++\bv^-\times\bn^-\]
with $K^+$ and $K^-$ the two elements adjacent to $F$, and use the convention that $[\![\bv\cdot\bn]\!]_F:= \bv\cdot\bn$ and $[\![\bv\times\bn]\!]_F:= \bv\times\bn$ if $F\in \cE_h^\Gamma$. 
For all meshes we assume that they are sufficiently regular (there exists a uniform positive constant $\eta_1$ such that each element $K$ is star-shaped with respect to a ball of radius greater than $\eta_1 h_K$). It is also assumed that there exists $\eta_2>0$ such that for each element and every facet $F\in \partial K$, we have that $h_F\geq \eta_2 h_K$, see, e.g., \cite{ern04}). For $\ell \geq 0$, by $\mathbb{P}_\ell (K)$  we denote  the space of polynomials of total degree at most $\ell$ defined locally on the generic element $K \in \cT_h$. 

For the approximation of velocity and pressure we use the nonconforming Crouzeix--Raviart Stokes inf-sup stable element (see  
\cite{crouzeix1973conforming}) where  the velocity space consists of piecewise vector-valued $d$-linear polynomials on each dimension and continuous at the barycentre of the intra-element facets,   the discrete pressure consist of piecewise constant functions, and for sake of inf-sup stability we also need that the curl of the discrete velocity lives in the space of vorticity and so we take piecewise vector-valued constants. This gives {(now focusing on $d=2,3$ only)} 
\begin{align}\label{eq:fe-spaces}
\nonumber \bV_h&:=\{\bv_h \in \bL^2(\Omega): \bv_h \in \mathbb{P}_1(K)^d \ \forall K\in \cT_h, \quad J_F([\![\bv_h]\!]_F) = \boldsymbol{0} \ \forall F\in \cE^{\mathrm{int}}_h, \quad J_F(\bv_h|_F)=\boldsymbol{0} \ \forall F \in \cE_h^\Gamma\},\\
\bW_h&:=\{\btheta_h \in \bL^2(\Omega): \ \btheta_h|_K \in \mathbb{P}_0(K)^{d(d-1)/2} \quad \forall K \in \cT_h \},\\
%^{\mathrm{int}}\cup \cE_h^{\Gamma} \},\\
\rQ_h&:= \{ q_h \in \rL^2_0(\Omega): \ q_h|_K \in \mathbb{P}_0(K) \quad \forall K \in \cT_h\}.\nonumber
\end{align}
For any facet $F\in \cE_h$ with barycentre $C_F$, the nodal functional $J_F$ is defined by  
\[J_F(\bv) = \bv(C_F) \quad \text{or}\quad J_F(\bv) = \frac{1}{h_F}\int_F \bv\, \mathrm{d}s\]
and the degrees of freedom associated with facets on $\cE_h^\Gamma$ vanish for any $\bv_h\in \bV_h$. We recall the definition of the Crouzeix--Raviart interpolation $\CRinterp :\bV \to \bV_h$ as 
\[ \CRinterp \bv(C_F)= J_F(\bv) \qquad \forall F\in \cE_h.\]
 {We also recall that element-wise integration by parts on a given $K\in \cT_h$ readily gives that $\CRinterp$ preserves the averages of first derivatives. 

Next, defining the lowest-order Raviart--Thomas space 
\[\mathbf{RT}_0(\cT_h)  :=\{\bv_h\in \bH(\div\,,\Omega): \forall K \in \cT_h,\ \exists\,\bc_K\in \mathbb{R}^d,a_K\in \mathbb{R}: \ \bv_h|_K(\bx)= \bc_K + a_K\bx\}\]
we recall the Raviart--Thomas interpolation $\RTinterp :\bV \oplus \bV_h \to \mathbf{RT}_0(\cT_h)$ as 
\begin{equation}\label{eq:def-RT} \bn_F\cdot[\RTinterp \bv](C_F)= \frac{1}{h_F}\int_F \bv\cdot\bn_F \qquad \forall F\in \cE_h.\end{equation}
Note that even if $\mathbf{RT}_0(\cT_h)\not\subset\bV_h$ (since the tangential components of Raviart--Thomas functions are not necessarily continuous at each $C_F$), the interpolation is well-defined for $\bv\in \bV_h$ and we have that  (see, e.g., \cite{linke2017optimal}) 
\[\RTinterp \CRinterp \bv = \RTinterp \bv \qquad \forall \bv \in \bV.\]
For the subsequent analysis  
%{we will assume that $\kappa$ is piecewise constant, and shall write $\kappa_K$ whenever clear from the context}.  
we consider the following broken  norm for the space $\bV_h$ 
\begin{align}\label{eq:h-norm} \|\bv_h\|^2_h&:=
\sum_{K\in \cT_h}\bigl(\| \frac{1}{{\sqrt{\kappa}}}{\bv_h}\|^2_{0,K} +\nu \|\bcurl\bv_h\|_{0,K}^2 +\|\div\,\bv_h\|_{0,K}^2  \bigr) + \! \! \sum_{F\in \cE^{\mathrm{int}}_h} 
\frac{{1}}{h_F}\bigl(
\nu \|[\![\bv_h\times\bn]\!]_F\|^2_{0,F}+
\|[\![\bv_h\cdot\bn]\!]_F\|^2_{0,F}\bigr) 
\end{align}
as well as the piecewise $\bH^1(\cT_h)$-seminorm
\[\pwnorm{ \bv}^2:= \sum_{K \in \cT_h} |\bnabla \bv|^2_{1,K}. \]
We also have that the Raviart--Thomas interpolator is  stable on $\bV$ and also on $\bV_h$
\begin{equation}\label{eq:iterp-RT}
 {\| \RTinterp \bv \|_h  \leq C_{\mathrm{RT}}\|\bv\|_{h} \quad \forall\,\bv \in \bV \cup \bV_h}.\end{equation}

The following approximability bounds are known for Crouzeix--Raviart and Raviart--Thomas interpolants
\begin{subequations}\begin{align}\label{eq:aprox-CR}
\|\bv-\CRinterp \bv \|_h&\leq C_{\mathrm{CR}} h |\bv|_{2,\Omega} \qquad \forall\, \bv\in\bH_0^2(\Omega),
\\
\label{eq:aprox-RT-2}
  \|\bv- \RTinterp \bv\|_{0,\Omega}  &\leq C_F h%\sqrt{\sum_{K\in \cT_h}\,
  \|\bv\|_{h} \qquad \forall \bv\in\bV\cup\bV_h \end{align}
\end{subequations}
with \eqref{eq:aprox-CR} stated in \cite[Section 2.3]{brennecke2015optimal}, and the constant $C_{\mathrm{CR}}$ depends only on the mesh regularity. In addition, the constant $C_F$ only depends on the shape of the triangles/tetrahedra (maximum angle) but not on their size (see, e.g., \cite{carstensen12,john16}). 
Let $\mathcal{P}_h$ denote the $\rL^2$ projection operator, which satisfies the following approximation property (see \cite[Theorem 3.6]{gatica14}):
\begin{equation}\label{eq:L2-projection}
\|\mathcal{P}_h q - q\|_{0, \Omega}\leq C_{\rP}\, h |q|_{1,\Omega} \qquad \forall\, q\in\rH^1(\Omega).
\end{equation}

Since the method is nonconforming in the velocity space, for the discrete setting we will require the broken curl and broken divergence operators (associated with the non-diagonal part of the bilinear form $a_h(\bullet,\bullet)$ and the discrete bilinear form $b_h(\bullet,\bullet)$)  
\[ \bcurl_h : \bV \oplus \bV_h \to \bL^2(\Omega), \quad \div_h: \bV \oplus \bV_h \to \rL^2(\Omega)\]
in the following sense 
\[ (\bcurl_h\bv_h)|_K:= \bcurl(\bv_h|_K) \quad \text{and} \quad (\div_h\bv_h)|_K:= \div(\bv_h|_K) \qquad \forall K\in \cT_h. \]
{With these ingredients, we can define} element-wise variational forms. The forms that require modification are as follows
\begin{subequations}
\begin{align}\label{def:ah}
a_h(\bu_h,\bomega_h;\bv_h,\btheta_h) &: = 
\int_\Omega \frac{1}{\kappa}
 {\bu_h}\cdot  {\RTinterp \bv_h} + \sum_{F\in \cE^{\mathrm{int}}_h}%^{\mathrm{int}}}
\frac{\vartheta}{h_F}\int_F \bigl(
\nu [\![\bu_h\times\bn]\!]_F\cdot[\![\bv_h\times\bn]\!]_F +
[\![\bu_h\cdot\bn]\!]_F[\![\bv_h\cdot\bn]\!]_F \nonumber \\
&\quad + \sqrt{\nu}\int_\Omega \bomega_h\cdot\bcurl_h\bv_h + \sqrt{\nu}\int_\Omega \btheta_h\cdot\bcurl_h\bu_h -\int_\Omega \bomega_h\cdot\btheta_h,\nonumber\\ 
&= \sum_{K\in \cT_h}\int_K \frac{1}{{\kappa}} {\bu_h}\cdot {\RTinterp }\bv_h %\nonumber \\
%&\quad 
+\!\! \sum_{F\in \cE^{\mathrm{int}}_h}%^{\mathrm{int}}}
\frac{\vartheta}{h_F}\int_F \bigl(
\nu [\![\bu_h\times\bn]\!]_F\cdot[\![\bv_h\times\bn]\!]_F +
[\![\bu_h\cdot\bn]\!]_F[\![\bv_h\cdot\bn]\!]_F
 %\sqrt{\nu} \sum_{F\in \cE_h} \frac{\gamma}{h_F}\int_F 
\bigr)\nonumber\\
& \quad + \sqrt{\nu}\sum_{K\in \cT_h}\int_K \bomega_h \cdot\bcurl\bv_h
%b_h(\bv_h,\btheta_h):=
+\sqrt{\nu}\sum_{K\in \cT_h}\int_K \btheta_h \cdot\bcurl\bu_h -\int_\Omega\bomega_h\cdot\btheta_h,\\
b_h(\bv_h,\btheta_h;q_h)&:= - \int_\Omega q_h \div_h\bv_h = -\sum_{K\in \cT_h} \int_K q_h\div\,\bv_h ,\label{def:bh}\\
c_h^{\wh\bu_h}(\bu_h,\bomega_h;\bv_h,\btheta_h)& := 
-\frac{1}{\sqrt{\nu}} \int_\Omega (
%{\RTinterp 
 {\wh\bu_h}\times\bomega_h)\cdot \RTinterp \bv_h + \rF\int_\Omega |\wh\bu_h|\,
 {\bu_h}\cdot \RTinterp \bv_h, \label{def:ch} \\
F_h(\bv_h,\btheta_h)&:= \int_\Omega \f\cdot \RTinterp \bv_h\label{def:Fh}
\end{align}\end{subequations}
with $\vartheta>0$  a sufficiently large, user specified penalty parameter. The stabilisation in $a_h(\bullet,\bullet)$ uses normal and tangential jumps across inter-element boundaries, which are needed for controlling the consistency error and in general for the convergence of the scheme, as discussed in, e.g., \cite{hansbo2003discontinuous,knobloch2000korn} for elasticity equations (see also for example \cite{knobloch05,turek07}  for the case of nonconforming schemes on quadrilaterals).  We also provide numerical evidence in Section~\ref{sec:examples} that if $\vartheta=0$ then the method does not converge. Note also that, in \cite{brenner08} the jumps do not require a penalisation parameter since in that formulation the curl-curl and div-div terms are explicitly present in the continuous and discrete bilinear form (and the jump terms only contribute to maintain consistency). Also, note that the 2D Crouzeix--Raviart space used in \cite{brenner08} is also element-wise divergence-free, but the underlying continuous space only sets tangential components on boundary edges. Another variant in \cite{hansbo2010linear} imposes continuity only of the tangential components at the edges' midpoints. 

The interpolation of the test velocity in the right-hand side functional \eqref{def:Fh}  follows the definition proposed in \cite{linke2014role},  {but we stress that one could use any smoother operator such that the velocity error (in the broken norm \eqref{eq:h-norm}) is proportional to the corresponding best approximation error \cite{verfurth2019quasi}}. We proceed similarly for the convective and Forchheimer nonlinearities in  \eqref{def:ch}, {as well as for the Brinkman term.}

Having introduced the additional notations described above, the nonlinear discrete problem consists in finding $((\bu_h, \bomega_h),p_h)\in (\bV_h\times\bW_h)\times \rQ_h$, such that:
\begin{equation}\label{eq:weak-formulation-h}
\begin{array}{llll}
a_h(\bu_h,\bomega_h;\bv_h,\btheta_h)  &  +\quad b_h(\bv_h,\btheta_h;p_h) & +\quad c_h^{\bu_h}(\bu_h,\bomega_h;\bv_h,\btheta_h) & = F_h(\bv_h,\btheta_h), \\ [1ex]
 \qquad b_h(\bu_h,\bomega_h;q_h) & \, & \,& = 0
\end{array}
\end{equation}
for all $((\bv_h,\btheta_h),q_h)\in (\bV_h\times\bW_h)\times \rQ_h$. Note that the interpolated test discrete velocity on the right-hand side functional induce a variational crime approach that maps discretely divergence-free test functions to divergence-free functions in $\bH(\div,\Omega)$. This can also be regarded as a smoothing approach that permits to have a discrete load $F_h$ well defined for all continuous functionals on $\bV$. This setting has been used extensively in, e.g.,  \cite{ahmed21,verfurth2019quasi,linke2014role,linke2017optimal}, with the additional aim of achieving pressure robustness of the formulation.  We also recall that interpolated  test velocities are used in the convective nonlinearity.  Finally, using the Cauchy--Schwarz and H\"older inequalities, it is clear that the bilinear forms $a_h$ and $b_h$ are bounded
\begin{align*}
a_h(\bu_h,\bomega_h;\bv_h,\btheta_h)&\leq (\|\bu_h\|_{h} + \|\bomega_h\|_{0,\Omega})(\|\bv_h\|_{h} + \|\btheta_h\|_{0,\Omega}),\\
b_h(\bv_h,\btheta_h;q_h)& \leq \|q_h\|_{0,\Omega}(\|\bv_h\|_{h} + \|\btheta_h\|_{0,\Omega})
\end{align*}
%sbc{I would expect a constant in the first inequality, due to the fact that we don't have the penalty coefficient in the norm and the presence of the RT interpolant.}
as well as $c_h^{\wh\bu_h}$ and $F_h$:
\begin{subequations}\begin{align}\label{eq:bound-c-h}
|c_h^{\wh\bu_h}(\bu_h,\bomega_h;\bv_h,\btheta_h)|&\leq C_{c_h}\|\wh\bu_h\|_h(\|\bu_h\|_h+\|\bomega_h\|_{0,\Omega})(\|\bv_h\|_h+\|\btheta_h\|_{0,\Omega}),
\\
\label{eq:bound-F-h}
\big|F_h(\bv_h,\btheta_h)\big|&\leq  C_{F_h} \|\f\|_{0,\Omega}(\|\bv_h\|_h + \|\btheta_h\|_{0,\Omega})
\end{align}\end{subequations}
with $C_{c_h}>0$ and $C_{F_h}>0$ depending on the boundedness constant of the operator $\RTinterp $ denoted by $C_{\mathrm{RT}}$, as well as on the penalty parameter $\vartheta$. 
%%%%%%%%%%%%%%%%%%%%%%%%%%%%%%%%%%%%%%%
\subsection{Further properties of the discrete problem}
\paragraph{Discrete kernel properties.}
First, we denote the kernel of the bilinear form $b_h(\bullet,\bullet)$ as $\bV^0_h \times\bW_h$ (noting  that the discrete vorticity space does not play an active role), and  from \cite[Lemma~4.62]{john16} we can see that, since the broken divergence of an element-wise affine function is element-wise constant, we can readily choose as test function $q_h = \div_h\bv_h$, yielding the characterisation 
\begin{equation}\label{eq:kernel-B-h}
\bV^0_h:=\{\bv_h\in\bV_h:\quad \div_h\,\bv_h=0\}\,.
\end{equation}
Similarly, we stress that  
\begin{equation}\label{eq:curlvh}
\bcurl_h \bv_h \in \bW_h \qquad \forall \bv_h\in \bV_h.\end{equation}

The following lemma corresponds to the discrete version of Lemma~\ref{lemma:invert-a}.  {It depends on a mesh size smallness assumption, which can easily be avoided -- and the proof thus further simplified -- either if we have a discrete K\"orn-type inequality using the broken curl part of the discrete velocity norm (which is indeed valid trivially in the 2D case thanks to a discrete Poincar\'e inequality for Crouzeix--Raviart elements \cite{brenner08}), or if the first term in the definition of $a_h(\bullet,\bullet)$ is symmetric (for example, if it has also the Raviart--Thomas interpolation applied to the trial discrete function). We opt to keep the present form as it makes the a posteriori analysis more straightforward.} 
\begin{lemma}[Invertibility on the kernel]\label{lemma:invert-a-h}
The restriction of   $a_h(\bullet,\bullet)$ to the kernel of   $b_h(\bullet,\bullet)$ induces an invertible operator,  {provided that the mesh size $h$ is sufficiently small:}
\begin{equation}\label{assumption-h-coer}
 {    h \leq \frac{1}{2C_F \sqrt{\nu}}.}
\end{equation}
\end{lemma}
\begin{proof}
 {First, we note that for any $\bz_h \in \bV_h$, from the definition \eqref{eq:h-norm}, it readily holds 
\begin{equation}\label{eq:aux-m} \|\bz_h\|_{0,\Omega} \leq \sqrt{{\kappa_{\max}}} \|\bz_h\|_h.\end{equation}
Then, we can assert that 
\begin{align}\label{eq:aux-z}
\nonumber (\bz_h,\RTinterp\bz_h)_{0,\Omega} & = 
{(\bz_h, \RTinterp\bz_h -\bz_h)_{0,\Omega}}   + (\bz_h,\bz_h)_{0,\Omega}  \geq - \|\bz_h -\RTinterp\bz_h\|_{0,\Omega}\|\bz_h\|_{0,\Omega} + \|\bz_h\|^2_{0,\Omega} \\
& \geq -C_F h \sqrt{{\kappa_{\max}}} \|\bz_h\|^2_h + \|\bz_h\|^2_{0,\Omega}
%\nonumber (\bz_h,\RTinterp\bz_h)_{0,\Omega} & = (\bz_h -\RTinterp\bz_h,\RTinterp\bz_h)_{0,\Omega}   + (\RTinterp\bz_h,\RTinterp\bz_h)_{0,\Omega} \\
%\nonumber & \geq - \|\bz_h -\RTinterp\bz_h\|_{0,\Omega}\|\RTinterp\bz_h\|_{0,\Omega} + \|\RTinterp\bz_h\|^2_{0,\Omega} \\
%& \geq -C_FC_{\mathrm{RT}} h \|\bz_h\|^2_h + \|\RTinterp\bz_h\|^2_{0,\Omega},
\end{align}
having used the Cauchy--Schwarz inequality, as well as \eqref{eq:aprox-RT-2} 
and \eqref{eq:aux-m} 
in the last step.}

Next, and similarly to the continuous case, it is clear that $a_h(\bullet,\bullet)$ is bounded. In addition, for all $(\bz_h,\bzeta_h)\in\bV^0_h\times\bW_h$ (see \eqref{eq:kernel-B-h}), and owing to  \eqref{eq:curlvh}, we can define $\wh\bv_h:= 2\bz_h$ and $\wh\btheta_h:= \sqrt{\nu}\bcurl\bz_h - \bzeta_h$,  {and invoke \eqref{eq:aux-z}}, from which we  obtain  
\begin{align*}
&a_h(\bz_h,\bzeta_h;\wh\bv_h,\wh\btheta_h)\\
&\quad =  {(\frac{2}{\kappa} \bz_h,\RTinterp\bz_h)_{0,\Omega}} + \sum_{F\in \cE^{\mathrm{int}}_h}
\frac{2\vartheta}{h_F}\int_F \bigl(
\nu [\![\bu_h\times\bn]\!]^2_F +
[\![\bu_h\cdot\bn]\!]^2_F\bigr)
+ \nu \sum_{K\in \cT_h} \|\bcurl\bz_h\|^2_{0,K} + \|\bzeta_h\|^2_{0,\Omega}\\
&\quad  {\ \geq \frac{2}{{\kappa_{\max}}}  \|\bz_h\|^2_{0,\Omega} - \frac{2C_F h}{\sqrt{{\kappa_{\max}}}} \|\bz_h\|^2_h }+ \sum_{F\in \cE^{\mathrm{int}}_h}
\frac{2\vartheta}{h_F}\int_F \bigl(
\nu [\![\bu_h\times\bn]\!]^2_F +
[\![\bu_h\cdot\bn]\!]^2_F\bigr) 
+ \nu \sum_{K\in \cT_h} \|\bcurl\bz_h\|^2_{0,K} + \|\bzeta_h\|^2_{0,\Omega}\\
&\quad \geq  {2\biggl(1 - \frac{C_Fh}{\sqrt{{\kappa_{\max}}}}\biggr)}\|\bz_h\|^2_{h} + \|\bzeta_h\|^2_{0,\Omega}.
\end{align*}
Thus, using \eqref{assumption-h-coer} together with the fact that $\|\wh\bv_h\|_{h}=2\|\bz_h\|_{h}$ and $\|\wh\btheta_h\|_{0,\Omega}\leq (1+\sqrt{\nu})(\|\bz_h\|_{h} + \|\bzeta_h\|_{0,\Omega})$, we can conclude that
\begin{equation}\label{eq:inf-sup-a-h}
\sup_{\0\neq (\bv_h,\btheta_h)\in \bV^0_h\times\rQ}
\frac{a_h(\bz_h,\bzeta_h;\bv_h,\btheta_h)}{\|\bv_h\|_h + \|\btheta_h\|_{0,\Om}}\, \geq\, \frac{a_h(\bz_h,\bzeta_h;\wh\bv_h,\wh\btheta_h)}{\|\wh\bv_h\|_h + \|\wh\btheta_h\|_{0,\Om}} 
\geq \alpha_h\,( \|\bz_h\|_h + \|\bzeta_h\|_{0,\Om})
\end{equation}
for all $(\bz_h, \bzeta_h)\in \bV^0_h\times\rQ$, with $\alpha_h:=\frac{1}{2(3+\sqrt{\nu})}$. %$ - \dfrac{2C_FC_{\mathrm{RT}} h}{\kappa}$}. 
Thus, the result follows directly from %the application of 
Lemma~\ref{lemma:invertibility-of-T}.
\end{proof}

\paragraph{Discrete inf-sup conditions.}
The discrete inf-sup condition for the discrete divergence operator is satisfied for Crouzeix--Raviart elements \cite{crouzeix1973conforming}. This is recalled in the following result.
\begin{lemma}
    The pair $(\bV_h,\rQ_h)$ is inf-sup stable with constant $\beta_h>0$ independent of the mesh size
\begin{equation}\label{eq:inf-sup-b-h}
0< \beta_h:= \inf_{q_h\in \rQ_h\setminus\{0\}}\sup_{\bv_h\in\bV_h\setminus\{\boldsymbol{0}\}} \frac{b_h(\bv_h,\btheta_h;q_h)}{\|\bv_h\|_h\|q_h\|_{0,\Omega}}.       
\end{equation}
\end{lemma}
In addition, we have the following properties, shown in \cite[Lemma 3.1]{linke2017optimal}.
\begin{lemma}
 For all $\bv_h\in \bV_h$ there holds 
 \begin{subequations}
 \begin{align}
     b_h(\bv_h,\bullet;q) &= b\,(\RTinterp \bv_h,\bullet;q) \qquad \forall q\in  \rL^2(\Omega),\\
    b\,(\RTinterp \bv_h,\bullet;q)& = \int_\Omega \nabla q \cdot (\RTinterp \bv_h) \qquad \forall q\in \rH^1(\Omega).
 \end{align}\end{subequations}
\end{lemma}

\subsection{Discrete fixed-point arguments}
Here we proceed similarly as in the continuous setting. To begin with, we define the bilinear form $\cA_h: \big([\bV_h\times\bW_h]\times \rQ_h\big)\times\big([\bV_h\times\bW_h]\times \rQ_h)\to\bbR$ as
\begin{equation}\label{eq:def-A-h}
\cA_h(\bz_h,\bzeta_h,r_h;\bv_h,\btheta_h,q_h):=a_h(\bu_h,\bomega_h;\bv_h,\btheta_h)\, + \,b_h(\bv_h,\btheta_h;p_h) \, +\, b_h(\bu_h,\bomega_h;q_h).\, 
\end{equation}
It is easy to see that $\cA_h$ is bounded (since $a_h$ and $b_h$ are), and furthermore, using \eqref{eq:inf-sup-a-h}, \eqref{eq:inf-sup-b-h}, and \cite[Proposition 2.36]{ern04}, we have that $\cA_h$ satisfies the following inf-sup condition
\begin{equation}\label{eq:global-inf-sup-A-h}
\sup_{\0\neq((\bv_h,\btheta_h),q_h)\in [\bV_h\times\bW_h]\times \rQ_h } \frac{ \cA_h(\bz_h,\bzeta_h,r_h;\bv_h,\btheta_h,q_h)}{\|((\bv_h,\btheta_h),q_h)\|_h} \geq \wt\gamma\,\big\|((\bz_h,\bzeta_h),r_h)\big\|_h
\end{equation}
for all $((\bz_h,\bzeta_h),r_h)\in [\bV_h\times\bW_h]\times \rQ_h$, where $\|((\bz_h,\bzeta_h),r_h)\|_h:=\|\bz_h\|_h + \|\bzeta_h\|_{0,\Om} + \|r_h\|_{0,\Om}$, and $\wt\gamma>0$ is the discrete version of $\gamma$ (cf. \eqref{def:gamma}).

Let us introduce the following set
\begin{equation}\label{eq:def-W-h}
\bK_h\,:=\,\Big\{ \wh\bu \in \bV_h :\quad  \|\wh\bu_h\|_h\leq \frac{2}{\wt\gamma}C_{F_h}\|\f\|_{0,\Om}   \Big\}
\end{equation}
with $\wt\gamma$ the global inf-sup constant defined in \eqref{eq:global-inf-sup-A-h}. 
Then, and again analogously to the continuous case, we define the following fixed-point operator
\begin{equation}\label{eq:def-J-h}
\crF_h:\bK_h\to \bK_h, \quad \wh\bu_h\to\crF_h(\wh\bu_h)=\bu_h,
\end{equation}
where, given $\wh\bu_h\in\bK_h$, $\bu_h$ is the first component of %$(\bu_h,\bomega_h)$, where 
$((\bu_h,\bomega_h),p_h)\in[\bV_h\times\bW_h]\times \rQ_h$, the solution of the linearised version of problem \eqref{eq:weak-formulation-h}: Find $((\bu_h,\bomega_h),p_h)\in[\bV_h\times\bW_h]\times \rQ_h$ such that
\begin{equation}\label{eq:linear-weak-formulation-h}
\begin{array}{llll}
a_h(\bu_h,\bomega_h;\bv_h,\btheta_h)  &  +\quad b_h(\bv_h,\btheta_h;p_h) & +\quad c_h^{\wh\bu_h}(\bu_h,\bomega_h;\bv_h,\btheta_h) & = F_h( \RTinterp \bv_h,\btheta_h), \\ [1ex]
 \qquad b_h(\bu_h,\bomega_h;q_h) & \, & \,& = 0,
\end{array}
\end{equation}
for all $((\bv_h,\btheta_h),q_h)\in [\bV_h\times\bW_h]\times \rQ_h$.

It is clear that $((\bu_h,\bomega_h),p_h)$ is a solution to \eqref{eq:weak-formulation-h} if and only if $\bu_h$ satisfies $\crF_h(\bu_h) = \bu_h$, and consequently, the well-posedness of \eqref{eq:weak-formulation-h} is equivalent to the unique solvability of the fixed-point problem: Find $\bu_h\in \bK_h$ such that
\begin{equation}\label{eq:fixed-point-problem-h}
\crF_h(\bu_h) = \bu_h.
\end{equation}

In what follows we focus on \eqref{eq:fixed-point-problem-h}.
We start by establishing that $\crF_h$ is well-defined.

\begin{lemma}[Wellposedness of the discrete linearised problem]\label{lem:well-def-J-h}
Let	$\wh\bu_h \in \bK_h$ and assume that 
\begin{equation}\label{eq:assumption-J-h}
\dfrac{4}{\wt\gamma^2}C_{c_h} C_{F_h}\|\f\|_{0,\Omega} \leq 1
\end{equation}
with the positive constant  $\wt\gamma$ in \eqref{eq:global-inf-sup-A-h}.
Then, there exists a unique $((\bu_h,\bomega_h),p_h)\in[\bV_h\times\bW_h]\times \rQ_h$ solution to \eqref{eq:linear-weak-formulation-h}.
In addition, there holds
\begin{equation}\label{eq:dep-h}
\|((\bu_h,\bomega_h),p_h)\|_h \leq \dfrac{2}{\wh\gamma}C_{F_h}\|\f\|_{0,\Omega}.
\end{equation}
\end{lemma}
\begin{proof}
Given $\wh\bu_h\in\bK_h$, we proceed as in the proof of Lemma~\ref{lem:well-def-J} and define the bilinear form
\begin{equation}\label{def:A-hat-u-h}
\cB_h^{\wh\bu_h}(\bz_h,\bzeta_h,r_h;\bv_h,\btheta_h,q_h) := 
\cA_h(\bz_h,\bzeta_h,r_h;\bv_h,\btheta_h,q_h)\,+\,c_h^{\wh\bu_h}(\bz_h,\bzeta_h;\bv_h,\theta_h).
\end{equation}
Using \eqref{eq:bound-c-h}, \eqref{eq:global-inf-sup-A-h}, \eqref{eq:assumption-J-h} and \cite[Proposition 2.36]{ern04} we obtain the inf-sup condition
\begin{equation}\label{eq:BNB-1-h}
\sup_{\0\neq((\bv_h,\btheta_h),q_h)  \in [\bV_h\times\bW_h]\times \rQ_h}
\frac{\cB_h^{\wh\bu_h}(\bz_h,\bzeta_h,r_h;\bv_h,\btheta_h,q_h)}{\|((\bv_h,\btheta_h),q_h)\|_h}
\geq \dfrac{\wt\gamma}{2} \|((\bzeta_h,\bz_h),r_h)\|_h
\end{equation}
for all $((\bv_h,\btheta_h),q_h)\in[\bV_h\times\bW_h]\times \rQ_h$.
Therefore, owing to the fact that for finite dimensional linear problems  surjectivity and injectivity are equivalent, from \eqref{eq:BNB-1-h} and 
Lemma~\ref{lemma:invertibility-of-T2} 
%the Banach--Ne\v cas--Babu\v ska theorem 
we obtain that there exists a unique $((\bu_h,\bomega_h),p_h)\in[\bV_h\times\bW_h]\times \rQ_h$ satisfying \eqref{eq:linear-weak-formulation-h} with $\bu_h\in\bK_h$. %This concludes the proof.    
\end{proof}

The following theorem establishes the well-posedness of the nonlinear discrete problem \eqref{eq:weak-formulation-h}.
\begin{theorem}[Unique solvability of the discrete nonlinear problem]\label{th:unique-solution-h}
Let $\f \in \bL^2(\Omega)$ such that
\begin{equation}\label{eq:assumption-J-h-2}
\dfrac{4}{\wt\gamma^2}C_{c_h} C_{F_h}\|\f\|_{0,\Omega} < 1
\end{equation}
with $\wt\gamma$ the positive constant in \eqref{eq:global-inf-sup-A-h}. Then, $\crF_h$ (cf. \eqref{eq:def-J-h}) has a unique fixed-point $\bu_h\in\bK_h$. Equivalently, problem \eqref{eq:weak-formulation-h} has a unique solution $((\bu_h,\bomega_h),p_h)\in[\bK_h\times\bW_h]\times \rQ_h$. %with $\bu\in\bK$.
This discrete solution satisfies 
\begin{equation}\label{eq:stability-h}
\|((\bu_h,\bomega_h),p_h)\|_h \leq \dfrac{2}{\wt\gamma}C_{F_h}\|\f\|_{0,\Om}.
\end{equation}
\end{theorem}
\begin{proof}
Employing \eqref{eq:BNB-1-h} and \eqref{eq:dep-h}, along with  \eqref{eq:assumption-J-h-2}, the proof follows adapting the steps developed in the proof of Theorem~\ref{th:unique-solution}. Further details are omitted.
\end{proof}

%%%%%%%%%%%%%%%%%%%%%%%%%%%%%%%%%%%%%%%%%%%%%%%%%%%%%%%%%%%%%%%%%%%%%%%%%%%%%%
\section{A priori error bounds}\label{sec:apriori}
Now we turn to the error analysis. First we derive a Strang-type estimate. Then, under a small data assumption we show linear convergence of the method in the energy norm. Finally, we show that the velocity-vorticity error is independent of the pressure error. 
\begin{lemma}[C\'ea estimate]
Let $((\bu, \bomega),p)\in [\bV\times\bL^2(\Om)]\times \rL_0^2(\Om)$ and $((\bu_h,\bomega_h),p_h)\in[\bV_h\times\bW_h]\times \rQ_h$, the solution of 
\eqref{eq:weak-formulation} and \eqref{eq:weak-formulation-h}, respectively. Then there hold the error estimate
\begin{align}\label{eq:Strang}
%\begin{array}{cc}
\nonumber
 \|((\bu, \bomega),p) - ((\bu_h,\bomega_h),p_h)\|_h & \leq \Big(1 + \frac{2}{\wt\gamma} \Big)\inf_{ ((\bz_h,\bzeta_h),r_h)\in[\bV_h\times\bW_h]\times \rQ_h } \|((\bu, \bomega),p) - ((\bz_h,\bzeta_h),r_h)\|_h\\
&\quad  + \dfrac{2}{\wt\gamma}\sup_{\0\neq((\bv_h,\btheta_h),q_h)  \in [\bV_h\times\bW_h]\times \rQ_h}
\frac{\cB_h^{\bu_h}( \bu, \bomega, p; \bv_h,\btheta_h,q_h) - F_h(\bv_h,\btheta_h)}{\|((\bv_h,\btheta_h),q_h)\|_h}.
\end{align}
\end{lemma}
\begin{proof}
Let $((\bz_h,\bzeta_h),r_h)\in[\bV_h\times\bW_h]\times \rQ_h$ be arbitrary, we have the decomposition 
\begin{equation}\label{eq:error-dec}
\begin{array}{c}
\bu-\bu_h = \bu-\bz_h + (\bz_h -\bu_h) = \bu-\bz_h + \bchi_{\bu}, \quad 
\bomega-\bomega_h = \bomega-\bzeta_h + (\bzeta_h -\bomega_h) = \bomega-\bomega_h + \bchi_\bomega,\\
p-p_h = p-r_h + (r_h - p_h) = p-r_h + \chi_p .
\end{array}
\end{equation}
Then, from \eqref{eq:BNB-1-h}, using \eqref{eq:weak-formulation-h} and the Cauchy--Schwarz inequality, we have
\begin{align}\label{eq:auxiliar-eq}
%\begin{array}{l}
\nonumber 
\dfrac{\wt\gamma}{2} \|((\bchi_{\bu},\bchi_\bomega),\chi_p)\|_h &\leq 
\sup_{\0\neq((\bv_h,\btheta_h),q_h)  \in [\bV_h\times\bW_h]\times \rQ_h}
\frac{\cB_h^{\bu_h}(\bchi_{\bu},\bchi_\bomega,\chi_p;\bv_h,\btheta_h,q_h)}{\|((\bv_h,\btheta_h),q_h)\|_h}\\
\nonumber& \leq \|((\bu, \bomega),p) - ((\bz_h,\bzeta_h),r_h)\|_h \\
\nonumber & \quad + \sup_{\0\neq((\bv_h,\btheta_h),q_h)  \in [\bV_h\times\bW_h]\times \rQ_h}
\frac{\cB_h^{\bu_h}(\bu -\bu_h,\bomega - \bomega_h, p - p_h;\bv_h,\btheta_h,q_h)}{\|((\bv_h,\btheta_h),q_h)\|_h} \\
\nonumber & = \|((\bu, \bomega),p) - ((\bz_h,\bzeta_h),r_h)\|_h \\
& \quad + \sup_{\0\neq((\bv_h,\btheta_h),q_h)  \in [\bV_h\times\bW_h]\times \rQ_h}
\frac{ \cB_h^{\bu_h}( \bu, \bomega, p; \bv_h,\btheta_h,q_h) - F_h(\bv_h,\btheta_h)}{\|((\bv_h,\btheta_h),q_h)\|_h}.
\end{align}
Finally, \eqref{eq:Strang} is obtained directly from estimate \eqref{eq:auxiliar-eq} along with the error decomposition (cf. \eqref{eq:error-dec}) and the triangle inequality.
\end{proof}

The first term on the right-hand side of \eqref{eq:Strang} measures the approximation property of $[\bV_h\times\bW_h]\times \rQ_h$ with respect to the norm $\|((\bullet,\bullet),\bullet)\|_h$, while the second term captures the consistency error arising from the nonconforming discretisation.

Let $((\bu,\bomega),p)$ be the solution of the continuous problem. Similarly to \cite[Lemma 3.2]{brennecke2015optimal}, assuming that $((\bu,\bomega),p)\in (\bV\times\bH^1(\Omega) )\times\rH^1(\Omega)$, from \eqref{eq:weak-formulation}, integrating by parts, using the strong form of the momentum balance (first equation in \eqref{eq:ns-new}), and applying some algebraic manipulations, we get 
\begin{align}\label{eq:bound-B-F}
\nonumber &|\cB_h^{\bu_h}( \bu, \bomega, p; \bv_h,\btheta_h,q_h) - F_h(\bv_h,\btheta_h) |\\
\nonumber& = |\cB_h^{\bu_h}( \bu, \bomega, p; \bv_h,\btheta_h,q_h) - \int_\Omega (\kappa^{-1}\bu + \sqrt{\nu} \bcurl\bomega + \rF|\bu|\,\bu + \nabla p + \frac{1}{\sqrt{\nu}}\,\bomega\times \bu)\cdot\RTinterp \bv_h |\\
\nonumber & \leq\bigg|
%\dfrac{1}{\kappa}\sum_{K\in \cT_h}\int_K ({\color{orange} \bu_h  -\bu )\cdot \RTinterp \bv_h } }+
 \!\sqrt{\nu}\!\sum_{K\in \cT_h}\int_K \bcurl\bomega \cdot(\bv_h - \RTinterp \bv_h) +\!\! \sum_{K\in \cT_h}\int_K \nabla p \cdot(\bv_h - \RTinterp \bv_h) \bigg| \\
&  \qquad +\bigg|\frac{1}{\sqrt{\nu}} \int_\Omega ([\bu-\bu_h]\times\bomega)\cdot \RTinterp \bv_h + \rF\int_\Omega [|\bu_h|- |\bu|]\, \bu\cdot \RTinterp \bv_h  \bigg| .
\end{align}
%\paragraph{Approximation properties.}
%\begin{subequations}
%\begin{align}
% \| \RTinterp \bv \|_h & \leq \|\bv\|_h \qquad \forall \bv \in \bV,\\
% \| \RTinterp \bv -\bv\|_{0,\Omega} & \leq C\sum_{K\in \cT_h}h_K\|\nabla(\bv- \RTinterp \bv)\|_{0,K} \qquad \forall \bv\in\bV,\\ \label{eq:aprox-RT-2}
%  \| \RTinterp \bv -\bv\|_{0,\Omega} & \leq C_F\sum_{K\in \cT_h}h_K\|\nabla\bv\|_{0,K} \qquad \forall \bv\in\bV\cup\bV_h,
%\end{align}\end{subequations}
%where the constants $C,C_F$ only depend on the shape of the elements but not on their size (see, e.g., \cite{carstensen12,john16}). 
%From \cite[Section 2.3]{brennecke2015optimal}
%\begin{equation}\label{eq:aprox-CR}
%\|\bv- \CRinterp \bv\|_h\leq C_{\mathrm{CR}} h |\bv|_{2,\Omega}, \qquad \forall \bv\in\bH_0^2(\Omega).
%\end{equation}
Now, we are ready to determine the order of convergence of the proposed method.

\begin{theorem}[Rate of convergence]\label{th:rate1}
Let $((\bu,\bomega),p)$ and $((\bu_h,\bomega_h),p_h)$ solve the continuous and discrete problems \eqref{eq:weak-formulation} and \eqref{eq:weak-formulation-h}, respectively. Assume that the data satisfies 
\begin{equation}\label{eq:small-data}
C_\rS^2 C_{\mathrm{RT}}  \frac{8}{\gamma\wt\gamma} \Big(\frac{1}{\sqrt{\nu}} + \rF \Big)\|\f\|_{0,\Om}\leq 1
\end{equation}
and that $((\bu,\bomega),p)\in (\bH_0^2(\Omega)\times\bH^1(\Omega) )\times\rH^1(\Omega)$. 
Then there exists $C_{\mathrm{rate}}>0$, independent of $h$, such that 
\begin{equation}\label{eq:rate}
\begin{array}{ll}
\|((\bu, \bomega),p) - ((\bu_h,\bomega_h),p_h)\|_h  \leq\,  
C_{\mathrm{rate}}\,h\Big\{  |\bu|_{2,\Omega} \,+\, |\bomega|_{1,\Omega} \,+\, |p|_{1,\Omega}\Big\}.
\end{array}
\end{equation}
\end{theorem} 

\begin{proof}
To prove the result, we must estimate the two terms on the right-hand side of \eqref{eq:Strang}. For the first term, we can use the properties in \eqref{eq:L2-projection} %\cite[Theorem 3.6]{gatica14} 
(to estimate the errors for $\bomega$ and $p$) and \eqref{eq:aprox-CR} (to estimate the error for $\bu$), while for the second term, we apply estimate \eqref{eq:bound-B-F}, along with interpolation properties \eqref{eq:iterp-RT} and \eqref{eq:aprox-RT-2}, and the Cauchy--Schwarz inequality, to obtain
\begin{align*}
\|((\bu, \bomega),p) - ((\bu_h,\bomega_h),p_h)\|_h  & \leq\,  (C_{\mathrm{CR}} + 2 C_\rP)\Big(1 + \frac{2}{\wt\gamma} \Big)\,h\Big\{  |\bu|_{2,\Omega} \,+\, |\bomega|_{1,\Omega} \,+\, |p|_{1,\Omega}\Big\}\\
&\quad + C_F (\sqrt{\nu} +1) \frac{2}{\wt\gamma} h \Big\{   \|\bcurl\bomega\|_{0,\Omega} \,+\, \|\nabla p\|_{0,\Omega}\Big\} \\
& \quad + C_\rS^2 C_{\mathrm{RT}}\frac{2}{\wt\gamma} (\frac{1}{\sqrt{\nu}} + \rF )(\|\bu\|_{0,\Omega} + \|\bomega\|_{0,\Omega})\|\bu-\bu_h\|_h.
\end{align*}
From the latest estimate, using the fact that $(\bu,\bomega)$ satisfies estimate \eqref{eq:dep-b} and applying hypothesis \eqref{eq:small-data}, we obtain \eqref{eq:rate}.
\end{proof}

The estimate above can be refined to reflect the pressure-robustness of the formulation. Consider the continuous and discrete problems in their reduced form (in the continuous and discrete kernels $\bV_0$ and $\bV_h^0$, respectively)
\begin{subequations}
\begin{equation}\label{eq:reduced-cont}
a(\bu,\bomega;\bv,\btheta) + c^{\bu}(\bu,\bomega; \bv, \btheta)  = F(\bv,\btheta) \quad \forall (\bv,\btheta) \in \bV_0 \times \bW,    
\end{equation}
and 
\begin{equation}\label{eq:reduced-discr}
a_h(\bu_h,\bomega_h;\bv_h,\btheta_h) + c_h^{\bu_h}(\bu_h,\bomega_h; \bv_h, \btheta_h)  = F_h(\bv_h,\btheta_h) \quad \forall (\bv_h,\btheta_h) \in \bV_h^0 \times \bW_h,    
\end{equation}\end{subequations}
which are equivalent to \eqref{eq:weak-formulation} and \eqref{eq:weak-formulation-h}, respectively. 

In order to show a pressure-robust refinement of the previous results, we require an auxiliary bound regarding the piecewise norm control in the discrete norm $\|\bullet\|_h$ of vector fields from $\bV + \bV_h$.  In turn, for this as well as for the a posteriori error estimation later on, we will employ the companion operator 
 \begin{equation}\label{eq:companion}
 \cJ\in \mathcal{L}(\bV_h; \bV) 
 \end{equation}
 that is a 
right-inverse of the Crouzeix--Raviart 
interpolation $\CRinterp \in \mathcal{L}(\bV; \bV_h)$, satisfying 
\[ \| h_{\cT_h}(\bw_h - \cJ\bw_h)\|_{0,\Omega} \lesssim \pwnorm{\bw_h-\cJ\bw_h},\]
as well as other additional $L^2$ 
orthogonality properties not needed herein (see the precise design for 
2D and 3D in  \cite{CGS15,carstensen2024adaptive}).  

\begin{lemma}\label{lem:aux-24} 
For any $\bv\in \bV$ and $\bw_h\in \bV_h$ and $1\leq s \leq 6$ in 3D (and $1\leq s<\infty$ in 2D), there holds
\begin{equation}\label{eq:aux-24}
    \|\bv + \bw_h \|_{\bL^s(\Omega)} + \pwnorm{\bv + \bw_h} \leq C_{\sharp} \| \bv+\bw_h\|_h
\end{equation}
with $C_\sharp$ depending on $\Omega$ and the shape-regularity of the mesh $\cT_h$ (and as well on $s$ in the 2D case). 
\end{lemma}

\begin{proof}
We first add $\pm \cJ\bw_h$ to the left-hand side of \eqref{eq:aux-24}, use triangle inequality, and  invoke the  well-known discrete Sobolev embedding with constant $C_{\mathrm{dS}}(s)>0$:
    \[\|\bw\|_{\bL^s(\Omega)} \leq C_{\mathrm{dS}}(s)\pwnorm{\bw} \qquad \forall \bw \in \bV\]
   applied on the term $\bv - \cJ\bw_h \in \bV$. 
   Then, we employ the continuous Sobolev embedding \eqref{eq:Sobolev-inequality}, and this gives 
\begin{align*}
\mathrm{LHS}_{\eqref{eq:aux-24}} & :=    \|\bv + \bw_h \|_{\bL^s(\Omega)} + \pwnorm{\bv + \bw_h} 
     \leq C_{\mathrm{dS}}(s) \pwnorm{\bv + \cJ\bw_h} + \|\bw_h - \cJ\bw_h\|_{\bL^s(\Omega)} + \pwnorm{\bv + \bw_h}\\
     &\leq (1+C_{\mathrm{dS}}(s))\pwnorm{\bv+\cJ\bw_h} +  \|\bw_h - \cJ\bw_h\|_{\bL^s(\Omega)} + \pwnorm{\bw_h - \cJ\bw_h}\\
     & \leq (1+C_{\mathrm{dS}}(s)) \pwnorm{\bv+\cJ\bw_h} + (1+ C_{\mathrm{S}})\pwnorm{\bw_h - \cJ\bw_h}\\
     & \leq (1+C_{\mathrm{dS}}(s))[\| \bcurl(\bv + \cJ\bw_h)\|_{0,\Omega} + \| \div(\bv + \cJ\bw_h)\|_{0,\Omega} ] + (1+ C_{\mathrm{S}})\pwnorm{\bw_h - \cJ\bw_h} 
\end{align*}    
with  the equivalence --- valid for $\bv+\cJ\bw_h \in \bV$ --- between the piecewise norm $\pwnorm{\bullet}$ and the semi-norm in $\bH_0(\bcurl,\Omega)\cap \bH_0(\div,\Omega)$ \cite[Lemma 2.5 \& Remark 2.7]{girault79} in the last step. 

Using next again triangle inequality and the definition of the broken curl, the broken divergence, and the discrete velocity norm $\|\bullet\|_h$, from the bounds above we readily get 
\begin{align*}
\mathrm{LHS}_{\eqref{eq:aux-24}} & \leq [(1+C_{\mathrm{dS}}(s)) C_{\mathrm{norm}} + (1+ C_{\mathrm{S}})] \pwnorm{\bw_h - \cJ\bw_h} \\
&\qquad + (1+C_{\mathrm{dS}}(s))[\| \bcurl_h(\bv + \bw_h)\|_{0,\Omega} + \| \div_h(\bv + \bw_h)\|_{0,\Omega} ]\\
%& \leq (1+C_{\mathrm{S}}) C_{\mathrm{norm}} \pwnorm{\bw_h - \cJ\bw_h} + (1+ C_{\mathrm{dS}}(s))C_{\mathrm{norm}}'\|\bv+\bw_h\|_h \\
& \leq C_\sharp \| \bv+\bw_h\|_h
\end{align*}
with the following estimate from \cite{carstensen2024adaptive} in the last step:
\[ \pwnorm{\bw_h - \cJ\bw_h} \lesssim 
\pwnorm{\bv + \bw_h},\]
%\sum_{e\in \cE_h} h_e^{-1} \| \jump{\bw_h}\|^2_{0,e},\]
as well as the fact that 
\[ \| \bcurl_h(\bv + \bw_h)\|_{0,\Omega} + \| \div_h(\bv + \bw_h)\|_{0,\Omega} \lesssim \pwnorm{\bv + \bw_h}.\]
Therefore $C_\sharp>0$ depends on $C_{\mathrm{S}},C_{\mathrm{dS}}(s)$ and on the parameter-dependent constant $C_{\mathrm{norm}}>0$. \end{proof}

Lemma~\ref{lem:aux-24} implies, in particular, that 
\begin{equation}\label{new-bound} {\|\bu - \bu_h\|_{\bL^4(\Omega)} \leq C_{\sharp} \| \bu-\bu_h\|_h.}\end{equation}
%\cred{[Note: Instead of Lemma \ref{lem:aux} below, in the pressure robustness associated with the new formulation suggested by the a posteriori analysis, we now require the following bound \begin{equation}\label{new-bound} \|\bu - \bu_h\|_{\bL^4(\Omega)} \leq C_{\sharp} \| \bu-\bu_h\|_h.\end{equation}Will this be a simple consequence of a discrete Sobolev embedding or do we need a less trivial argument? Note that the $\|\bullet\|_h$ norm is not the same as $|\bullet|_{1,\cT_h}$ \dots ]}

\begin{theorem}[Pressure-robust error bound]\label{th:pressure-robustness}
    Assume that $(\bu,\bomega),(\bu_h,\bomega_h)$ are the unique solutions to \eqref{eq:reduced-cont} and \eqref{eq:reduced-discr}, respectively. Suppose further that the data satisfies \eqref{eq:small-data} with $r=1-C_\rS^2 C_{\mathrm{RT}}  \frac{8}{\gamma\wt\gamma} \Big(\frac{1}{\sqrt{\nu}} + \rF\Big)\|\f\|_{0,\Om}>0$, and that the continuous vorticity is more regular $\bomega\in \bH^1(\Omega)$. Then 
    \begin{equation}\label{eq:robust-cea}
        \|\bu-\bu_h\|_h + \|\bomega-\bomega_h\|_{0,\Omega} \leq 
       Ch|\bomega|_{1,\Omega} +  \inf_{\substack{(\bv_h,\btheta_h) \in \bV_h^0\times \bW_h \\ \alpha_h\|\bv_h\|_h=1}}
        \biggl[\frac{1}{r} + \frac{1}{r\alpha_h}\biggr](\|\bu-\bv_h\|_h + \|\bomega - \btheta_h\|_{0,\Omega}) .
    \end{equation}
\end{theorem}
\begin{proof}
Let us adopt the notation $\vec{u} := (\bu, \bomega),\ \vec{v}:=(\bv,\btheta),\  \vec{z}:=(\bz,\bzeta) \in \bV^0\times\bW$ and similarly for their discrete counterparts $\vec{u}_h,\vec{v}_h,\vec{z}_h\in \bV_h^0\times \bW_h$, denoting  the corresponding discrete norm as, e.g., $\|\vec{z}_h\|:=\|\bz_h\|_h+\|\bzeta_h\|_{0,\Omega}$. First we decompose 
\[ \vec{u} - \vec{u}_h = \vec{u} - \vec{v}_h + \vec{v}_h - \vec{u}_h = \vec{u} - \vec{v}_h - \vec{y}_h\]
and note that $\vec{y}_h : = \vec{u}_h - \vec{v}_h$ belongs to $\bV_h^0\times \bW_h$. Then, using the inf-sup condition of $a_h(\bullet,\bullet)$, the definition of $\vec{y}_h$, and the definition of the discrete problem \eqref{eq:reduced-discr}, we can write
%, for any $\vec{z}_h$, 
the following   
\begin{align}
 \nonumber \alpha_h \|\vec{z}_h \|\|\vec{y}_h\| & \leq \sup_{0 \neq \vec{z}_h\in \bV_h^0\times \bW_h} a_h(\vec{y}_h,\vec{z}_h) = \sup_{0 \neq \vec{z}_h\in \bV_h^0\times \bW_h} a_h(\vec{u}_h-\vec{v}_h,\vec{z}_h)\\
\nonumber  & = \sup_{0 \neq \vec{z}_h\in \bV_h^0\times \bW_h} \bigl( a_h(\vec{u}-\vec{v}_h,\vec{z}_h) - a_h(\vec{u} ,\vec{z}_h) + a_h(\vec{u}_h,\vec{z}_h)\bigr)\\
  & = \sup_{0 \neq \vec{z}_h\in \bV_h^0\times \bW_h} \bigl(a_h(\vec{u}-\vec{v}_h,\vec{z}_h) + \bigl[ 
  %F(\RTinterp 
  F_h(\vec{z}_h) - c_h^{\bu_h}(\vec{u}_h,\vec{z}_h) - a_h(\vec{u},\vec{z}_h) \bigr]\bigr).\label{aux01}
\end{align}
Consequently, dividing through $\alpha_h \|\vec{z}_h \|$ and using the boundedness of $a_h(\bullet,\bullet)$ 
%in \eqref{aux01} \sbc{is this the right reference? I would use (4.16).}, 
as well as the triangle inequality $\|\vec{u}-\vec{u}_h\| \leq \|\vec{u}-\vec{v}_h\| + \|\vec{y}_h\|$, we obtain the following estimate
\begin{equation}
    \|\vec{u}-\vec{u}_h\| \leq \inf_{\vec{v}_h\in \bV_h^0 \times \bW_h} \biggl(1 + \frac{1}{\alpha_h}\biggr) \|\vec{u}-\vec{v}_h\| + \sup_{
    %\substack{
    \vec{z}_h\in \bV_h^0 \times \bW_h
    %\\ \alpha_h\|\bv_h\|_h=1}
    } \frac{\bigl| 
    %F(\RTinterp \bv_h) 
    F_h(\vec{z}_h) - c_h^{\bu_h}(\vec{u}_h,\vec{z}_h) - a_h(\vec{u},\vec{z}_h) \bigr|}{\alpha_h\|\vec{z}_h\|}
\end{equation}
composed by the best approximation in $\bV_h^0\times \bW_h$ and the consistency error. 

Note that 
even if $\RTinterp  \bz_h$ is not in $\bH^1(\Omega)$, we have 
\begin{equation}\label{aux02} \|\RTinterp \bz_h\|_{\bL^4(\Omega)} \leq \|\RTinterp  \bz_h - \bz_h\|_{\bL^4(\Omega)}  + \|\bz_h\|_{\bL^4(\Omega)} \leq C \|\bz_h\|_{\bL^4(\Omega)} \leq C \|\bnabla_h\bz_h\|_{0,\Omega}\leq  {C_*}\|\bz_h\|_h,\end{equation}
thanks to triangle inequality, inverse estimates, and discrete Sobolev properties \cite[Theorem 4.12]{linke2016}. 

Next, we proceed to add $\pm c_h^{\bu}(\vec{u},\vec{z}_h)$ to 
 the numerator in the consistency error $a_h(\vec{u},\vec{z}_h) + c_h^{\bu_h}(\vec{u}_h,\vec{z}_h) - F_h(\vec{z}_h)$. The $+$ will contribute to complete a full residual $a_h(\vec{u},\vec{z}_h) + c_h^{\bu}(\vec{u},\vec{z}_h) - F_h(\vec{z}_h)$, so we need to investigate first the remainder terms as follows, adding and subtracting appropriate terms.   Applying H\"older's inequality, {property  \eqref{aux02}, and reverse triangle inequality}, gives 
\begin{align}
 \nonumber | c_h^{\bu_h}(\vec{u}_h,\vec{z}_h) - c_h^{\bu}(\vec{u},\vec{z}_h)| & = 
 \bigl|  - \frac{1}{\sqrt{\nu}}\int_\Omega  {\bu_h} \times (\bomega_h-\bomega) \cdot \RTinterp \bz_h + \rF \int_\Omega |\bu_h| {(\bu_h - \bu)} \cdot \RTinterp \bz_h \\
 \nonumber & \qquad + \frac{1}{\sqrt{\nu}}\int_\Omega  {(\bu-\bu_h )}\times \bomega \cdot \RTinterp \bz_h - \rF \int_\Omega [|\bu|-|\bu_h|] {\bu} \cdot \RTinterp \bz_h\bigr|\\
 \nonumber & \leq {\frac{C^2_*}{\sqrt{\nu}}} \|\bu_h\|_h \|\bomega-\bomega_h\|_{0,\Omega} \|\bz_h\|_h +
 {\rF C_*^2 \|\bu_h\|_h \|\bu-\bu_h\|_{0,\Omega}  \|\bz_h\|_h} \\
\nonumber & \qquad + {\frac{C_*}{\sqrt{\nu}} \|\bu-\bu_h\|_{\bL^4(\Omega)} \|\bomega\|_{0,\Omega} \|\bz_h\|_h  + \rF C_* \|\bu-\bu_h\|_{\bL^4(\Omega)} \|\bu\|_{0,\Omega}  \|\bz_h\|_h}\\
\label{aux03} & \lesssim {C_*[M_hC_*+MC_\sharp]\biggl(\rF + \frac{1}{\sqrt{\nu}}\biggr)} ( \|\bu-\bu_h\|_h+\|\bomega-\bomega_h\|_{0,\Omega}) \|\bz_h\|_h 
\end{align}
where for the last estimation we have used that $\vec{u},\vec{u}_h$ are solutions to the continuous and discrete problems featuring a continuous dependence on data that we denote here {by $M = \frac{2}{\gamma}\|\f\|_{0,\Om}$ and $M_h = \frac{2}{\wt\gamma}C_{F_h}\|\f\|_{0,\Om}$, respectively (cf. \eqref{eq:stability} and \eqref{eq:stability-h}, respectively); we have also used \eqref{new-bound}.} %Lemma~\ref{lem:aux}.  

We now look again at the numerator of the consistency error and  rewrite $\f$ in terms of the left-hand side of the momentum balance equation in \eqref{eq:ns-new}, use the fact that 
\[ \int_\Omega \nabla p \cdot \RTinterp \bz_h  =0\]
(see, e.g., \cite{linke2014role}) 
and apply integration by parts on the term $\sqrt{\nu}\int_\Omega \bcurl_h\bz_h\cdot\bomega$, to get 
\begin{align*}
& a_h(\vec{u},\vec{z}_h) + c_h^{\bu_h}(\vec{u}_h,\vec{z}_h) - %F(\RTinterp \bv_h)  
F_h(\vec{z}_h) + [ c_h^{\bu_h}(\vec{u}_h,\vec{z}_h) - c_h^{\bu}(\vec{u},\vec{z}_h)] \\
& \quad = \sqrt{\nu}\int_\Omega(\bcurl\bu-\bomega)\cdot\bzeta_h + \sqrt{\nu}\int_\Omega \bcurl\bomega\cdot(\bz_h - \RTinterp \bz_h) + [ c_h^{\bu_h}(\vec{u}_h,\vec{z}_h) - c_h^{\bu}(\vec{u},\vec{z}_h)]\\
& \quad \leq \sqrt{\nu} \|\bcurl \bomega \|_{0,\Omega} \| \bz_h - \RTinterp \bz_h\|_{0,\Omega}   + {C_*[M_hC_*+MC_\sharp]\biggl(\rF + \frac{1}{\sqrt{\nu}}\biggr)} ( \|\bu-\bu_h\|_h+\|\bomega-\bomega_h\|_{0,\Omega}) \|\bz_h\|_h\\
& \quad \leq \biggl[ ch|\bomega|_{1,\Omega} + {C_*[M_hC_*+MC_\sharp]\biggl(\rF + \frac{1}{\sqrt{\nu}}\biggr)} \| \vec{u}-\vec{u}_h\|\biggr]
 \|\bz_h\|_h.\end{align*}
 Here we have used the estimate in \eqref{aux03} and the fact that if $\vec{u}$ is the exact smooth solution of \eqref{eq:reduced-cont} then $\sqrt{\nu}\bcurl \bu = \bomega$ in $\Omega$, and have also used approximation properties of  $\RTinterp $. This yields 
\[\sup_{
    %\substack{
    \vec{z}_h\in \bV_h^0 \times \bW_h
    %\\ \alpha_h\|\bv_h\|_h=1}
    } \frac{\bigl| 
    %F(\RTinterp \bv_h) 
    F_h(\vec{z}_h) - c_h^{\bu_h}(\vec{u}_h,\vec{z}_h) - a_h(\vec{u},\vec{z}_h) \bigr|}{\alpha_h\|\vec{z}_h\|} 
    \leq Ch |\bomega|_{1,\Omega} + {C_*[M_hC_*+MC_\sharp]\biggl(\rF + \frac{1}{\sqrt{\nu}}\biggr)} \| \vec{u}-\vec{u}_h\|.
    \]
And the proof is complete after combining this estimate, the small data assumption \eqref{eq:small-data}, and \eqref{aux02}. 
\end{proof}

%%%%%%%%%%%%%%%%%%%%%%%%%%%%%%%%%%%%%%%%%%%%%%%%%%%%%%%%%%%%%%%%%%%%%%%%%%%%%%
\section{A posteriori error analysis}\label{sec:aposteriori}
%\subsection{Error estimator and its robustness}
{This section is devoted to a reliable and efficient computable  error control. The derivation of the reliable error bound departs from the assumption} that $(\bu,\bomega,p)\in\bV\times\bL^2(\Om)\times \rL_0^2(\Om)$ solves
\eqref{eq:weak} and  $(\bu_h, \bomega_h,p_h)\in \bV_h\times\bW_h\times \rQ_h$ solves 
\eqref{eq:weak-formulation-h} to define the errors 
$\be_\bu:=\bu-\cJ\bu_h$, $\be_\bomega:=\bomega-\bomega_h $, 
and  $e_p:= p-p_h$
as in \eqref{eq:error-dec} except that $\be_\bu:=\bu-\cJ\bu_h\in \bV$ 
is {\em not} equal to 
$\bu-\bu_h\in\bV+\bV_h$.

Recall from  \eqref{eq:dep-b}  that $\bu$ belongs to $\bK$   and thereafter
define the bilinear form $ \cB^{\bu}\big(\bullet,\bullet\big)$ 
as in \eqref{def:A-hat-u}
(with $\wh\bu$ replaced by ${\bu}$). Recall \eqref{eq:BNB-1} and deduce that there exists
some test function $(\bv,\btheta,q)  \in \bV\times\bL^2(\Om)\times \rL_0^2(\Om)$ of norm
$\|((\bv,\btheta),q)\|\le 1$ at most one and 
\[
\frac \gamma2 \|(\be_{\bu},\be_\bomega,e_p ) \|
= \cB^{\bu}\big(\be_{\bu},\be_\bomega,e_p ; \bv,\btheta,q\big).
\]
The solution  $(\bu,\bomega,p)\in\bV\times\bL^2(\Om)\times \rL_0^2(\Om)$ to
\eqref{eq:weak} also satisfies 
\(
 \cB^{\bu}\big(\bu,\bomega,p ; \bv,\btheta,q\big)=F( \bv,\btheta),
\)
whence
\[
\frac \gamma2 \|(\be_\bu,\be_\bomega,e_p ) \|
=\int_\Omega \f \cdot \bv-\cB^{\bu}\big(\cJ\bu_h,\bomega_h,p_h ; \bv,\btheta,q\big).
\]
The right-hand side in this identity defines a residual in terms of the test functions 
$ \bv,\btheta$, and $q$. The discrete solution  $(\bu_h, \bomega_h,p_h)\in \bV_h\times\bW_h\times \rQ_h$ to 
\eqref{eq:weak-formulation-h} involves the discrete right-hand side from \eqref{def:Fh} and the operator $\RTinterp $ from \eqref{eq:def-RT}. 
With the definition of the Crouzeix--Raviart interpolation 
$\bv_h:=\CRinterp \bv$ and the 
piecewise constant integral means $\bomega_h$ and $p_h$ of
 $\bomega$ and $p$, respectively, we investigate the first identity of
 \eqref{eq:weak-formulation-h}
 for $\btheta_h=\bzero$ and \eqref{def:ah}-\eqref{def:ch}, 
 namely
\begin{align*}
\int_\Omega \f \cdot \RTinterp \bv_h
&= \int_\Omega {\frac{1}{\kappa}}\bu_h\cdot\RTinterp \bv 
+\!\! \sum_{F\in \cE^{\mathrm{int}}_h}
\frac{\vartheta}{h_F}\int_F \bigl(
\nu [\![\bu_h\times\bn]\!]_F\cdot[\![\bv_h\times\bn]\!]_F +
[\![\bu_h\cdot\bn]\!]_F[\![\bv_h\cdot\bn]\!]_F
\bigr)\nonumber\\
& \quad + \int_\Omega \left( \sqrt{\nu}\ \bomega_h \cdot\bcurl\bv
-p_h\div\,\bv\right) + 
\int_\Omega \bigl( \rF  |\bu_h|\, \bu_h 
-\nu^ {-1/2}\bu_h\times\bomega_h
\bigr)\cdot \RTinterp \bv
\end{align*}
with  $\int_K \left( \sqrt{\nu}\ \bomega_h \cdot\bcurl\bv_h
-p_h\div\,\bv_h\right) =\int_K \left( \sqrt{\nu}\ \bomega_h \cdot\bcurl\bv
-p_h\div\,\bv\right) $ for all $K\in \cT_h$ from 
the integral mean property of the gradients 
for the  Crouzeix--Raviart interpolation in the last step.
The combination of the last two identities and the definitions  
 \eqref{def:A-hat-u} and \eqref{eq:def-A} reveal the key identity 
\begin{align*}
\frac \gamma2 \|(\be_{\bu},\be_\bomega,e_p ) \|
&=\int_\Omega \f \cdot (\bv- \RTinterp \bv)
+\int_\Omega{\kappa^{-1}} (\bu_h\cdot\RTinterp \bv -(\cJ\bu_h)\cdot\bv) \\
&\quad +\!\! \sum_{F\in \cE^{\mathrm{int}}_h}
\frac{\vartheta}{h_F}\int_F \bigl(\nu [\![\bu_h\times\bn]\!]_F\cdot[\![\bv_h\times\bn]\!]_F 
+[\![\bu_h\cdot\bn]\!]_F[\![\bv_h\cdot\bn]\!]_F\bigr)\nonumber\\
& \quad + 
\int_\Omega \bigl( \rF  |\bu_h|\,\bu_h -\nu^ {-1/2} \bu_h\times\bomega_h
\bigr)\cdot \RTinterp \bv
+\int_\Omega \bigl(\nu^ {-1/2}\bu\times\bomega_h- \rF  |\bu|\,\cJ\bu_h 
\bigr)\cdot \bv\\ &
\quad - \int_\Omega \left( \nu^ {1/2}\btheta \cdot\bcurl \cJ\bu_h- \bomega_h\cdot \btheta\right) 
+\int_\Omega q\; \div \cJ\bu_h.
\end{align*}
The jump terms with $[\![\bu_h]\!]_F=[\![\bu_h-\cJ\bu_h]\!]_F$ and 
$[\![\bv_h]\!]_F = [\![\bv_h-\bv]\!]_F$ allow for standard trace inequalities and interpolation
local interpolation error estimates and eventually verify
\[
\!\! \sum_{F\in \cE^{\mathrm{int}}_h}
\frac{\vartheta}{h_F}\int_F \bigl(\nu [\![\bu_h\times\bn]\!]_F\cdot[\![\bv_h\times\bn]\!]_F 
+[\![\bu_h\cdot\bn]\!]_F[\![\bv_h\cdot\bn]\!]_F\bigr)
\lesssim (1+\nu) \pwnorm{ \bu_h-\cJ\bu_h}%_{H^1(\cT_h)}
\]
with $ \pwnorm{ \bv-\bv_h}%|_{H^1(\cT_h)}
\lesssim 1$ in the last step.
The remaining terms that involve  a factor $\bv$, $\bv_h=\CRinterp \bv$, or 
$\RTinterp \bv_h=\RTinterp \bv$ combine to 
one residual term plus perturbations. The residual reads 
\[
\int_\Omega 
 \bigl(\f -\kappa^ {-1}\bu_h+ \nu^ {-1/2} \bu_h\times\bomega_h-\rF  |\bu_h|\,\bu_h 
\bigr)\cdot (\bv- \RTinterp \bv)
+\int_\Omega{\kappa^{-1}}(\bu_h-\cJ\bu_h) \cdot\bv
\]
and the remaining perturbations  read
\begin{align*}
&\nu^ {-1/2} \int_\Omega  \left((\bu-\bu_h)\times\bomega_h\right) \cdot\bv 
+ \int_\Omega \rF  \bigl( (|\bu_h|-  |\bu|)\,\bu_h  + |\bu|(\bu_h-\cJ\bu_h)
 \bigr)\cdot \bv \\ &\le 
 \bigl[ \nu^ {-1/2} \|\bomega_h\|_{0,\Omega} +  \rF  \| \bu_h  \|_{0,\Omega} \bigr] \,\| \bu-\bu_h\|_{\bL^4(\Omega)} \|\bv\|_{\bL^4(\Omega)} 
 + \rF \|\bu\|_{\bL^4(\Omega)} \|\bv\|_{\bL^4(\Omega)} \|\bu_h-\cJ\bu_h\|_{0,\Omega} \\
 &\le  
  C_{dS}(4)C_{S}(4)  \bigl[ \nu^ {-1/2} \|\bomega_h\|_{0,\Omega} + \rF  \| \bu_h  \|_{0,\Omega}\bigr] 
 \pwnorm{ \bu-\bu_h} %\\ & \quad
 + C_{dS}(4) C_{S}(4)  \rF |\bu|_{1,\Omega} h_{\max} C_I  
 \pwnorm{\bu_h-\cJ\bu_h}%|_{\bH^1(\cT_h)}
\end{align*}
 with discrete Sobolev (resp.  Sobolev) inequalities with constant $  C_{dS}(4) \lesssim 1$
(resp. $  C_{S}(4)  \lesssim 1$), 
 interpolation error estimates, and  $ | \bv|_{1,\Omega}\le 1$ in the last step.

The discrete equations also reveal 
$a_h(\bu_h,\bomega_h; \bzero, \btheta_h)=0=b_h(\bu_h,\bomega_h;q_h)$ 
for all piecewise constant $ \btheta_h \in \bW_h$ and
for all  piecewise constant $ q_h \in Q_h$ 
with integral mean zero over the domain.
Those identities localise (utilise $\int_\Omega \div_h  \bu_h=0$ from a piecewise 
integration by parts to overcome the global constraint in  $Q_h$)  
and lead to the discrete identities 
\[
\sqrt{\nu}\, \bcurl_h \bu_h=\bomega_h\quad\text{and}\quad  \div_h \bu_h=0
\]
for the piecewise constant functions (from piecewise action of the differential  
operators).  Hence  %\vspace{-6mm}
\begin{align*}
&- \int_\Omega \left( \sqrt{\nu}\ \btheta \cdot\bcurl \cJ\bu_h- \bomega_h\cdot \btheta\right) 
+\int_\Omega q\; \div \cJ\bu_h \\ 
&\qquad =
\int_\Omega  \sqrt{\nu}\ \btheta \cdot\bcurl_h (\bu_h-\cJ\bu_h) 
-\int_\Omega q\; \div_h (\bu_h-\cJ\bu_h)\\ 
&\qquad \le ( \sqrt{\nu}\|  \btheta\| +\sqrt 2  \|q\|) \; \pwnorm{ \bu_h-\cJ\bu_h}%|_{\bH^1(\cT_h)}
\le \sqrt{2+\nu} \; \pwnorm{ \bu_h-\cJ\bu_h} %_{\bH^1(\cT_h)}
\end{align*}
with Cauchy inequalities and $\|((\bv,\btheta),q)\|\le 1$ in the last steps.

The  combination of all the above estimates and the  standard error estimate 
$\| h_\cT^ {-1} (\bv-\RTinterp \bv)\|_{0,\Omega}\lesssim \pwnorm{ \bv_h}%|_{\bH^1(\cT_h)}
\le 1$
reveal that 
\begin{align*}
&\|(\be_\bu,\be_\bomega,e_p ) \| 
-2C_{dS}(4)C_{S}(4)\gamma^ {-1} \left( \nu^ {-1/2}\, \|\bomega_h\|_{0,\Omega}  + \rF  \| \bu_h  \|_{0,\Omega}\right)  
\pwnorm{\bu-\bu_h} 
\\
&\qquad \lesssim 
(1+\nu^ {1/2} + (\rF +{\kappa_{\min}^{-1}})h_{\max} ) \pwnorm{\bu_h-\cJ\bu_h}  + \|h_\cT (\f- \kappa^ {-1}\bu_h  +\nu^ {-1/2} \bu_h\times\bomega_h -  \rF  |\bu_h|\,\bu_h )\|_{0,\Omega}.
\end{align*}
(The notation $\lesssim $ includes generic constants as well as  $ \gamma \approx 1$, while we keep $\gamma$
in the rather explicit negative term of the lower bound.) 
Since  $\be_\bu= \bu-\cJ\bu_h$, the   triangle inequality 
\(
 \pwnorm{\bu-\bu_h}%|_{\bH^1(\cT_h)} 
 \le \|(\be_\bu,\be_\bomega,e_p ) \| +  \pwnorm{\bu_h-\cJ\bu_h},%|_{\bH^1(\cT_h)}
\) provides  
\begin{align}\label{eqccapostlast1}
 &\pwnorm{\bu-\bu_h}%|_{\bH^1(\cT_h)}
 + \| \be_\bomega\|_{0,\Omega} + \|e_p \|_{0,\Omega} \nonumber
-2C_{dS}(4)C_{S}(4)\gamma^ {-1} \left( \nu^ {-1/2}\, \|\bomega_h\|_{0,\Omega}  + \rF  \| \bu_h  \|_{0,\Omega}\right)  
\pwnorm{\bu-\bu_h}%|_{\bH^1(\cT_h)} 
\\
&\quad \lesssim 
(1+\nu^ {1/2} + (\rF +{\kappa_{\min}^{-1}})h_{\max} ) \pwnorm{\bu_h-\cJ\bu_h} 
+ \|h_\cT (\f- \kappa^ {-1}\bu_h  +\nu^ {-1/2} \bu_h\times\bomega_h -  \rF  |\bu_h|\,\bu_h )\|_{0,\Omega}.
\end{align}
The proof concludes with a discussion of the factor 
$C_{dS}(4)C_{S}(4) \left( \nu^ {-1/2}\, \|\bomega_h\|_{0,\Omega}  + \rF  \| \bu_h  \|_{0,\Omega}\right) \le \gamma/4$
which follows from \eqref{eq:stability-h} for small sources $\f$ in $\bL^ 2(\Omega)$. 
The point is that the latter estimate allows us to absorb the negative term on the lower bound
of \eqref{eqccapostlast1} and this leads to the a posteriori error control. 
The discrete Friedrichs
inequality $ \| \bu_h  \|_{0,\Omega}\le  C_{dF}  \pwnorm{ \bu_h }$ for all Crouzeix--Raviart functions with
homogeneous boundary conditions and 
$\pwnorm{\bu_h} \leq C_{\sharp} \| \bu_h\|_h$ by  \eqref{eq:aux-24} provide 
\[
\nu^ {-1/2}\, \|\bomega_h\|_{0,\Omega}  + \rF  \| \bu_h  \|_{0,\Omega}
\le \nu^ {-1/2}\, \|\bomega_h\|_{0,\Omega}  +  C_{dF} C_{\sharp} \rF\| \bu_h  \|_h
\le \sqrt{ \nu^ {-1}+ C_{dF}^2 C_{\sharp}^2 \rF^2} \dfrac{2}{\wt\gamma}C_{F_h}\|\f\|_{0,\Om}
\]
with a Cauchy inequality and \eqref{eq:stability-h} in the last step. Hence  in case that
\begin{equation}
\label{eqccapostlast1}
 \|\f\|_{0,\Om}\le
     \gamma{\wt\gamma} /
     \left( 8 C_{F_h} C_{dS}(4)C_{S}(4)    \sqrt{ \nu^ {-1}+ C_{dF}^2 C_{\sharp}^2 \rF^2} \right),
\end{equation}
the lower bound in \eqref{eqccapostlast1} provides error control over 
$\frac 12 \pwnorm{\bu-\bu_h}%|_{\bH^1(\cT_h)}
+ \| \be_\bomega\|_{0,\Omega} + \|e_p \|_{0,\Omega} $. This concludes the proof of the a~posteriori error
estimate.  The final form of the explicit residual-based a~posteriori error estimate 
employs  the well-established formula 
\[
\pwnorm{\bu_h-\cJ\bu_h}^2%|_{H^1(\cT_h)}^ 2
 \lesssim 
\sum_{F\in \cE} h_F^ {-1} \| [\bu_h]_F\times \bn_F\|_F^ 2
\approx \sum_{F\in \cE} h_F \| [\bD\bu_h]_F\times \bn_F\|_F^ 2.
\]
This gives rise to the explicit residual-based 
a~posteriori error  estimator with the contribution 
\begin{equation}\label{eq:local-estimator} 
\eta^2(K):= |K|^{2/d}
\| \f- \kappa^ {-1}\bu_h  
+\nu^ {-1/2} \bu_h\times\bomega_h -  \rF  |\bu_h|\,\bu_h \|_{0,K}^ 2
+ %\cred{\nu?}
|K|^{1/d} \sum_{F\in\mathcal{F}(K)} \| [\bD \bu_h]_F\times \bn_F \|_{0,F}^2
\end{equation}
for each triangle $K\in\cT_h$ and the global version $\eta(\mathcal{T}_h):=\sqrt{\sum_{K\in\mathcal{T}_h} \eta^ 2(K)}$.

\begin{theorem}[A posteriori error control]\label{thm:apost} 
Provided the source is sufficiently small such that \eqref{eqccapostlast1}
holds, we have reliability 
   \[
 \pwnorm{\bu-\bu_h}+ \| \be_\bomega\|_{0,\Omega} + \|e_p \|_{0,\Omega} \lesssim \eta(\mathcal{T}_h).
   \]
{Efficiency holds} even in local form up to data oscillations: For any $K\in\cT_h$ with neighbourhood 
$\Omega(K)$ covered by  the union of all simplices in $\cT_h$
with zero distance to $K$, {we have} 
\[
\eta(K)\lesssim \| D_\textrm{pw} (\bu-\bu_h)\|_{0,\Omega(K)} +
\|\bu-\bu_h\|_{0,K}+\| \bomega-\bomega_h\|_{0,K}+\|p_h-p\|_{0,K}+\textrm{osc}_k(\f, K)
.
\]
The generic multiplicative constants behind the notation $\lesssim$ 
exclusively depend on the material constants, upper bounds of the {solutions} $\bu$ and 
$\bu_h$ in $\bL^ 2(\Omega)$, and the shape regularity of $\cT_h$.
\end{theorem}

\begin{proof}
Reliability  follows from the analysis prior to the statement of Theorem~\ref{thm:apost}. The remaining efficiency 
follows with Verf\"urth's bubble-function methodology. This is well established for 
the side contributions 
\[
|K|^{1/(2d)}  \| [\bD \bu_h]_F\times \bn_F \|_{L^2(\partial K)}
\lesssim  \| D_\textrm{pw} (\bu-\bu_h)\|_{0,\Omega(K)}
\] 
and follows verbatim \cite{carstensen2024adaptive,verfurth1989posteriori}. The volume contribution, however,  challenges with a technical 
(possibly unexpected) trap.
The overall idea is to design a local test  function $\bv_K$ that 
allows an evaluation of the residual functional
and thereby involves the exact solution. 
The data oscillations arise from the approximation of the source $\f $ in the very first step. 
In order to estimate the volume contribution 
\[
\mu_K:= |K|^{1/d}
\| \f- \kappa^ {-1}\bu_h  +\nu^ {-1/2} \bu_h\times\bomega_h 
-  \rF  |\bu_h|\,\bu_h \|_{0,K}
\]
of  $K\in\cT_h$, 
we consider the  volume-bubble function  
 ${b_K}\in W^{1,\infty}_0({K})$ on a simplex ${K\subset \mathbb{R}^d}$. The latter 
 is the product
of all $d+1$ barycentric coordinates of the vertices of $K$ times a factor $(d+1)^{d+1}$ 
for the normalisation $0\le b_K\le 1=\max b_K$ in $K$.
With the abbreviation 
\[
\bv_K:=\Pi_1\f- \kappa^ {-1}\bu_h  +\nu^ {-1/2} \bu_h\times\bomega_h -  \rF  |\bu_h|\,\bu_h,
\]
the  admissible test function reads
$ b_K \bv_K$ and belongs to $\bV$ (for $b_K$ is extended by zero). 
The typical application in 
Verf\"urth's bubble-function methodology considers only polynomial test functions that allow a standard  inverse estimate  
\(
\| \g \|_{0,K}\le C_\mathrm{eq}\| b_K^{1/2}\g \|_{0,K}
\quad\text{for all }\g\in P_1(K)^{d}.
\)
In the application of this paper $\bv_K$ is  polynomial up to this extra term 
$b_K \rF  |\bu_h|\,\bu_h$ and hence requires a  new inverse estimate. 
After surprisingly large efforts, Appendix \ref{sec:inverse} provides  
\[
\| |\f|\f+\g\|_{0,K}\le C_\mathrm{eq}\| b_K^{1/2}( |\f|\f+\g)\|_{0,K}
\quad\text{for all }\f,\g\in P_1(K)^{d}.
\]
The  test function $ b_K \bv_K\in \bV$ therefore satisfies 
the novel  nonstandard inverse estimate 
\[
\|%\Pi_1\f- \kappa^ {-1}\bu_h  +\nu^ {-1/2} \bu_h\times\bomega_h -  \rF  |\bu_h|\,\bu_h 
\bv_K\|_{0,K} \le C_\mathrm{eq}\|b_K^ {1/2} \bv_K\|_{0,K}.
\]
The remaining arguments in the efficiency proof are standard nowadays and (partly) apply verbatim.
This, a triangle inequality, and the definition of the oscillations reveal 
\begin{equation}\label{eqccneweffeiciencyproffapril_1}
\mu_K
%|K|^{2/d}/2 \| \f- \kappa^ {-1}\bu_h  +\nu^ {-1/2} \bu_h\times\bomega_h 
%-  \rF  |\bu_h|\,\bu_h \|_{\bL^ 2(K)}^{2}
\le |K|^{1/d}\| \f-\Pi_1\f \|_{0,K} +  \bv_K\|_{0,K}
\le  \textrm{osc}_k(\f, K)+C_\mathrm{eq}|K|^{1/d}\|b_K^ {1/2} \bv_K\|_{0,K}.
\end{equation}
With $\bv$ replaced by  $ b_K \bv_K\in \bV$ in the first equation in \eqref{eq:weak},
the  last term relates to 
\begin{align}\label{eqccneweffeiciencyproffapril_2}
\|b_K^ {1/2} \bv_K\|_{0,K}^2&=\int_K  b_K \bv_K \cdot (
\Pi_1\f- \kappa^ {-1}\bu_h  +\nu^ {-1/2} \bu_h\times\bomega_h -  \rF  |\bu_h|\,\bu_h)\nonumber
\\
&= \int_K  b_K \bv_K \cdot (\Pi_1-1) \f
+\int_K  \left(\nu^{1/2} \bomega\cdot \bcurl ( b_K \bv_K )  - p\,\div ( b_K \bv_K )\right) 
\\
&\quad +\int_K  b_K \bv_K \cdot \left(
\kappa^ {-1}(\bu-\bu_h)  
-\nu^ {-1/2}( \bu\times\bomega-\bu_h\times\bomega_h) 
+ \rF ( |\bu|\,\bu- |\bu_h|\,\bu_h)\right). \nonumber
\end{align}
The last right-hand side \eqref{eqccneweffeiciencyproffapril_2}
consists of six summands  we enumerate $S_1,\dots, S_6$
in the order displayed. Notice that there is a factor $|K|^{1/d}$ omitted compared to
\eqref{eqccneweffeiciencyproffapril_1} and hence we can afford a factor $|K|^{-1/d}$ in the upper bounds of $S_1,\dots, S_6$. A Cauchy inequality and the definition of the oscillations provide
\[
S_1= \int_K  b_K \bv_K \cdot (\Pi_1-1) \f
\le \|b_K^ {1/2} \bv_K\|_{0,K}|K|^{-1/d}\textrm{osc}_k(\f, K).
\]
Recall that $ b_K \bv_K $ has a support on the simplex $K$ and vanishes on its boundary $\partial K$.
Hence the Gauss and Stokes theorems show $\int_K \div ( b_K \bv_K )=0$ and 
$ \int_K  \bomega\cdot \bcurl ( b_K \bv_K )=0$. Since $p_h$ and $\bomega_h$ are constant on $K$, 
we infer 
\begin{align*}
S_2+S_3&=
\nu^{1/2}\int_K ( \bomega-\bomega_h)\cdot \bcurl ( b_K \bv_K ) 
+ \int_K (p_h-p) \,\div ( b_K \bv_K )\\
&\lesssim (\| \bomega-\bomega_h\|_{0,K}+\|p_h-p\|_{0,K}) \|b_K^ {1/2} \bv_K\|_{0,K}|K|^{-1/d} 
\end{align*}
with an inverse estimate $|b_K \bv_K|_{1,K}\lesssim  \|b_K^ {1/2} \bv_K\|_{0,K}|K|^{-1/d}$
in the last step. A Cauchy inequality controls  the term 
\[
S_4=\int_K  b_K \bv_K \cdot 
\kappa^{-1}(\bu-\bu_h)\le {\kappa^{-1}_{\min}}\|b_K^ {1/2} \bv_K\|_{0,K}\| \bu-\bu_h\|_{0,K} .
\]
The nonlinear term $S_5$ is related to
$
\nu^ {1/2}S_5\le\|b_K^ {1/2} \bv_K\|_{\bL^\infty(K)}
%\bu\times\bomega-\bu_h\times\bomega_h
\|(\bu-\bu_h)\times\bomega+\bu_h\times(\bomega-\bomega_h)\|_{\bL^1(K)}.
$
Since $\bomega$ and $\bu_h$ are bounded in $L^2(\Omega)$ by the a priori error analysis, 
a Cauchy inequality for the $L^1$ integrals may be written as
$\|(\bu-\bu_h)\times\bomega\|_{\bL^1(K)}\lesssim \|\bu-\bu_h\|_{0,K}$ and
$\|\bu_h\times(\bomega-\bomega_h)\|_{\bL^1(K)}\lesssim \| \bomega-\bomega_h\|_{0,K}$.
This and the inverse estimate $\|b_K^ {1/2} \bv_K\|_{\bL^\infty(K)}\lesssim 
|K|^{-1/d}\|b_K^ {1/2} \bv_K\|_{0,K}$ establish
\[
S_5\lesssim \nu^{-1/2} |K|^{-1/d}\|b_K^ {1/2} \bv_K\|_{0,K}
\left( \|\bu-\bu_h\|_{0,K}+\| \bomega-\bomega_h\|_{0,K}
\right).
\]
We can afford the above inverse estimate $\|b_K^ {1/2} \bv_K\|_{\bL^\infty(K)}\lesssim 
|K|^{-1/d}\|b_K^ {1/2} \bv_K\|_{0,K}$ in the
 nonlinear term 
\[
S_6\lesssim|K|^{-1/d}\|b_K^ {1/2} \bv_K\|_{0,K}
 |\rF| \| |\bu|\,\bu- |\bu_h|\,\bu_h\|_{\bL^ 1(K)}.
\]
The elementary estimate $|\, |a|a-|b| b\, |\le(|a|+|b|) |a-b| $ for vectors $a,b\in\mathbb{R}^ d$ 
and a Cauchy and triangle inequality provide 
\[
\| |\bu|\,\bu- |\bu_h|\,\bu_h\|_{\bL^1(K)}\le 
(\| \bu\|_{0,K}+\| \bu_h\|_{0,K} ) \| \bu- \bu_h\|_{0,K}.
\]
Let us hide several constants in the notation $\lesssim$ as before but also  
upper bounds of $\| \bu_h\|_{0,K}$ and $\| \bu\|_{0,K}$ to infer
\[
S_6\lesssim|K|^{-1/d}\|b_K^ {1/2} \bv_K\|_{0,K} |\rF| \| \bu- \bu_h\|_{0,K}.
\]
Let us hide the material constants $|\rF|$, {$\kappa_{\min}^{-1},\kappa_{\max}^{-1}$}, $\nu$,   and  $\nu^ {-1}$
in the notation $\lesssim$ 
for a summary of the above estimates of $S_1$,\dots, $S_6$ in 
\eqref{eqccneweffeiciencyproffapril_2}. After a division of 
the factor $\|b_K^ {1/2} \bv_K\|_{0,K}$ (if positive as else there is nothing left to prove) 
we infer 
\[
|K|^{1/d} \|b_K^ {1/2} \bv_K\|_{0,K}\lesssim
 \|\bu-\bu_h\|_{0,K}+\| \bomega-\bomega_h\|_{0,K}+\|p_h-p\|_{0,K}+\textrm{osc}_k(\f, K).
\]
The combination with \eqref{eqccneweffeiciencyproffapril_1} concludes the proof of the 
local efficiency 
\[
\mu_K \lesssim 
\|\bu-\bu_h\|_{0,K}+\| \bomega-\bomega_h\|_{0,K}+\|p_h-p\|_{0,K}+\textrm{osc}_k(\f, K)
\]
of the volume contribution in lower order terms. The efficiency proof does not require any smallness assumption neither on the mesh-size nor on the closeness of the exact and discrete solution. 
\end{proof}

 %%%%%%%%%%%%%%%%%%%%%%%%%%%%%%%%%%%%%%%%%%%%%%%%%%%%%%%%%%%%%%%%%%%%%%%%%%%%%%%%%%%%%%%%%%%%%%%%%%%%%%%
 \section{Numerical results}\label{sec:examples}
 In this section we report and discuss a number of numerical examples to illustrate the performance of the proposed mixed finite element schemes and a posteriori error estimators.  The realisation of the numerical methods designed in the paper is conducted with a combination of open source software packages in the so-called \texttt{Gridap} ecosystem~\cite{badia20,verdugo22,badia22}. In all cases, we use a Newton method with exact Jacobian, setting a tolerance of $10^{-8}$ on the $\ell^2$ norm of the increment and $10^{-12}$ on the $\ell^\infty$ norm of the residual. The linear systems were solved either with UMFPACK (2D cases) or the multifrontal massively parallel sparse direct solver MUMPS (3D cases).

 The adaptive mesh refinement procedure follows Algorithm~\ref{alg:1}. It comprises an standard SOLVE~$\rightarrow$ ESTIMATE $\rightarrow$ MARK $\rightarrow$ ADAPT loop. This algorithm can be in principle combined with any kind of adaptive mesh approach. Our implementation of Algorithm~\ref{alg:1} particularly leverages hierarchically-adapted (i.e., nested) non-conforming octree-based meshes; see, e.g., \cite{badia20}. Forest-of-octrees meshes can be seen as a two-level
decomposition of $\Omega$, referred to as macro and micro level, respectively. The macro level is a suitable {\em conforming} partition $\mathcal{C}_h$ of $\Omega$ into quadrilateral ($d=2$) or hexahedral cells ($d=3$). This mesh, which may be generated, e.g., using an unstructured mesh generator,  is referred to as the coarse mesh. At the micro level, each of the cells of $\mathcal{C}_h$ becomes the root of an adaptive octree with cells that can be recursively and dynamically refined or coarsened using the so-called $1:2^d$ uniform partition rule. If a cell is marked for refinement, then it is split into $2^d$ children cells by subdividing all parent cell edges. If all children cells of a parent cell are marked for coarsening, then they are collapsed into the parent cell. The union of all leaf cells in this hierarchy forms the decomposition of the domain at the micro level, i.e., $\cT_h$.  While these meshes are made of quadrilaterals or cubes, we split their elements into simplices (2 triangles per mesh quadriteral in 2D, 6 tetrahedra per mesh cube in 3D) to obtain the simplicial meshes for our formulation; see, e.g., Figure~\ref{fig:aposte}.

 \begin{algorithm}[t!]
 \algsetup{indent=2em}
 \caption{Adaptive Mesh Refinement and coarsening algorithm}
  \label{alg:1}
  \begin{algorithmic}[1]
  \STATE \textbf{INPUT:} coarse mesh $\mathcal{C}_h$, $\theta^{\rm r} \in (0, 1)$, $\theta^{\rm c} \in (0, 1)$, $l_{\mathrm{max}}$
  \STATE \textbf{OUTPUT:} solution of \eqref{eq:weak-formulation-h} on an adapted mesh
  \STATE Set $\cT_h$ to be the result of (optionally) applying several levels of uniform refinement to $\mathcal{C}_h$
  \STATE SOLVE the discrete problem \eqref{eq:weak-formulation-h} on $\cT_h$
  \FOR {$l = 1,\ldots l_{\mathrm{max}}$}
    \STATE ESTIMATE: for every cell $K\in \cT_h$, compute the local error indicator $\eta(K)$ from \eqref{eq:local-estimator}    
    \STATE MARK for refinement a $\mathcal{M}^{\rm r} \subset \cT_h$ with largest $\eta(K)$ such that $|\mathcal{M}^{\rm r}| \approx \theta^{\rm r} |\cT_h|$  
    \STATE MARK for coarsening another set $\mathcal{M}^{\rm c} \subset \cT_h$ with smallest $\eta(K)$ such that $|\mathcal{M}^{\rm c}| \approx \theta^{\rm c}|\cT_h|$\; 
    \STATE ADAPT: refine $K \in \mathcal{M}^{\rm r}$ and coarsen $K \in \mathcal{M}^{\rm c}$ to construct a new $\mathcal{T}_{h}$ for the next step
    \STATE SOLVE the discrete problem \eqref{eq:weak-formulation-h} on $\cT_h$
  \ENDFOR
  \end{algorithmic}
 \end{algorithm}
 
The adaptive meshes resulting from this approach are non-conforming. In particular, they have hanging faces at the interface of neighbouring cells at different levels of refinement. Mesh non-conformity requires special adaptations to the Crouzeix--Raviart finite element space used in our formulation.  In particular, following \cite{bangerth17}, we impose a set of linear multi-point constraints for the velocity degrees of freedom located at these non-conforming interfaces such that the average of the averages on the finer children faces equals the average on the coarse parent face (so-called Option C in \cite{bangerth17}). While this approach was mathematically proven in \cite{bangerth17} to lead to optimal convergence rates in the case of the Douglas–Santos–Sheen–Ye finite element, our numerical results confirm that this is also the case for the Crouzeix--Raviart finite element.  We used the \texttt{GridapP4est.jl}~\cite{GridapP4est2025} Julia package in order to handle such kind of meshes (including facet integration on non-conforming interfaces as per required by the stabilization terms in the formulation and the computation of the a posteriori error estimator) and finite element space constraints. This package, built upon the \texttt{p4est} meshing engine~\cite{Burstedde2011}, is endowed with the so-called Morton space-filling curves, and it provides high-performance and low-memory footprint algorithms to handle forest-of-octrees.

\paragraph{Verification of convergence to smooth solutions.} First we exemplify the convergence of the method against smooth manufactured solutions and illustrate the pressure robustness property. For this we proceed similarly as in \cite[Sect. 5.2]{linke2014role} and consider the domain $\Omega = (0,1)^2$ with  two values for the kinematic viscosity $\nu =1$ and $\nu = 10^{-4}$, and with two discretisations (including or not the interpolation of suitable test 
%and trial 
functions in $a_h$, $F_h$ and $c_h^{\bu_h}$ that lead to a variational crime). We also consider a manufactured stream function $\xi = x^2(1-x)^2y^2(1-y^2)$ and a manufactured velocity, vorticity and Bernoulli pressure that solve the coupled system \eqref{eq:ns-new} and that satisfy also the homogeneous boundary condition and zero-mean Bernoulli pressure constraint \eqref{eq:boundary}, as follows
 \[ \bu = \bcurl\,\xi, \quad \bomega = \sqrt{\nu}\,\bcurl\,\bu, \quad p = x^3 + y^3-\frac12.\]
The remaining model parameters and stabilisation constant are taken as $\rF = 1$, $\kappa = 1$, and  $\vartheta = 10$, respectively.

\begin{table}[t!]
\setlength{\tabcolsep}{5pt}
    \centering
 \begin{tabular}{|r|c|c|c|c|c|c|c|c|c|}
    \hline
DoFs & $h$  & $ \|\bu-\bu_h\|_h$%texttt{e}_{\bu}$ 
& rate & $\|\bomega-\bomega_h\|_{0,\Omega}$%\texttt{e}_{\bomega}$ 
& rate & $\|p-p_h\|_{0,\Omega}$
%\texttt{e}_p$ 
& rate & $\texttt{loss}_{\mathrm{div}}$ & $\texttt{loss}_{\mathbf{curl}}$  \\
\hline
    \multicolumn{10}{|c|}{standard $\mathbf{CR}-\mathbb{P}_0-\mathbb{P}_0$ scheme, with kinematic viscosity $\nu=1$}\\
    \hline
         33 & 0.7071 & 6.09e-02 & $\star$  & 5.60e-02 & $\star$  & 1.93e-01 & $\star$ &  3.47e-18 & 1.04e-17\\
    145 & 0.3536 & 3.56e-02 & 0.774 & 3.26e-02 & 0.779 & 1.01e-01 & 0.925 & 6.94e-18 & 2.78e-17\\
    609 & 0.1768 & 1.82e-02 & 0.969 & 1.64e-02 & 0.994 & 5.27e-02 & 0.945 & 2.43e-17 & 2.78e-17\\
   2497 & 0.0884 & 9.05e-03 & 1.007 & 8.04e-03 & 1.028 & 2.72e-02 & 0.952 & 1.38e-16 & 6.94e-17\\
  10113 & 0.0442 & 4.50e-03 & 1.008 & 3.97e-03 & 1.019 & 1.38e-02 & 0.975 & 1.08e-15 & 2.06e-16\\
  40705 & 0.0221 & 2.25e-03 & 1.003 & 1.97e-03 & 1.007 & 6.97e-03 & 0.989 & 1.80e-15 & 4.02e-16\\
 163329 & 0.0110 & 1.12e-03 & 1.001 & 9.86e-04 & 1.002 & 3.50e-03 & 0.995 & 1.49e-13 & 8.33e-16\\
 \hline
    \multicolumn{10}{|c|}{$\mathbf{CR}-\mathbb{P}_0-\mathbb{P}_0$ scheme with variational crime, with kinematic viscosity $\nu=1$}\\
    \hline
     33 & 0.7071 & 5.59e-02 & $\star$  & 5.38e-02 & $\star$  & 1.71e-01 & $\star$ & 1.73e-18 & 4.34e-18\\
    145 & 0.3536 & 3.43e-02 & 0.705 & 3.30e-02 & 0.705 & 9.31e-02 & 0.873 & 6.94e-18 & 2.78e-17\\
    609 & 0.1768 & 1.75e-02 & 0.971 & 1.66e-02 & 0.995 & 4.91e-02 & 0.923 & 2.78e-17 & 4.16e-17\\
   2497 & 0.0884 & 8.66e-03 & 1.015 & 8.08e-03 & 1.036 & 2.57e-02 & 0.932 & 6.77e-17 & 8.33e-17\\
  10113 & 0.0442 & 4.30e-03 & 1.011 & 3.97e-03 & 1.023 & 1.32e-02 & 0.964 & 1.19e-15 & 2.26e-16\\
  40705 & 0.0221 & 2.14e-03 & 1.003 & 1.98e-03 & 1.008 & 6.67e-03 & 0.984 & 4.77e-15 & 4.16e-16\\
 163329 & 0.0110 & 1.07e-03 & 1.001 & 9.86e-04 & 1.002 & 3.35e-03 & 0.993 & 5.55e-15 & 1.05e-15\\
    \hline
    \multicolumn{10}{|c|}{standard $\mathbf{CR}-\mathbb{P}_0-\mathbb{P}_0$ scheme, with kinematic viscosity $\nu=10^{-4}$}\\
    \hline
        33 & 0.7071 & 3.85e-02 & $\star$ & 1.31e-03 &  $\star$ & 1.81e-01 &  $\star$ & 1.39e-17 & 8.67e-19\\
    145 & 0.3536 & 3.78e-02 & 0.029 & 4.25e-03 & -1.699 & 9.30e-02 & 0.958 & 2.78e-17 & 1.73e-18\\
    609 & 0.1768 & 1.33e-02 & 1.503 & 2.43e-03 & 0.806 & 4.72e-02 & 0.977 & 8.33e-17 & 1.73e-18\\
   2497 & 0.0884 & 5.68e-03 & 1.230 & 1.37e-03 & 0.823 & 2.37e-02 & 0.994 & 8.33e-17 & 1.30e-18\\
  10113 & 0.0442 & 2.56e-03 & 1.151 & 6.28e-04 & 1.129 & 1.19e-02 & 0.999 & 1.18e-16 & 1.30e-18\\
  40705 & 0.0221 & 1.21e-03 & 1.083 & 2.00e-04 & 1.653 & 5.94e-03 & 1.000 & 4.02e-16 & 3.79e-18\\
 163329 & 0.0110 & 5.89e-04 & 1.034 & 5.50e-05 & 1.860 & 2.97e-03 & 1.000 & 9.99e-16 & 8.24e-18\\
   \hline
    \multicolumn{10}{|c|}{$\mathbf{CR}-\mathbb{P}_0-\mathbb{P}_0$ scheme with variational crime, with kinematic viscosity $\nu=10^{-4}$}\\
    \hline
     33 & 0.7071 & 5.79e-03 &  $\star$ & 5.89e-04 &  $\star$ & 1.70e-01 &  $\star$ & 1.73e-18 & 2.71e-20\\
    145 & 0.3536 & 4.44e-03 & 0.384 & 4.31e-04 & 0.449 & 9.26e-02 & 0.880 & 6.94e-18 & 1.08e-19\\
    609 & 0.1768 & 1.20e-03 & 1.883 & 1.86e-04 & 1.216 & 4.72e-02 & 0.972 & 2.26e-17 & 4.34e-19\\
   2497 & 0.0884 & 3.48e-04 & 1.789 & 8.82e-05 & 1.073 & 2.37e-02 & 0.993 & 6.42e-17 & 8.67e-19\\
  10113 & 0.0442 & 1.03e-04 & 1.758 & 4.31e-05 & 1.033 & 1.19e-02 & 0.998 & 1.11e-16 & 1.36e-18\\
  40705 & 0.0221 & 3.39e-05 & 1.602 & 2.12e-05 & 1.023 & 5.94e-03 & 1.000 & 2.22e-16 & 2.87e-18\\
 163329 & 0.0110 & 1.35e-05 & 1.329 & 1.05e-05 & 1.016 & 2.97e-03 & 1.000 & 4.88e-15 & 9.81e-18\\
    \hline
\end{tabular}
    \caption{Error decay with respect to smooth manufactured solutions in 2D for two methods, and a large and a small viscosity. This illustrates the (Bernoulli) pressure robustness property, evidenced more clearly in the third and fourth blocks of the table for \eqref{eq:weak-formulation-h}.}
    \label{table:cv1}
\end{table}

We generate successively refined simplicial grids and compute errors between the approximate and exact solutions on each refinement level. For uniform mesh refinement, the experimental convergence order is computed as
	\[\texttt{rate} = \frac{\log(e_i(\bullet))- \log(e_{i+1}(\bullet))}{\log(h_i)- \log(h_{i+1})}\]
with $e_i(\bullet)$ denoting the error associated with the quantity $\bullet$ in its natural norm and  $h_i$  the mesh~size corresponding to a refinement level $i$. 
 The error history (errors and experimental convergence rates)  shown in Table \ref{table:cv1}  confirms the optimal convergence of the non-conforming scheme with and without the variational crime approach, for all variables in their respective norms. In particular we confirm that for the larger viscosity the methods deliver comparable results. On the other hand, for the smaller viscosity, with the modified scheme the velocity error (measured in the broken norm) is almost two orders of magnitude smaller than that obtained with the standard scheme (the latter violates the  invariance condition discussed in \cite{linke2017optimal} since the applied gradient force on the right-hand side induces an incompatible velocity field). In the table we also report on the local kernel characterising properties $\div_h\bu_h = 0$ and $\bcurl_h\bu_h - \bomega_h =\boldsymbol{0}$, shown by projecting onto the pressure and vorticity spaces these residuals, and taking the $\ell^\infty$ norm of the corresponding vector of degrees of freedom. We obtain a machine precision accuracy for all  cases. The approximate solutions computed with the modified scheme using the second last mesh refinement and the smaller viscosity are shown in Figure~\ref{fig:ex1}, indicating well-resolved patterns. 
 
We note that for this relatively small problem we simply use a direct method for each tangent system in the linearisation process, and confirm that, in order to  reach the stopping criterion on each refinement level, at most four and two Newton--Raphson iterations were needed for the small and large viscosity cases, respectively.

\begin{figure}[t!]
    \centering
    \includegraphics[width=0.325\linewidth]{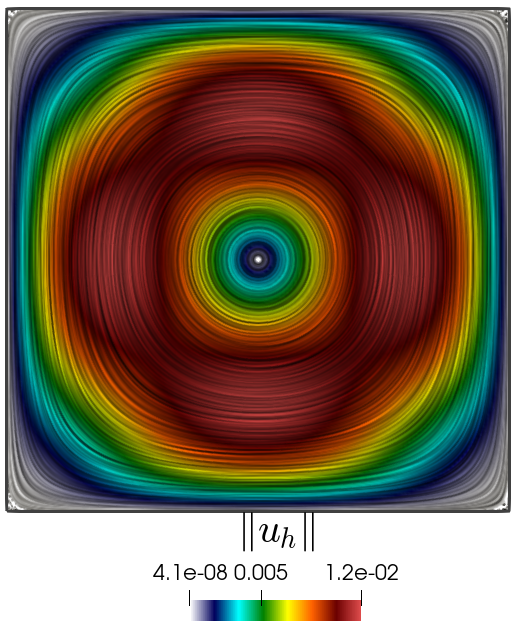}   \includegraphics[width=0.325\linewidth]{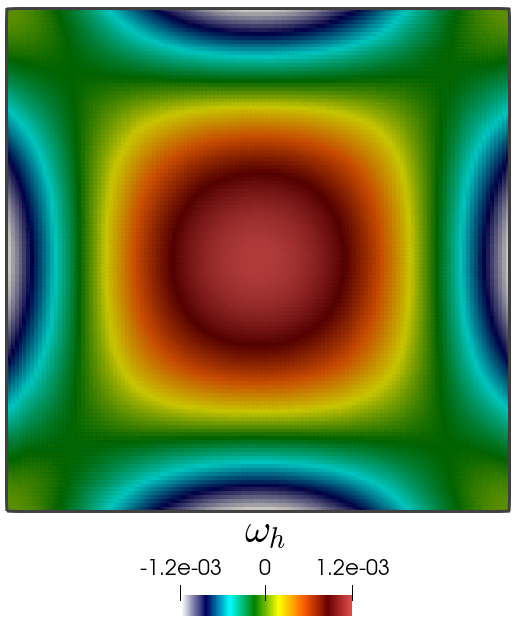}   \includegraphics[width=0.325\linewidth]{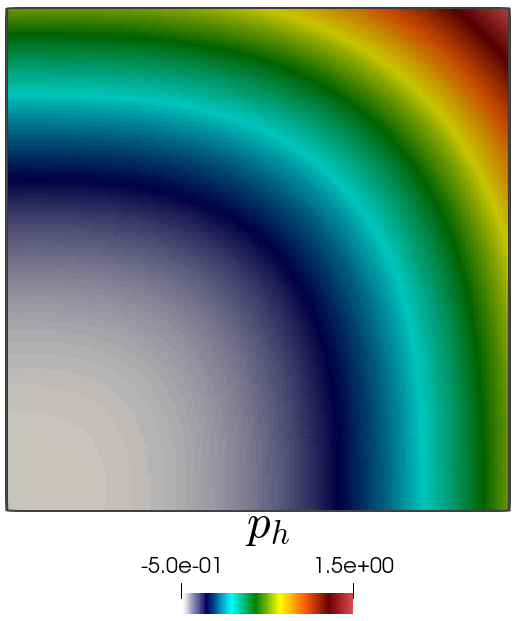}
    \caption{Approximate velocity (line integral contours and magnitude), vorticity, and Bernoulli pressure profiles computed with the modified $\mathbf{CR}-\mathbb{P}_0-\mathbb{P}_0$ scheme, with kinematic viscosity $\nu=10^{-4}$.}
    \label{fig:ex1}
\end{figure}

In order to illustrate the need for the {stabilisation parameter $\vartheta$} as discussed in Section~\ref{sec:discrete-problem}, we run again -- only for four coarse mesh refinement levels -- the test reported in the fourth block of Table~\ref{table:cv1} ($\mathbf{CR}-\mathbb{P}_0-\mathbb{P}_0$ modified scheme, with kinematic viscosity $\nu=10^{-4}$) and compare that block with the results in Table~\ref{table:cv2}, that use $\vartheta\in \{0,0.01,1\}$. {These results determine the value of the parameter that we consider in the subsequent tests. We note that} without stabilisation the method does not converge in the velocity nor vorticity, whereas for not big enough stabilisation one observes a slightly suboptimal convergence in the velocity field (as well as a larger Newton iteration count). {This seems to be qualitatively consistent with the general behaviour of stabilised Crouzeix--Raviart elements for e.g. elasticity \cite{hansbo2003discontinuous}}. Again as in Table~\ref{table:cv1}, the energy error in the modified method is not affected by the relatively larger Bernoulli pressure error.

\begin{table}[t!]
\setlength{\tabcolsep}{5pt}
    \centering
 \begin{tabular}{|r|c|c|c|c|c|c|c|c|c|}
    \hline
DoFs & $h$  & $ \|\bu-\bu_h\|_h$%texttt{e}_{\bu}$ 
& rate & $\|\bomega-\bomega_h\|_{0,\Omega}$%\texttt{e}_{\bomega}$ 
& rate & $\|p-p_h\|_{0,\Omega}$
%\texttt{e}_p$ 
& rate & $\texttt{loss}_{\mathrm{div}}$ & $\texttt{loss}_{\mathbf{curl}}$  \\
\hline
    \multicolumn{10}{|c|}{$\vartheta=0$}\\
    \hline
   33 & 0.7071 & 1.94e-02 & $\star$ & 5.67e-04 & $\star$ & 1.70e-01 & $\star$ & 6.94e-18 & 4.82e-20\\
    145 & 0.3536 & 5.81e-02 & -1.580 & 5.72e-04 & -0.012 & 9.26e-02 & 0.880 & 1.39e-17 & 2.22e-19\\
    609 & 0.1768 & 1.40e-01 & -1.273 & 5.72e-04 & 0.000 & 4.72e-02 & 0.972 & 5.55e-17 & 5.52e-19\\
   2497 & 0.0884 & 2.96e-01 & -1.077 & 5.72e-04 & 0.000 & 2.37e-02 & 0.993 & 2.22e-16 & 1.29e-18\\
  10113 & 0.0442 & 6.01e-01 & -1.020 & 5.71e-04 & 0.001 & 1.19e-02 & 0.998 & 4.44e-16 & 3.38e-18\\
\hline
    \multicolumn{10}{|c|}{$\vartheta=0.01$}\\
    \hline
           33 & 0.7071 & 1.76e-02 & $\star$ & 5.84e-04 & $\star$ & 1.70e-01 & $\star$ & 2.17e-19 & 2.37e-20\\
    145 & 0.3536 & 3.68e-02 & -1.067 & 5.61e-04 & 0.057 & 9.26e-02 & 0.880 & 2.08e-17 & 1.08e-19\\
    609 & 0.1768 & 3.56e-02 & 0.048 & 3.88e-04 & 0.534 & 4.72e-02 & 0.972 & 2.78e-17 & 3.79e-19\\
   2497 & 0.0884 & 2.21e-02 & 0.690 & 2.32e-04 & 0.742 & 2.37e-02 & 0.993 & 5.55e-17 & 5.42e-19\\
  10113 & 0.0442 & 1.18e-02 & 0.904 & 1.24e-04 & 0.902 & 1.19e-02 & 0.998 & 1.11e-16 & 1.63e-18\\
\hline
    \multicolumn{10}{|c|}{$\vartheta=1$}\\
    \hline
      33 & 0.7071 & 5.93e-03 & $\star$ & 6.46e-04 & $\star$ & 1.70e-01 & $\star$ & 6.94e-18 & 1.08e-19\\
    145 & 0.3536 & 3.97e-03 & 0.579 & 3.76e-04 & 0.780 & 9.26e-02 & 0.880 & 8.67e-18 & 2.17e-19\\
    609 & 0.1768 & 9.68e-04 & 2.037 & 1.70e-04 & 1.148 & 4.72e-02 & 0.972 & 2.78e-17 & 4.34e-19\\
   2497 & 0.0884 & 3.29e-04 & 1.555 & 8.42e-05 & 1.010 & 2.37e-02 & 0.993 & 5.55e-17 & 5.42e-19\\
  10113 & 0.0442 & 1.41e-04 & 1.222 & 4.20e-05 & 1.004 & 1.19e-02 & 0.998 & 1.73e-16 & 1.68e-18\\
    \hline
\end{tabular}
    \caption{Error decay with respect to mesh refinement (only coarser levels are shown) using smooth manufactured solutions in 2D for the modified $\mathbf{CR}-\mathbb{P}_0-\mathbb{P}_0$ scheme, with small kinematic viscosity $\nu=10^{-4}$ and three values of the penalisation parameter (compare also with the fourth block in Table~\ref{table:cv1}).}
    \label{table:cv2}
\end{table}

We also showcase the convergence properties in 3D by considering the following smooth manufactured solutions in $\Omega = (0,1)^3$
\[ \bu = \begin{pmatrix}
    \sin(\pi x)\cos(\pi y)\cos(\pi z)\\ -2\cos(\pi x)\sin(\pi y)\cos(\pi z)\\ \cos(\pi x)\cos(\pi y)\sin(\pi z)
\end{pmatrix}, \quad \bomega = \sqrt{\nu}\,\bcurl\,\bu, \quad p = \sin(\pi x)\sin(\pi y)\sin(\pi z)\]
and take model and stabilisation parameters as follows $\nu=0.01$, $\kappa = 100$, $\rF = 10$, $\vartheta = 1$. The error history is presented in Table~\ref{table:cv3}, where the results are consistent with the 2D case: they confirm the optimal order of convergence in the three unknowns, and give evidence of the expected conservation properties. For these runs the Newton--Raphson algorithm has taken up to six iterations to converge on each mesh refinement. The components of the approximate solution are portrayed in Figure~\ref{fig:ex13d}. 

\begin{table}[t!]
\setlength{\tabcolsep}{5pt}
    \centering
 \begin{tabular}{|r|c|c|c|c|c|c|c|c|c|}
    \hline
DoFs & $h$  & $ \|\bu-\bu_h\|_h$%texttt{e}_{\bu}$ 
& rate & $\|\bomega-\bomega_h\|_{0,\Omega}$%\texttt{e}_{\bomega}$ 
& rate & $\|p-p_h\|_{0,\Omega}$
%\texttt{e}_p$ 
& rate & $\texttt{loss}_{\mathrm{div}}$ & $\texttt{loss}_{\mathbf{curl}}$  \\
\hline 
  43 & 0.6124 & 1.94e+00 & $\star$ & 3.42e-01 & $\star$ & 3.15e-01 & $\star$ & 4.44e-16 & 5.55e-17 \\
    409 & 0.3062 & 1.33e+00 & 0.417 & 2.59e-01 & 0.402 & 2.24e-01 & 0.495 & 1.78e-15 & 2.22e-16 \\
   3553 & 0.1531 & 9.83e-01 & 0.432 & 1.64e-01 & 0.663 & 1.13e-01 & 0.987 & 4.44e-15 & 5.83e-16 \\
  29569 & 0.0765 & 5.16e-01 & 0.931 & 8.68e-02 & 0.915 & 5.21e-02 & 1.114 & 1.34e-14 & 1.78e-15 \\
 241153 & 0.0383 & 2.57e-01 & 1.004 & 4.39e-02 & 0.982 & 2.52e-02 & 1.046 & 3.60e-14 & 4.22e-15 \\
1947649 & 0.0191 & 1.28e-01 & 1.002 & 2.21e-02 & 0.994 & 1.25e-02 & 1.015 & 7.90e-14 & 8.88e-15 \\
     \hline
\end{tabular}
    \caption{Error decay with respect to mesh refinement   using smooth manufactured solutions in 3D for the modified (with variational crime) $\mathbf{CR}-\mathbb{P}_0-\mathbb{P}_0$ scheme, with   $\nu=0.01$, $\kappa = 100$, $\rF = 10$, $\vartheta = 1$.}
    \label{table:cv3}
\end{table}

\begin{figure}[t!]
    \centering
    \includegraphics[width=0.325\linewidth]{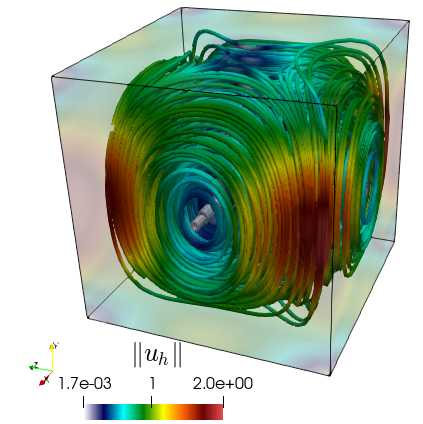}   \includegraphics[width=0.325\linewidth]{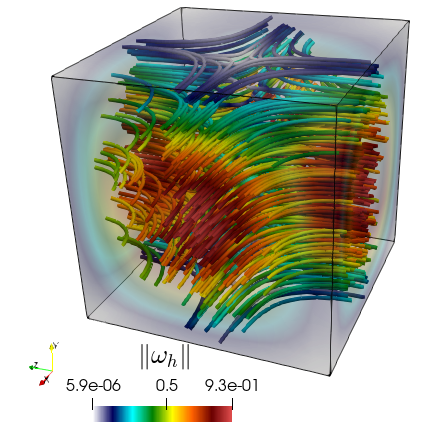}   \includegraphics[width=0.325\linewidth]{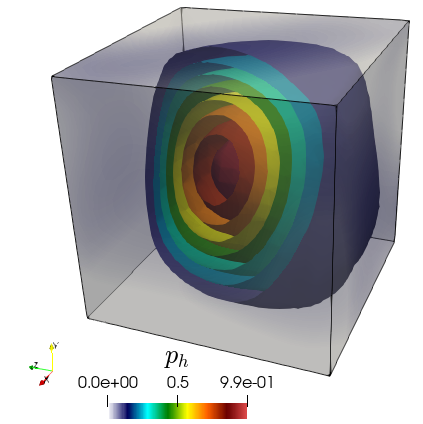}
    \caption{Approximate velocity (streamlines), vorticity (streamlines), and Bernoulli pressure profiles on {a mesh with $h=0.0383$}, computed with the modified $\mathbf{CR}-\mathbb{P}_0-\mathbb{P}_0$ scheme ({representing 241153 DoFs}), with kinematic viscosity $\nu=1$.}
    \label{fig:ex13d}
\end{figure}

\paragraph{Testing the robustness of the posteriori error estimator with non-smooth solutions.}
Next, we assess the convergence of the method when approximating a non-smooth exact solution under uniform mesh refinement, and check how the adaptive mesh refinement guided by the a posteriori error estimator defined in \eqref{eq:local-estimator} is able to restore optimal convergence. For this we take unity parameters $\nu = \rF = \kappa =1$, and use the following manufactured solutions defined on the L-shaped domain $\Omega = (-1,1)^2\setminus([0,1)\times(-1,0])$ (see, e.g., \cite{carstensen2005posteriori,verfurth1989posteriori}):
\begin{gather*}
    \bu = r^\lambda \begin{pmatrix}(1+\lambda)\sin(\theta)\psi(\theta) + \cos(\theta) \psi'(\theta)\\ \sin(\theta)\psi'(\theta) - (1+\lambda)\cos(\theta)\psi(\theta)
            \end{pmatrix}, \quad \bomega = \sqrt{\nu}\bcurl \bu, \quad    
    p = -\nu \frac{r^{\lambda -1 }}{1-\lambda}((1+\lambda)^2\psi'(\theta) + \psi'''(\theta))
\end{gather*}
in polar coordinates centred at the origin $(r,\theta) \in (0,\infty)\times (0,\frac{3\pi}{2})$, where 
\[ \psi(\theta) = \frac{\sin((1+\lambda)\theta)\cos(\lambda w)}{1+\lambda} - \cos ((1+\lambda) \theta) - \frac{\sin((1-\lambda)\theta)\cos(\lambda w)}{1-\lambda} + \cos ((1-\lambda) \theta).\]
Here $\lambda = \frac{856399}{1572864} \approx 0.5444837$ is the smallest positive solution of $\sin(\lambda w) + \lambda \sin (w) = 0$, and we take $w = \frac{3\pi}{2}$.  Second-order derivatives for velocity (and first order derivatives for pressure and vorticity) are not square integrable, and therefore these solutions do not have higher regularity ($\bu \notin \bH^1(\Omega), \omega, p\notin H^1(\Omega)$). Nevertheless, the exact boundary velocity is zero on the reentrant edges (at $\theta = 0$ and $\theta = \frac{3\pi}{2}$) and so the boundary data oscillation can be considered of high order. Note that the exact velocity and pressure above are such that $\sqrt{\nu}\bcurl \bomega + \nabla p = \boldsymbol{0}$ and so $\boldsymbol{f} = \kappa^{-1}\bu + \nu^{-1/2}\bomega \times \bu + \rF|\bu|\bu$. For this test the stabilisation parameter is taken as $\vartheta = 10$.

\begin{figure}[t!]
    \centering
    \includegraphics[width=0.325\linewidth]{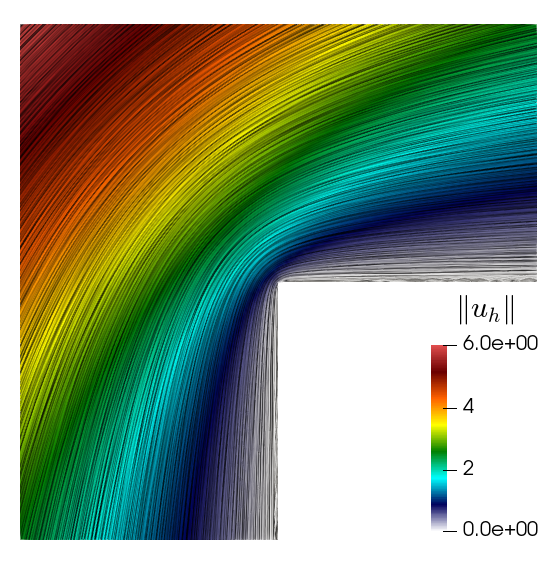}
    \includegraphics[width=0.325\linewidth]{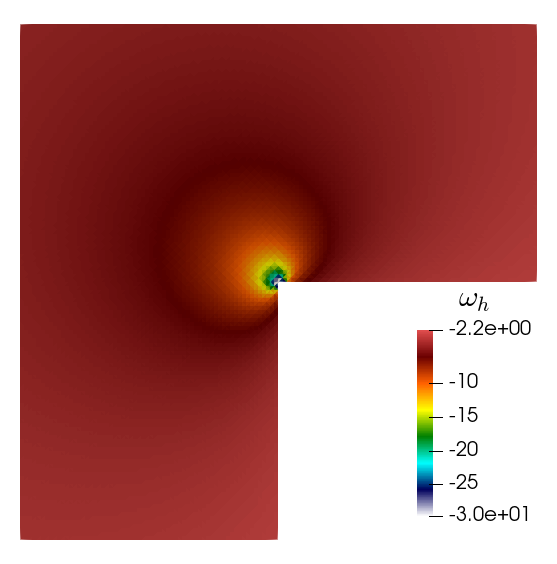}
    \includegraphics[width=0.325\linewidth]{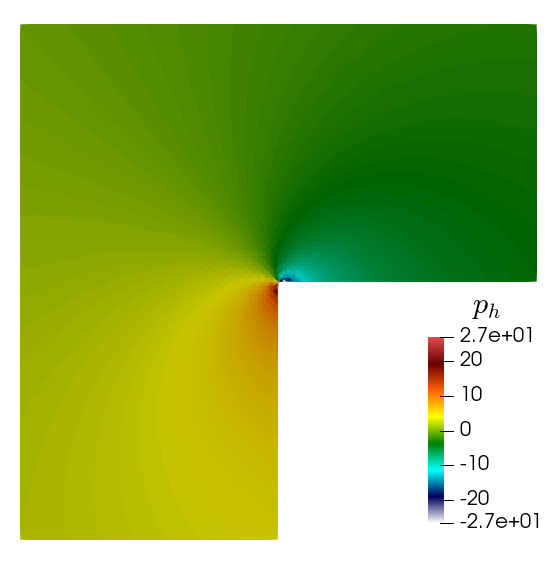}\\
    \includegraphics[width=0.325\linewidth]{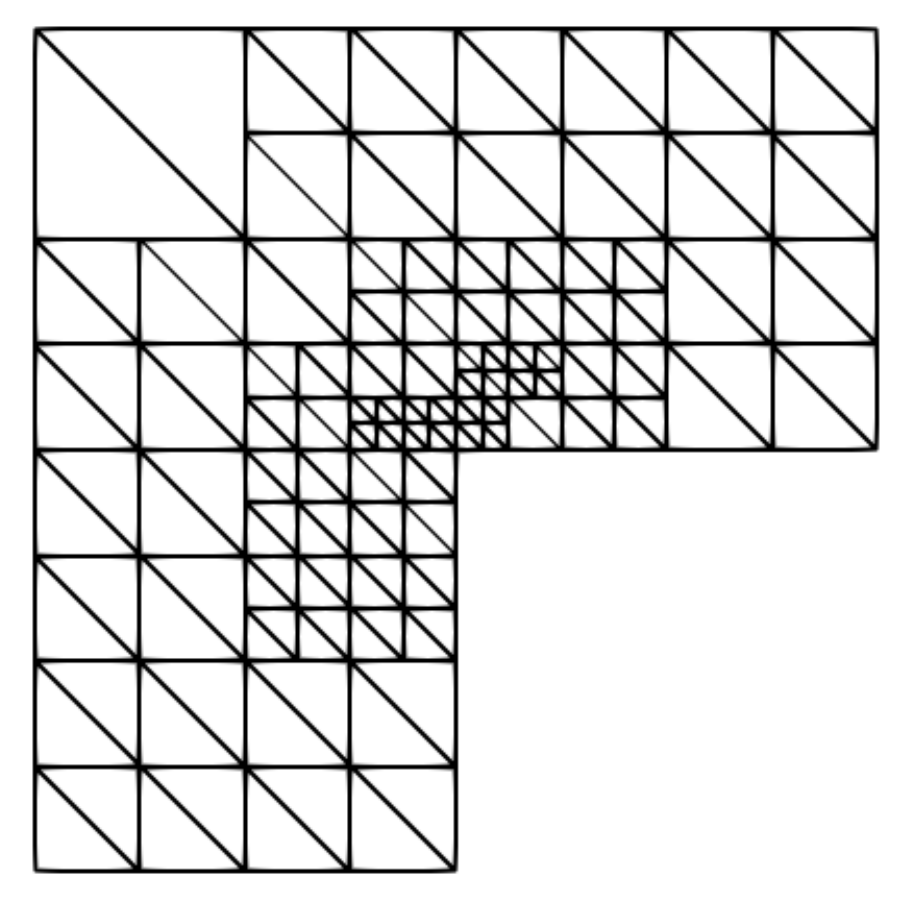}
    \includegraphics[width=0.325\linewidth]{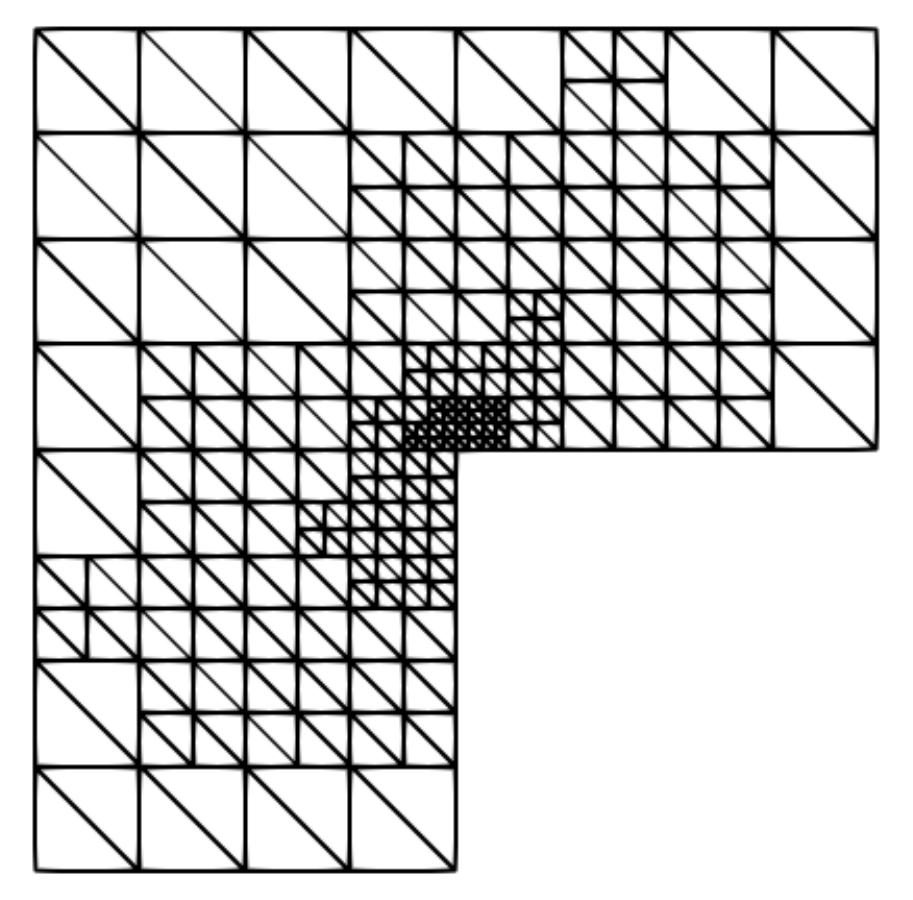}
    \includegraphics[width=0.325\linewidth]{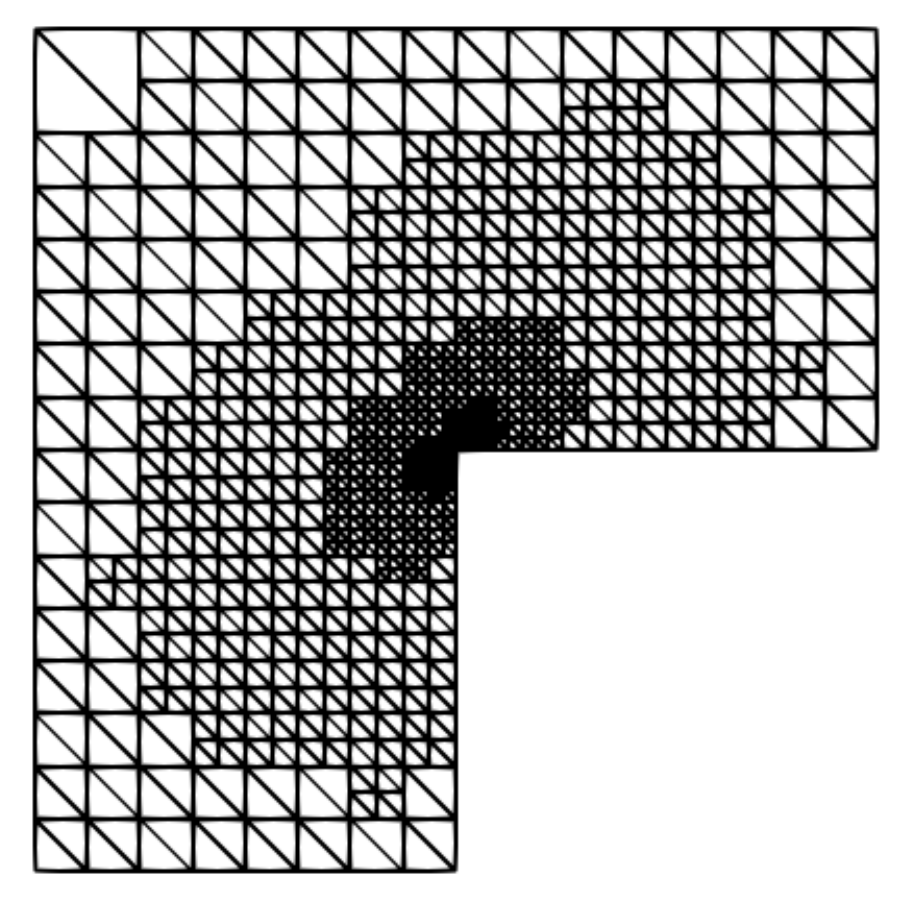}%\\
    \caption{Sample of approximate solutions for the convergence test on an L-shaped domain (top rows), and coarse meshes produced after three, six, and nine steps of the adaptive refinement algorithm guided by the a posteriori error estimator.}
%    , and convergence history for the cases of uniform (dashed lines) and adaptive mesh refinement (solid lines).}
    \label{fig:aposte}
\end{figure}

We show in Figure~\ref{fig:aposte} the approximate solutions as well as a sample of adaptively refined meshes generated by Algorithm~\ref{alg:1} (where it is evident that the solution singularity induces a  concentration of elements near the reentrant corner). In Table~\ref{table:cv-adapt} we display  the difference in error decay between the uniform and adaptive mesh refinement as well as the effectivity index 
\[\texttt{eff}(\eta) = \frac{\|\bu-\bu_h\|_h + \|\bomega-\bomega_h\|_{0,\Omega} + \|p-p_h\|_{0,\Omega}}{\eta(\cT_h)}\] 
in both cases. For adaptive mesh refinement, the experimental convergence order is computed as
	\[\texttt{rate} = \frac{\log(e_i(\bullet)){/} \log(e_{i+1}(\bullet))}{-\frac12[\log(\texttt{DoF}_i)/\log(\texttt{DoF}_{i+1})]}\]
with $e_i(\bullet)$   the error associated with the quantity $\bullet$ in its natural norm and  $\texttt{DoF}_i$ the total number of degrees of freedom corresponding to a refinement level $i$. 
We used a coarse mesh $\mathcal{C}_h$ with three quadtrees to mesh the L-shaped domain, no initial uniform refinements, and $\theta^{\rm r}=0.275$ and $\theta^{\rm c}=0.0$ in Algorithm~\ref{alg:1}. We can observe that the error under uniform mesh refinement goes to zero with a suboptimal rate (the expected $O(h^\lambda)$) while for the adaptive case we recover optimal linear convergence, necessitating much fewer degrees of freedom to achieve the same error. In addition, from the last column of the table we can see that the effectivity index remains bounded between 1.6 and 1.9 for all cases. Irrespective of the refinement strategy and refinement level, the Newton--Raphson scheme has taken no more than four iterations to reach the prescribed tolerance of $10^{-8}$. 

\begin{table}[t!]
\setlength{\tabcolsep}{5pt}
    \centering
 \begin{tabular}{|r|c|c|c|c|c|c|c|c|c|c|}
    \hline
DoFs & $h$  & $ \|\bu-\bu_h\|_h$%texttt{e}_{\bu}$ 
& rate & $\|\bomega-\bomega_h\|_{0,\Omega}$%\texttt{e}_{\bomega}$ 
& rate & $\|p-p_h\|_{0,\Omega}$
%\texttt{e}_p$ 
& rate & $\texttt{loss}_{\mathrm{div}}$ & $\texttt{loss}_{\mathbf{curl}}$ &
$\texttt{eff}(\eta)$\\
\hline
    \multicolumn{11}{|c|}{uniform mesh refinement}\\
    \hline
         23 & 0.7071 & 2.64e+0 & $\star$ & 2.50e+0 & $\star$ & 2.57e+0 & $\star$ & 6.21e-11 & 8.88e-16 & 1.635 \\
    105 & 0.3536 & 1.91e+0 & 0.467 & 1.80e+0 & 0.473 & 2.00e+0 & 0.366 & 2.12e-11 & 1.78e-15 & 1.826 \\
    449 & 0.1768 & 1.36e+0 & 0.489 & 1.27e+0 & 0.496 & 1.69e+0 & 0.241 & 7.28e-12 & 4.44e-15 & 1.810 \\
   1857 & 0.0884 & 9.51e-01 & 0.516 & 8.91e-01 & 0.517 & 1.23e+0 & 0.459 & 2.50e-12 & 8.88e-15 & 1.797 \\
   7553 & 0.0442 & 6.59e-01 & 0.529 & 6.17e-01 & 0.530 & 8.61e-01 & 0.514 & 8.81e-13 & 3.73e-14 & 1.790 \\
  30465 & 0.0221 & 4.54e-01 & 0.537 & 4.25e-01 & 0.537 & 5.96e-01 & 0.530 & 3.41e-13 & 8.26e-14 & 1.786 \\
 122369 & 0.0110 & 3.12e-01 & 0.540 & 2.92e-01 & 0.541 & 4.11e-01 & 0.537 & 2.27e-13 & 1.55e-13 & 1.784 \\
\hline
    \multicolumn{11}{|c|}{adaptive mesh refinement}\\
    \hline
       23 & 0.7071 & 2.64e+0 & $\star$ & 2.50e+0 & $\star$ & 2.57e+0 & $\star$ & 6.21e-11 & 8.88e-16 & 1.635 \\
     79 & 0.3536 & 2.18e+0 & 0.305 & 2.05e+0 & 0.317 & 2.14e+0 & 0.155 & 1.27e-06 & 8.88e-16 & 1.876 \\
    169 & 0.1768 & 1.64e+0 & 0.757 & 1.53e+0 & 0.768 & 1.86e+0 & 0.609 & 2.12e-11 & 3.55e-15 & 1.811 \\
    405 & 0.0884 & 1.25e+0 & 0.625 & 1.16e+0 & 0.646 & 1.62e+0 & 0.212 & 7.42e-09 & 3.55e-15 & 1.782 \\
    921 & 0.0442 & 8.94e-01 & 0.808 & 8.26e-01 & 0.817 & 1.31e+0 & 0.572 & 2.50e-12 & 7.11e-15 & 1.759 \\
   1891 & 0.0221 & 6.54e-01 & 0.871 & 6.00e-01 & 0.886 & 1.03e+0 & 0.726 & 8.63e-13 & 1.24e-14 & 1.726 \\
   3833 & 0.0110 & 4.72e-01 & 0.920 & 4.33e-01 & 0.925 & 7.72e-01 & 0.815 & 3.13e-13 & 1.78e-14 & 1.698 \\
   7771 & 0.0055 & 3.37e-01 & 0.960 & 3.07e-01 & 0.973 & 5.72e-01 & 0.848 & 1.21e-13 & 2.84e-14 & 1.671 \\
  15153 & 0.0028 & 2.39e-01 & 1.020 & 2.18e-01 & 1.032 & 4.22e-01 & 0.911 & 6.39e-14 & 4.26e-14 & 1.649 \\
  29121 & 0.0014 & 1.72e-01 & 1.011 & 1.56e-01 & 1.024 & 3.13e-01 & 0.919 & 5.68e-14 & 7.82e-14 & 1.625 \\
    \hline
\end{tabular}
    \caption{Error decay, divergence and curl loss, and effectivity index tabulated with respect to mesh refinement using non-smooth manufactured solutions in the L-shaped domain for the modified $\mathbf{CR}-\mathbb{P}_0-\mathbb{P}_0$ scheme, and using uniform vs. adaptive mesh refinement.}
    \label{table:cv-adapt}
\end{table}

\section{Concluding remarks}\label{sec:concl}
{We have developed and analysed a novel non-augmented vorticity–velocity–Bernoulli formulation for incompressible flow in highly permeable porous media, governed by the Navier--Stokes--Brinkman--Forchheimer equations. The approach avoids least-squares or stabilisation-based augmentation by formulating the problem as a nested saddle-point system and working in divergence-free velocity spaces. This framework enables both continuous and discrete solvability analyses under small data assumptions, relying on fixed-point arguments and inf-sup conditions in divergence-free subspaces. At the discrete level, we employ a Crouzeix--Raviart finite element method enhanced with tangential and normal jump penalisation to control the full $\mathbf{H}(\bcurl) \cap \mathbf{H}(\div)$ velocity norm, and we ensure pressure-robustness through a modified variational form using a Raviart--Thomas interpolant. We introduced a fully computable residual-based a posteriori error estimator, which is shown to be reliable and efficient---the latter hinging on a new inverse inequality tailored to the Forchheimer nonlinearity. Numerical experiments validate all theoretical results and demonstrate that adaptive mesh refinement, driven by the proposed estimator, restores optimal convergence even in the presence of solution singularities. The practical implementation leverages a lightweight, parallel octree-based infrastructure capable of generating non-conforming adaptive meshes, where appropriate multipoint constraints allow the recovery of optimal approximation properties. The combined contributions in formulation, analysis, estimator design, and implementation  provide a foundation for further development of adaptive solvers for nonlinear porous flow models and illustrate the benefits of structure-preserving discretisations in complex multiphysics regimes. 

Possible directions of future work include the extension of the analysis to more general data regimes, particularly beyond the smallness assumptions required for fixed-point arguments, potentially through monotonicity or compactness-based techniques. It would also be valuable to investigate higher-order nonconforming discretisations that retain divergence-free properties and support pressure-robustness. Furthermore, we want to address the coupling with transport and reaction models within the same porous medium, where interface dynamics or multiphysics constraints may necessitate tailored trace operators and stabilisation. {We also mention that  the analysis of the non steady-state counterpart of this problem can be performed following the abstract Showalter theory. The range condition and monotonicity properties proceed quite similarly as in \cite{anaya23a}, as well as the semi-discrete solvability and the a priori analysis. On the other hand, the fully discrete a priori estimates and its a posteriori error analysis remains an open problem. Finally, it would be worth to explore automatic selection of local stabilisation parameters following the recent work  \cite{bringmann2024local}.}
}

\subsection*{Acknowledgement}
This work has been supported by  Monash Mathematics through a \textsc{Gordon Preston Sabbatical Fellowship}; by the Australian Research Council through the \textsc{Future Fellowship} grant FT220100496 and \textsc{Discovery Project} grant DP22010316; and by the National Research and Development Agency (ANID) of the Ministry of Science, Technology, Knowledge and Innovation of Chile through the postdoctoral program \textsc{Becas Chile} grant 74220026. Computational resources were provided by the Australian Government through NCI under the National Computational Merit Allocation Scheme (NCMAS) and the ANU Merit Allocation Scheme (ANUMAS).

%%%%%%%%%%%%%%%%%%%%%%%%%%%%%%%%%%%%%%%%%%%%%%%%%%%%%%%
%\small
\bibliographystyle{siam}
\bibliography{refs}

%%%%%%%%%%%%%%%%%%%%%%%%%%%%%%%%%%%%
\appendix
\section{A new  inverse estimate}\label{sec:inverse}

Let us first recall that the {scalar}  volume-bubble function  ${b_K}\in W^{1,\infty}_0({K})$ on a simplex ${K\subset \mathbb{R}^d}$ is the product
of all $d+1$ barycentric coordinates of the vertices of $K$ times a factor $(d+1)^{d+1}$ 
for the normalisation $0\le b_K\le 1=\max b_K$ in $K$. Note that typical efficiency estimates rely upon Verf\"urth's bubble-function methodology \cite{verfurth1989posteriori}. For the present case, the Forchheimer nonlinearity requires a non-trivial modification of this approach, presented in the following key inverse estimate.

\begin{lemma}[New inverse estimate]
There exists a universal positive constant $C_\mathrm{eq}$, that exclusively  depends on $d=2,3$, 
such that all affine vector-valued functions ${\f,\g\in \mathbb{P}_1(K)^d}$ satisfy
\[
\| |\f|\f+\g\|_{0,K}\le C_\mathrm{eq}\| b_K^{1/2}( |\f|\f+\g)\|_{0,K}.
\]
\end{lemma}

\noindent{\it Proof of independence of $K$}. At first glance, the most striking aspect of the lemma is the 
independence of $K$, that makes $C_\textrm{eq}$ a constant that depends on the reference simplex $K_\textrm{ref}$
of volume $|K_\textrm{ref}|>0$.
But that aspect follows from an elementary affine transformation $\Psi$ from $K$ to $K_\textrm{ref}$ {(which is smooth since we are assuming shape regularity)}: All integrals of both sides 
of the squared estimate are Lebesgue functions (and no derivative appears) so the transformation shows 
\[
\| |\f|\f+\g\|_{0,K}=\sqrt{ \frac{|K|}{|K_\textrm{ref}|}} \,  \| |\hat{\f}|\hat{\f}+\hat{\g}\|_{0,K_\textrm{ref}}
\quad\text{and}\quad
\| b^ {1/2}_K(|\f|\f+\g)\|_{0,K}=\sqrt{ \frac{|K|}{|K_\textrm{ref}|}} \, 
 \| b^ {1/2}_{K_\textrm{ref}} ( |\hat{\f}|\hat{\f}+\hat{\g})\|_{0,K_\textrm{ref}}
\]
for $\hat{\f}:=\f\circ \Psi $,  $\hat{\g}:=\g\circ \Psi $, $b_{K_\textrm{ref}}=b_K\circ \Psi$,
and $K_\textrm{ref}=\Psi(K)$. Thus, once the assertion holds on the reference simplex $K={K_\textrm{ref}}$, it holds
on any simplex $K$ with the same constant $C_\textrm{eq}$.  \hfill$\square$

\medskip 
\noindent{\it Proof in case $\g=\0$ or $\f\in \mathbb{P}_0(K)^d$}. 
If $\g=\0$ then only the term $|\f|\f $ with 
Euclid length  $|\f|^2$ 
arises and  $|\f|^2$   is a quadratic polynomial 
(for the affine $\f$). Therefore a classical inverse estimate provides $C_\textrm{eq,2}$ with
\[
\| q_2 \|_{0,K}\le C_\textrm{eq,2}\| b_K^ {1/2} q_2\|_{0,K}\quad\text{for all }
q_2\in \mathbb{P}_2(K)^d.
\]
This concludes the proof of the assertion for
 $\f\in \mathbb{P}_1(K)^d$, $\g=\0$, and  $ C_\textrm{eq}=C_\textrm{eq,2}$:
\[
\| |\f|\f\|_{0,K}=\| |\f|^2\|_{0,K}\le C_\textrm{eq,2}\| b_K^ {1/2} |\f|^2\|_{0,K}= C_\textrm{eq,2}\| b_K^ {1/2} |\f| \f\|_{0,K}.%\hspace{3cm}\square
\]
%\sbc{Should it be $C_\textrm{eq,2}$ instead of $C_\textrm{eq,1}$?}
In case $\f\in \mathbb{P}_0(K)^d$, $|\f|\f+\g$ is a polynomial and the above arguments verify
the assertion. {\hfill$\square$}

 \medskip 
  \noindent{\it Notation throughout the remaining parts of the proof}.
  The following notation on the affine function $\f=\ba+\bB\bxi$ for $\bxi=\bx-\textrm{mid}(K)$ 
  applies throughout the proof, {where $\textrm{mid}(K)$ denotes the barycentre of $K$}. So $\ba=\f(\textrm{mid}(K))$ and $\bB=\bD\f\in\mathbb{R}^{d\times d}$ has the maximal
  singular value $\sigma_1(\bB)\ge 0$ (and $\sigma_1(\bB)=0$ iff $\bB=\0$).
  
  The volume-bubble function $b_K$ with maximal value $1$ attained at $\bx=\textrm{mid}(K)$
  is continuous and vanishes outside $K$. The open set $\{ \bx\in \mathbb{R}^d: 1/2<b_K(\bx)\}$ lies compactly in the interior of 
  the simplex $K$ and  contains  its midpoint $\textrm{mid}(K)$. Consequently there exists some positive  $r^*<1$
  such that the open ball $\omega$ around  $\textrm{mid}(K)$ with radius $r^*$ is 
  contained in $\{ \bx\in \mathbb{R}^d: 1/2<b_K(\bx)\}$.
  Let $r^*$ denote the maximal radius with this property 
  and let $\omega$ denote the associated open ball. Notice that $r^*$ is a universal constant in the sense that it 
  exclusively depends on  the reference simplex $K=K_\textrm{ref}$ (since {once} we know the scaling invariance we can reduce the analysis to the reference element). Note that
  $\bx\in\Omega$ is equivalent to  $\bxi=\bx-\textrm{mid}(K)\in  \mathbb{R}^d$ 
  satisfies $|\bxi|<r^*$.
  
  \medskip
  
 \noindent{\it Proof of $\textrm{RHS}=0$ $\Leftrightarrow$  $\bD\f=\0$ and $\g=\bPi_0 |\f|\f$}. (Note that we are proving that $\textrm{RHS}=0$ $\Leftrightarrow$ $\f,\g \in \mathbb{P}_0(K)^d$ and $\f = \g$). 
 Let us denote by $\bPi_k$ the $L^2$ projection onto $\mathbb{P}_k(K)^d$. Suppose the upper bound in the assertion,
 called  $\textrm{RHS}$, vanishes   for the affine functions $\f$ and $\g$; i.e. $|\f|\f$ is affine in $\omega$. 
 If $\f(\bx)=\0$ for all $\bx\in\omega$ then $\bB=\bD \f=\0$.
 Otherwise  $|\f|$ is positive at some $\bx\in\omega$ and so (by continuity) also  in a small neighbourhood of $\bx$.
%Let $\bB=Df$ the constant derivative and let $b:=\bB\zeta$ for some arbitrary unit vector $\zeta$. 
We fix this point $\bx$ throughout this proof, select a 
direction $\bzeta\in\mathbb{R}^d$, and consider 
  \[
  g(t):=|\f(\bx+t\bzeta)|\, \f(\bx+t\bzeta)\cdot \f(\bx)
  \]
 as  a smooth function in the real parameter $t$ near zero. 
   Since $|\f|\f$ is affine in $\omega$, so is $g(t)$ affine in $t$ and its second derivative $g''(0)=0$
   vanishes. Abbreviate $\ba:=\f(\bx)\ne \0$, $\bb:=\bB\bzeta$, $\bB=\bD\f$,
   and $\f:= \f(\bx+t\bzeta)=\ba+t\bb$ and compute
   $g(t)=|\f| (\ba\cdot \f) $, $g'(t)= (\ba\cdot \f)\,  (\bb\cdot \f)/|\f| +|\f| \, (\ba\cdot \bb)$, and eventually 
   \[
   g''(0)=|\ba|\, |\bb|^2+(\ba\cdot \bb )^2/|\ba|.
  \]
  Since we have $\ba\ne \0$,  $g''(0)=0$ implies $\bb=\bB\bzeta=\0$. Recall $\bzeta$ is an arbitrary direction,
 $\bB\bzeta=\0$ follows for all $\bzeta\in \mathbb{R}^d$, whence $\bB=\0$. 
 The remaining details are straightforward  and 
 omitted for brevity.
   {\hfill$\square$}
     
\medskip

\noindent{\it Proof for  $\varepsilon |\f(\textrm{\rm mid}(K))|\le \sigma_1(\bD\f)$}. 
Recall that the extra condition means that for some small and positive $\varepsilon$, the  
parameters $\ba=\f(\textrm{mid}(K))$ and $\bB=\bD\f\in\mathbb{R}^{d\times d}$, 
that determine $\f\in \mathbb{P}_1(K)^d$, satisfy  $\varepsilon \alpha \le \beta$
for length $\alpha:=|\ba|$ and the first singular value $\beta:=\sigma_1(\bB)$.
We may exclude the case  $\beta=0$ (as then $0\le \alpha\le \beta/\varepsilon=0$ means $\f=\0$), 
when the assertion follows with $C_\textrm{eq,2}$ 
from classical inverse estimates mentioned above.
In the remaining  case $\beta>0$, the assertion can be rewritten with the quotient 
\[
Q(\f):=\max_{\g\in \mathbb{P}_1(K)^d} \| |\f|\f+\g\|_{0,K} / \| b_K^{1/2}( |\f|\f+\g)\|_{0,K} \le C_\textrm{eq}.
\]
The characterisation of  $\textrm{RHS}=0$ in the previous step reveals that $Q(\f)\ge 1$ 
(by $b_K\le 1$) is quotient of two positive terms (indeed, since $\beta>0$ then $\bB$ is not the zero matrix and $\f$ is not constant, whence the two terms in the quotient are positive). We rescale the data and observe $Q(\f)=Q(\f/t)$ for
any positive real $t$. The choice 
$t:=\alpha+\beta>0$ reveals that we may and will assume $\alpha+\beta=1$ without loss of generality 
in the sequel and  consider 
\[
\max_{ (\ba,\bB)\in \mathcal{C}(\varepsilon) } Q(\f) =:C_\textrm{eq}(\varepsilon)
\]
with the above notation $\f=\ba+\bB\bxi$  and the compact parameter set
\[
 \mathcal{C}(\varepsilon) =\{ (\ba,\bB)\in \mathbb{R}^{d\times (d+1)}: 
\; \alpha+\beta=1 \text{ and } \varepsilon\alpha\le\beta \text{ hold for } 
 \alpha= |\ba| \text{ and } \beta=\sigma_1(\bB)\}.
\]
Since  $\mathcal{C}(\varepsilon) $ is compact and does not include $\bB=\0$ (because 
$\varepsilon(1-\beta)=\varepsilon\alpha\le\beta$
implies $\beta\ge \varepsilon/(1+\varepsilon)$),
the quotient $Q(\f)=Q(\ba+\bB\bxi)$ depends continuously on $(\ba,\bB)$. Consequently, 
that maximum $C_\textrm{eq}(\varepsilon)<\infty$ of all $Q(\f)$ with $\f$ from 
the parameter set $\mathcal{C}(\varepsilon) $
is attained and in particular finite. (At this stage $C_\textrm{eq}(\varepsilon)<\infty$  
might monotonically depend on $\varepsilon>0$ with
 possibly $\mathcal{C}(\varepsilon) \to\infty $ as $\varepsilon\to 0$.)
 {\hfill$\square$}

\medskip

\noindent{\it Proof for  $ \sigma_1(\bD\f)<\varepsilon_0 |\f(\textrm{\rm mid}(K))|$}. A rescaling
shows that we may suppose a unit vector $ \f(\textrm{\rm mid}(K))=\ba$, $|\ba|=1$, and 
we rename $\bD\f$ as $t>0$ times $\bB$  for a matrix $\bB$ with largest singular value 
$\sigma_1(\bB)=1$. Hence the parameter regime translates into $0<t<\varepsilon_0$
for a small $\varepsilon_0<1/2$ defined below. The first ingredient in this proof is a Taylor expansion in the parameter $t$ of $|\f|\f$. Since $t$ is small and $|\ba|=1$, $|\f|$ is differentiable 
and allows for a power series expansion. From $\f=\ba+t\bb$ and $\bb=\bB\bxi$, we eventually infer  
$|\f|=1+ ( \ba\cdot \bb) t+ {t^2}\left( |\bb|^ 2 - (\ba\cdot \bb)^ 2 \right)/2 +O(t^3)$ and thereafter
\[
|\f|\f= \ba + t\left(\bb+(\ba\cdot \bb)\, \ba\right)+t^2\left( (|\bb|^2-(\ba\cdot \bb)^2 )\, \ba 
+ 2(\ba\cdot \bb)\, \bb\right)/2+\bdelta
\]
with some third-order remainder $\bdelta$. This expansion holds for all $\bx\in K$
and the difference $\bdelta\in L^\infty(K)^{d}$ satisfies $\|\bdelta\|_{L^\infty(K)^d}\le C_0 t^3$
for all $0<t<\varepsilon_0$ and $\varepsilon_0<1/2$
sufficiently small. After abbreviating the terms in the 
displayed expansion of $|\f|\f$ by $\ba\in \mathbb{P}_0(K)^d$, $\g_1:=\bb+(\ba\cdot \bb)\, \ba\in \mathbb{P}_1(K)^d$,
and $\g_2:= (|\bb|^2-(\ba\cdot \bb)^2 )\, \ba + 2(\ba\cdot \bb)\, \bb\in \mathbb{P}_2(K)^d$, we have
\begin{equation}\label{cceqeffstabproof1}
\| |\f|\f-\ba - t \g_1- t^2\g_2/2\|_{L^\infty(K)^d}\le  C_0 t^3.
\end{equation}
The second ingredient is the $b_K$ projection $P: L^2(K)^d\to  \mathbb{P}_1(K)^d$ defined,
%(uniquely) 
for any $\bq\in  L^2(K)^d$, by
\[
P\bq\in  \mathbb{P}_1(K)^d\quad\text{satisfies}\quad
\int_K b_K(\bq-P\bq)\cdot \bphi_1\, \dx=0 \quad\text{for all } \bphi_1\in  \mathbb{P}_1(K)^d.
\]
Notice that $P\in \mathcal{L}( L^2(K)^d) $ is a projection and the orthogonal projection
in the $b_K$ weighted
$L^2(K)^d$ scalar product $(\bullet,\bullet)_{b_K}:=(b_K\bullet,\bullet)_{K}$; $P$ is merely
an oblique projection in $L^2(K)^d$. Let $\| \bullet\|_{b_K}:=(\bullet,\bullet)_{b_K}^ {1/2}$
denote the induced  $b_K$ weighted $L^2(K)$ norm.
The stability 
$\| P\bq\|_{0,K}\le C_\textrm{eq,2} \|\bq\|_{b_K}$
is immediate from the above classical inverse estimate:
\(
C_\textrm{eq,2}^ {-2}  \| P\bq\|_{0,K}^2\le   \| P\bq\|_{b_K}^2=( P\bq,\bq)_{b_K}
\le  \| P\bq\|_{b_K}\| \bq\|_{b_K}
\)
with a Cauchy inequality in the Hilbert space $(L^2(K)^d, (\bullet,\bullet)_{b_K})$
in the last step, whence $ \| P\bq\|_{0,K}\le C_\textrm{eq,2} \| \bq\|_{b_K}$.

In the second step of the proof we establish  $\| (1-P)(|\f|\f) \|_{0,K}\le C_1 t^2$
for $0<t<\varepsilon_0$. A triangle inequality, the elementary estimate
$\|\bullet\|_{0,K}\le |K|^ {1/2}  \|\bullet\|_{L^\infty(K)}$ for the reference simplex $K=K_\textrm{ref}$ of volume $|K|\le 1$, and \eqref{cceqeffstabproof1} 
result in 
\[
\| (1-P)(|\f|\f) \|_{0,K}\le \| (1-P)(\ba - t \g_1- t^ 2\g_2/2)\|_{0,K}+ C_0 t^3.
\]
Since $(1-P)(\ba - t \g_1)=\0$, the above  stability $ \| P\bullet\|_{0,K}\le C_\textrm{eq,2} 
\| \bullet\|_{b_K}$ provides
\[
2t^{-2}  \| (1-P)(\ba - t \g_1- t^ 2\g_2/2)\|_{0,K}= \| (1-P)\g_2\|_{0,K}
 \le  (1+C_\textrm{eq,2})  \| \g_2\|_{0,K}.
 \]
 The point is that  $|\ba|=1=\sigma_1(\bB)$ is bounded and so is $\bb=\bB\bxi$ with $|\bxi|<1$
 (recall $K=K_\textrm{ref}$ is the reference simplex): Thus
 $|\bb|=|\bB\bxi|\le \sigma_1(\bB)=1$ and 
 eventually
 \[
 |\g_2|:= | (|\bb|^2-(\ba\cdot \bb)^2 )\, \ba + 2(\ba\cdot \bb)\, \bb|\le 
 | |\bb|^2-(\ba\cdot \bb)^2 |  + 2 |\ba\cdot \bb|
 \le 3\quad\textrm{ a.e. in } K.
 \]
The combination with the above estimates provides 
$C_1:=8 (1+C_\textrm{eq,2})+ C_0/2$ in 
\[
\| (1-P)(|\f|\f) \|_{0,K}\le  8 t^2(1+C_\textrm{eq,2})  + C_0 t^3 \le C_1 t^2.
\]
%for (by $t<\varepsilon_0<1/2$).
The third step of the proof establishes  $ C_2 t^2\le \| b_K^{1/2}(1-P)(|\f|\f) \|_{0,K}$.
Recall $b_K>1/2$ on the ball $\omega$ around $\textrm{\rm mid}(K)$ of radius $r^*$.
Since $\sigma_1(\bB)=1$ we find some unit vector $\bzeta$ with $|\bB\bzeta|=1$ and 
hence at $\bx:=\textrm{\rm mid}(K)+r^*\bzeta$ we have  $\bb=\bB\bxi$ for 
$\bxi=\bx-\textrm{\rm mid}(K)=r^*\bzeta$ and infer $|\bb|=r^*|\bB\bzeta|=r^*$. 
At other points in $\omega$, $|\bb|\le r^*$, for instance
 $\bb$ vanishes at $\textrm{\rm mid}(K)$. Altogether
 \(
 r^*=\| \bb \|_{L^\infty(\omega)^d}.
 \)
This and the calculation $\ba\cdot \g_2= |\bb|^2+(\ba\cdot \bb)^2$ imply
\[
(r^*)^2=\| \bb \|^2_{L^\infty(\omega)^d}\le \| \ba\cdot \g_2 \|_{L^\infty(\omega)}
\le 2 \min_{g_0\in\mathbb{R}}  \| \ba\cdot \g_2- g_0 \|_{L^\infty(\omega)}
\]
with optimal value $g_0=  \| \ba\cdot \g_2 \|_{L^\infty(\omega)}/2$ in the last step
(recall that $ \ba\cdot \g_2\ge 0$ and it vanishes at $\textrm{\rm mid}(K)$). 
The choice of $\overline{g}$ as the integral mean of $\ba\cdot \g_2$ over $\omega$
is therefore an upper bound,
\[
 \min_{g_0\in\mathbb{R}}  \| \ba\cdot \g_2- g_0 \|_{L^\infty(\omega)}
 \le  \| \ba\cdot \g_2-  \overline{g} \|_{L^\infty(\omega)}
 \le  C_\textrm{inv}  \| \ba\cdot \g_2-  \overline{g} \|_{L^2(\omega)}
\]
from another inverse estimate (from equivalence of norms in $\mathbb{P}_2(\Omega)$)
in the last step. The key insight starts  with the quadratic polynomial
$\ba\cdot \g_2-  \overline{g} $ with integral mean zero
over the symmetric (fixed) domain $\omega$ such that 
$\ba\cdot \g_2-  \overline{g}$ is
symmetric along all straight lines  through  $\textrm{\rm mid}(K)$, while $\bxi$ 
is asymmetric (with respect to the centre  $\textrm{\rm mid}(K)$). Thus the
product   $  \int_\omega (\ba\cdot \g_2- g_0)  \bxi \dx=\0 $ vanishes. This means  
 that $\ba\cdot \g_2- g_0$ has vanishing
moments up to degree one: $\ba\cdot \g_2- g_0$ is $L^2(\omega)$ perpendicular to
all affine functions. Consequently
\[
 \| \ba\cdot \g_2-  \overline{g} \|_{L^2(\omega)}
 =\min_{q_1\in \mathbb{P}_1(\omega)}  \| \ba\cdot \g_2- q_1 \|_{L^2(\omega)}
 \le 2  \min_{q_1\in \mathbb{P}_1(\omega)}  \| \ba\cdot \g_2- q_1 \|_{b_K}
 \le  2 \min_{\bq_1\in \mathbb{P}_1(\omega)^d}  \| \g_2- \bq_1 \|_{b_K}
\]
with $b_K>1/2$ a.e. in $\omega\subset K$ and $|\ba|= 1$ in the last steps.
Notice that the last term $ \min_{\bq_1\in \mathbb{P}_1(\omega)^d}  \| \g_2- \bq_1 \|_{b_K}$
is equal to  $\| (1-P)\g_2 \|_{b_K}$. The 
combination of the aforementioned inequalities reads
\[
\frac{ t^2 (r^*)^2}{8 C_\textrm{inv} }\le 
\| t^2(1-P)\g_2/2 \|_{b_K}
\le \|(1-P) (|\f|\f) \|_{b_K} + C_0\, t^ 3
\]
with \eqref{cceqeffstabproof1} and a triangle inequality in the last step.
Under the condition
\[
0<t\le \varepsilon_0:=\min\{ 1/2,  (r^*)^2/( 16 C_0 C_\textrm{inv})\},
\] 
we %infer $\frac{ t^2 (r^*)^2}{8 C_\textrm{inv} }- C_0\, t^ 3
%\ge t^2  (r^*)^2/(16 C_0 C_\textrm{inv} )$ and 
conclude the proof of the assertion
with $C_2= (r^*)^2/(16 C_0 C_\textrm{inv} )$ in the third step.

Step four concludes the proof for $t\le \varepsilon_0$.  The previous two steps provide 
\[
C_1^{-1}C_2 \| (1-P)(|\f|\f) \|_{0,K}\le  t^2 \, C_2 \le  \|(1-P) (|\f|\f) \|_{b_K} .
\]
Given any $\g\in \mathbb{P}_1(K)^d$ and $\f$ as before, a triangle and the last inequality
lead to
\begin{align*}
 \| |\f|\f +\g \|_{0,K}&\le  \| (1-P)(|\f|\f) \|_{0,K}+ \| \g+ P(|\f|\f) \|_{0,K}\\
& \le C_1C_2 ^{-1} \|(1-P) (|\f|\f) \|_{b_K} +C_\textrm{eq,2} \| \g+ P(|\f|\f) \|_{b_K}
%&\le \sqrt{ C_1^2C_2 ^{-2}+C_\textrm{eq,2}^ 2}  \| |\f|\f+\g \|_{b_K}
\end{align*}
with the aforementioned classical inverse estimate in the last step. A Cauchy inequality
in $\mathbb{R}^d$ and a Pythagoras identity
in the $b_K$ weighted Lebesgue norms based on the orthogonality 
$((1-P)(|\f|\f),\g+ P(|\f|\f))_{b_K}=0$ leads to
\[
 \| |\f|\f +\g \|_{0,K}\le  \sqrt{ C_1^2C_2 ^{-2}+C_\textrm{eq,2}^ 2}  \| |\f|\f+\g \|_{b_K}.
\]
This concludes the proof of the assertion with 
$C_\textrm{eq}^ 2 = C_1^2C_2 ^{-2}+C_\textrm{eq,2}^ 2 $.
{\hfill$\square$}
   
\medskip

\noindent{\it Conclusion of the proof of the lemma}. We depart from the previous definition
 of $\varepsilon_0$ as a universal constant (it exclusively depends on $K=K_\textrm{ref}$,
 whence solely on $d$) and see that the assertion follows with a constant 
 $C_\textrm{eq}=\sqrt{ C_1^2C_2 ^{-2}+C_\textrm{eq,2}^ 2}$ (that exclusively 
 depends on $d$) in case $\f\in \mathbb{P}_1(K)^d$ satisfies 
 $ \sigma_1(\bD\f)<\varepsilon_0 |\f(\textrm{\rm mid}(K))|$.
 The other regime is covered by the choice $\varepsilon=\varepsilon_0$ and leads to
 the assertion with a constant $C_\textrm{eq}(\varepsilon_0)$ in case 
 $\f\in \mathbb{P}_1(K)^d$ satisfies 
 $\varepsilon_0 |\f(\textrm{\rm mid}(K))|\le \sigma_1(\bD\f)$.
 This concludes the proof.
{\hfill$\square$}

\end{document}